\documentclass[11pt,a4paper]{amsart}

\usepackage[square, numbers, comma]{natbib}
\usepackage{amssymb}
\usepackage{amsmath}
\usepackage{amsthm}
\usepackage{amsfonts}
\usepackage{dsfont}
\usepackage{a4wide}
\usepackage{pdfsync}
\usepackage[colorlinks=true,linkcolor=blue,citecolor=blue,urlcolor=blue,breaklinks=true]{hyperref}
\usepackage{enumerate}
\usepackage{graphicx}
\usepackage{color}
\usepackage{cancel}
\usepackage{multirow}
\usepackage{tikz}
\usepackage{mathrsfs}
\usepackage{array}
\usepackage{esint}
\usepackage{enumitem}
\usepackage[justification=centering]{caption}
\usepackage{float}

\newtheorem{theorem}{Theorem}[section]
\newtheorem{proposition}[theorem]{Proposition}
\newtheorem{corollary}[theorem]{Corollary}
\newtheorem{lemma}[theorem]{Lemma}

\theoremstyle{remark}
\newtheorem{remark}[theorem]{Remark}

\theoremstyle{definition}
\newtheorem{definition}[theorem]{Definition}

\makeatletter
\renewcommand{\tocsection}[3]{\indentlabel{\@ifnotempty{#2}{\ignorespaces #1 #2.\,}}#3}

\def\l@subsection{\@tocline{2}{0pt}{2em}{}{}}
\def\l@subsection{\@tocline{2}{0pt}{2em}{}{}}
\makeatother
\newcommand{\dx}{\,{\rm d}x}
\newcommand{\dy}{\,{\rm d}y}
\newcommand{\dt}{\,{\rm d}t}
\newcommand{\dmu}{\,{\rm d}\mu}
\newcommand{\dr}{\,{\rm d}r}

\newcommand{\dnu}{\,{\rm d}\nu}
\newcommand{\dmuo}{\,{\rm d}\mu_{0}}

\newcommand{\Rn}{\mathbb{R}^d}
\newcommand{\RN}{\mathbb{R}^d}
\newcommand{\intr}{\int_{\Rn}}
\newcommand{\ints}{\int_{\{|y|\leq 1\}}}
\newcommand{\intb}{\int_{\{|y|\geq 1\}}}
\newcommand{\Lk}{\mathcal{L}}

\newcommand{\too}{t_{0}}
\newcommand{\xt}{(x,t)}

\newcommand{\xoo}{x_{0}}

\newcommand{\Ls}{(-\Delta)^{\alpha/2}}

\newcommand{\Lnu}{\mathcal{L}_{\nu}}

\newcommand{\intoii}{\int_{0}^{\infty}}
\newcommand{\Bd}{\mathcal{B}_{\textup{b}}(\Rn)}
\newcommand{\DA}{\mathcal{D}(\mathcal{A})}
\newcommand{\PP}{\mathcal{P}}
\newcommand{\lp}{L^{1}_{P_1}}
\newcommand{\luloc}{L^{1}_{\rm loc}((0,T);\lp)}
\newcommand{\li}{L^{\infty}((0,T);\lp)}

\usepackage{etoolbox}
\newcounter{taggedeq}
\pretocmd{\equation}{\stepcounter{taggedeq}}{}{}

\begin{document}

\title[Widder-type results for nonlocal heat equations]{On the nonlocal heat equation for certain L\'evy operators and the uniqueness of positive solutions}

\author[I.~Gonz\'alvez]{Irene Gonz\'alvez}
\address[I.~Gonz\'alvez]{Departamento de Matem\'aticas, Universidad Aut\'onoma de Madrid \& Instituto de Ciencias Matem\'aticas ICMAT (CSIC-UAM-UCM-UC3M). Campus de Cantoblanco, 28049 Madrid, Spain}
\email[]{irene.gonzalvez\@@{}uam.es}

\author[F.~Quir\'{o}s]{Fernando Quir\'{o}s}
\address[F.~Quir\'{o}s]{Departamento de Matem\'aticas, Universidad Aut\'onoma de Madrid \& Instituto de Ciencias Matem\'aticas ICMAT (CSIC-UAM-UCM-UC3M). Campus de Cantoblanco, 28049 Madrid, Spain}
\email[]{fernando.quiros\@@{}uam.es}
\urladdr{https://matematicas.uam.es/~fernando.quiros}

\author[F.~Soria]{Fernando Soria}
\address[F.~Soria]{Departamento de Matem\'aticas, Universidad Aut\'onoma de Madrid. Campus de Cantoblanco, 28049 Madrid, Spain}
\email[]{fernando.soria\@@{}uam.es}
\urladdr{https://matematicas.uam.es/~fernando.soria}

\author[Z.~Vondra\v{c}ek]{Zoran Vondra\v{c}ek}
\address[Z.~Vondra\v{c}ek]{\textsc{``Dr.~Franjo Tu}\scalebox{0.75}{\DJ}\textsc{man" Defense and Security University, Ilica 256b, 10000 Zagreb, Croatia}, and \textsc{Department of Mathematics, Faculty of Science, University of Zagreb, Bijeni\v{c}ka
		cesta 30, 10000 Zagreb, Croatia}}
\email[]{vondra\@@{}math.hr}
\urladdr{http://web.math.hr/~vondra/}

\begin{abstract}
We develop a Widder-type theory for nonlocal heat equations involving quite general L\'evy operators. Thus, we consider \emph{nonnegative} solutions and look for conditions on the operator that ensure: (i) uniqueness of nonnegative classical and very weak solutions with a given initial trace; (ii) the existence of an initial trace, belonging to certain admissibility class; and (iii) the existence of a solution, given by a representation formula, for any admissible initial trace. Such results are obtained first for purely nonlocal L\'evy operators defined through  positive symmetric L\'evy kernels comparable to radial functions with mixed polynomial growth, and then extended to more general operators, including anisotropic ones and operators that have both a local and a nonlocal part.
\end{abstract}

%%%%%%%%%%%%%%



\maketitle
\begin{center}\begin{minipage}{12cm}{\tableofcontents}\end{minipage}\end{center}

%%%%%%%%%%%%%%%%%%%%%%%%%%%%%%%%%%%%%%%%%%%%%%%%%%%%%%%%%%%%%%%%%%%%%%%%%%%%%%%%%%%%%%%%%
%%%%%%%%%%%%%%%%%%%%%%%%%%%%%%%%%%%%%%%%%%%%%%%%%%%%%%%%%%%%%%%%%%%%%%%%%%%%%%%%%%%%%%%%%
\section{Introduction}
\setcounter{equation}{0}

\noindent \emph{Goal.}
Our aim is to develop a Widder-type theory for nonlocal heat equations of the form
\begin{equation}\tag{NLHE}\label{a.nonlocalheatequation}	
    \partial_t u+\Lk u=0\quad\textrm{in }\RN\times (0,T),
\end{equation}
where $\mathcal{L}$ is a L\'evy-type operator as general as possible and $T\in (0,+\infty)$. Thus, we consider \emph{nonnegative} solutions to the problem and look for conditions on $\mathcal{L}$ that guarantee: (i) uniqueness of nonnegative classical and very weak solutions with a given initial trace; (ii) the existence of an initial trace, belonging to certain admissibility class; and (iii) the existence of a solution, given by a representation formula, for any admissible initial trace.

We start by considering the case of purely nonlocal L\'evy operators of the form
\begin{equation}\label{a.Lk}	
    \mathcal{L}_K u(x)=\int_{\mathbb{R}^d}\Big(u(x)-\frac{u(x+y)+u(x-y)}{2}\Big)K(y)\,\textup{d}y,
\end{equation}
where $K:\RN\to \mathbb{R}$ is a positive symmetric L\'evy kernel, that is, a measurable function such that
\begin{equation}\tag{A0}\label{A0}
    K(x)>0,\quad K(x)=K(-x)\quad\textrm{for a.e. }x\in\mathbb{R}^d,\qquad	\intr (1\wedge|y|^{2})K(y)\dy<\infty,
\end{equation}
with $a\wedge b:=\min\{a,b\}$. We will give additional growth conditions on $K$ allowing us to answer in the positive to all three questions: uniqueness, existence of a trace, and existence of a solution with a given (admissible) initial trace.

A careful analysis of the proofs will allow us to extend the results to some more general purely nonlocal operators of the form
\begin{equation}\label{a.Lnu}	
    \mathcal{L}_\nu u(x)=\int_{\mathbb{R}^d}\Big(u(x)-\frac{u(x+y)+u(x-y)}{2}\Big)\,\textup{d}\nu(y),
\end{equation}
where $\nu$ is a symmetric L\'evy measure, that is, a nonnegative Radon measure on $\mathbb{R}^d\setminus\{0\}$ satisfying
\begin{equation}\label{eq:Levy.condition}\tag{$\textup{A}_\nu$}
    \nu(B)=\nu(-B)\quad\text{for all }B\in\mathbb{B}(\mathbb{R}^d), \qquad \int_{\mathbb{R}^d}(1\wedge|y|^2)\,\textup{d}\nu(y)<\infty,
\end{equation}
with $\nu$ not necessarily absolutely continuous with respect to the Lebesgue measure. Here $\mathbb{B}(\mathbb{R}^d)$ denotes the $\sigma$-algebra of Borel sets in $\mathbb{R}^d$.  Our results cover, for instance, the family of anisotropic operators
$\Lk_{\nu}=\sum_{j=1}^\ell\Lk_{K_j}$, $\ell\in\{1,\dots,d\}$, where each $\mathcal{L}_{K_j}$ is of the form~\eqref{a.Lk} and acts only on the subspace $\mathbb{R}^{d_j}$, with $\mathbb{R}^d=\mathbb{R}^{d_1}\times\cdots\times\mathbb{R}^{d_\ell}$.

We will finally consider L\'evy operators having both a local and a nonlocal part, namely operators of the form $\mathcal{L}_{\Delta,K}:=-\Delta+\mathcal{L}_K$, with $\mathcal{L}_K$ as in~\eqref{a.Lk}. We will also obtain positive answers to the three questions if the kernel $K$ satisfies the same growth conditions as in the purely nonlocal case.

\noindent\emph{Widder's theory for the local heat equation.} The Cauchy problem for the local heat equation
\begin{equation}\label{eq:local.heat equation}	
    \partial_t u-\Delta u=0\quad\textrm{in }\mathbb{R}^d\times(0,T)
\end{equation}
is not well-posed in general: it has infinitely many classical solutions with zero initial datum, for instance, any multiple of Tychonov's counterexample~\cite{Tychonoff}. In order to get uniqueness, one has to restrict the class of functions within which one looks for solutions. One possibility, that may not seem natural from a physical point of view, is to consider only solutions that do not grow too rapidly as $|x|\to\infty$; see~\cite{Tychonoff}. A second common option is to require some integrability in space for almost every time, regarding solutions as curves in some $L^p$ space, as is done in the Theory of Semigroups. This approach matches well with the conservation of energy when $p=1$.

D.\,V.\,Widder considered a third approach: since the Third Principle of Thermodynamics asserts that temperatures remain always above zero Kelvin degrees, it is sensible to restrict the analysis to \emph{nonnegative} classical solutions.  In his celebrated paper~\cite{Widder} he proved two important results in this framework, both in the one-dimensional case $d=1$. The first one is a uniqueness result for nonnegative classical solutions with trivial initial data. Unfortunately, this does not give uniqueness for other initial data, since the difference of nonnegative solutions is not necessarily nonnegative. The second one is a representation theorem: if $u$ is a nonnegative classical solution of~\eqref{eq:local.heat equation}, there exists a nonnegative Borel measure $\mu$ such that $u(\cdot,t)= G_t*\mu$ for all $t>0$, where
\begin{equation}\label{p12}
    G_t(x)=G(x,t):=(4\pi t)^{-d/2}e^{-\frac{|x|^{2}}{4t}},\quad x\in\Rn,\ t>0,
\end{equation}
is the heat kernel (or fundamental solution) of the heat equation. Besides, $\mu$ cannot have too much mass near infinity,
\begin{equation}\label{eq:size.condition.measure.local.heat.equation}
    \int_{\mathbb{R}^d} e^{-|y|^2/(4t)}\,\textup{d}\mu(y)<\infty,\quad t\in(0,T).
\end{equation}
Some years later, Krzy\.{z}a\'{n}ski extended Widder's representation theorem to higher dimensions~\cite{Krzyzanski-1964}; see also~\cite{Aronson-1968}, where the uniqueness of the representing measure $\mu$ is proved. Note that~\eqref{eq:size.condition.measure.local.heat.equation} implies that $\mu$ is locally finite, hence a Radon measure, since $\mathbb{R}^d$ is a locally-compact Hausdorff space.

The converse of the representation theorem is also true: if $\mu$ is a Radon measure satisfying~\eqref{eq:size.condition.measure.local.heat.equation}, then $u(\cdot,t)= G_t*\mu$, $t>0$, is a solution to~\eqref{eq:local.heat equation};~see~\cite{Widder} for $d=1$ and~\cite{Aronson-1971} for higher dimensions.	

Widder's representation theorem did not settle the question of uniqueness for the Cauchy problem. For this we have to wait until Aronson's paper~\cite{Aronson-1971}, where he proves that if $\mu$ is the measure representing a nonnegative solution $u$ to the heat equation~\eqref{eq:local.heat equation}, then
\begin{equation}\label{eq:initial.data.Cauchy.problem}
    \lim_{t\to0^+}\intr\psi(x)u(x,t)\dx=\intr \psi\,\textup{d}\mu\quad\textrm{for all }\psi\in C_{\textup{c}}(\RN);
\end{equation}
that is, $u(\cdot,t)\to\mu$ as $t\to0^+$ as Radon measures in the weak* topology. This gives a positive answer to the uniqueness question for nonnegative solutions of the Cauchy problem having a Radon measure as initial trace, if we understand that $u$ has $\mu$ as initial datum when ~\eqref{eq:initial.data.Cauchy.problem} holds. In this paper the initial data is meant to be taken in this sense.

\begin{remark}
    Let $d=1$. If $\mu$ is the measure representing a nonnegative solution $u$ to~\eqref{eq:local.heat equation}, then
    \begin{equation*}
        \lim_{t\to0^+}\int_a^bu(x,t)\,\textup{d}x=\mu(a,b) + \frac{1}{2} \mu(a) + \frac{1}{2} \mu(b)\quad\text{for all }a, b\in\mathbb{R}\text{ with }a < b;
    \end{equation*}
    see~\cite{Widder-book}. Hence, if we understand that $u$ has $\mu$ as initial trace when
    \begin{equation*}
        \lim_{t\to0^+}\int_a^bu(x,t)\,\textup{d}x=\mu(a,b)\quad \text{for all }a < b\text{ such that } \mu\{a\} = \mu\{b\} = 0,
    \end{equation*}
    we would have uniqueness for nonnegative solutions of the Cauchy problem; see~\cite{Wilcox-1980}. A similar result in higher dimensions is given in~\cite{Aronson-1981}.
\end{remark}

%%%%%%%%%%%%%%%%%%%%
\noindent\emph{Precedents for nonlocal operators.} As already announced, our purpose is to extend Widder's theory to the nonlocal heat  equation~\eqref{a.nonlocalheatequation} for the widest possible class of nonlocal operators~$\mathcal{L}$. We have two direct nonlocal precedents that have been a source of inspiration for us,~\cite{Soria1} and~\cite{Matteo}, both of them dealing with the special case in which $\Lk$ is the fractional Laplacian $(-\Delta)^{\alpha/2}$, $\alpha\in (0,2)$, an operator of the form~\eqref{a.Lk} with $K_\alpha(y)=C_{d,\alpha} |y|^{-(d+\alpha)}$ for some normalization constant $C_{d,\alpha}$; {see also~\cite[Chapter 12]{Peral-book}.} In contrast with the local heat equation, the definition of very weak solution for this problem requires some integrability, at least $u\in L^1_{\textup{loc}}((0,T); L_{\alpha/2})$, where
\begin{equation*}
    L_{\alpha/2}(\mathbb{R}^d) := \left\{ u : \mathbb{R}^d \to \mathbb{R} \text{ measurable} : \int_{\mathbb{R}^d} \frac{|u(x)|}{1 + |x|^{d+\alpha}} \, \textup{d}x < \infty \right\}.
\end{equation*}
We will have an analogous situation when dealing with more general operators, with integrability spaces adapted to $\mathcal{L}$.

The first of these papers,~\cite{Soria1}, gives a representation formula in terms of the initial datum for nonnegative classical solutions to the equation that are continuous down to $t=0$. This implies uniqueness. The strategy is to prove uniqueness for very weak solutions within $L^{1}_{\textup{loc}}([0,T);L_{\alpha/2})$, and then show that nonnegative classical solutions are very weak solutions in this class. This last step is not as straightforward as in the local case, due to the nonlocal character of the fractional Laplacian. This approach can be adapted easily to cope with more general initial data. The second paper,~\cite{Matteo}, develops a complete Widder-type theory, though only for nonnegative very weak solutions, including the existence of an initial trace and considering the possibility of having (admissible) nonnegative Radon measures as initial data. The existence of a trace had already been proved in the previous paper~\cite{Bonforte-Vazquez-2014}, where the authors analyze a nonlinear variant of the equation.

Let us remark that the notions of very weak solution used in~\cite{Soria1} and~\cite{Matteo} do not coincide. We will make use of both of them (adapted to general operators), and of their equivalence in certain cases.

As for Widder-type results for more general linear nonlocal operators, we are only aware of the work~\cite{Quiros1}, where the authors prove uniqueness of very weak solutions to~\eqref{a.nonlocalheatequation} when $\Lk$ is an $\alpha$-stable diffusion operators given by an anisotropic symmetric L\'evy measure. Such operators have motivated us to develop the general theory in Section~\ref{section4}.

Operators of the form $\mathcal{L}_{\Delta,K}$ have attracted some attention in recent years; see, for instance,~\cite{Chen3}. However, we are not aware of any previous work dealing with Widder's theory for them.

%%%%%%%%%%%%%%%%%%%%%%%%%%%%%%%%%%%%%%%%%%
\noindent \emph{Concepts of solution and initial trace.} As we have just mentioned, there are in the literature different notions of solution to~\eqref{a.nonlocalheatequation}. Hence, before going further, let us make precise what we will mean by a very weak solution and a classical solution to this equation.

\begin{definition}\label{defWS}
	A \emph{very weak solution} of~\eqref{a.nonlocalheatequation} is a function $u\in L^1_{\textrm{loc}}(\Rn\times (0,T))$ such that:
	\begin{align}
        \label{a.2}
        &\int_{0}^{T}\intr|u(x,t)||\Lk \theta(x,t)|\dx\dt<\infty\quad\textrm{for all }\theta\in C^{\infty}_{\textup{c}}(\RN\times (0,T));\\
        \label{a.3}
        &\int_{0}^{T}\intr u(x,t)(\partial_t\theta(x,t)-\Lk\theta(x,t))\dx\dt=0\quad\textrm{for all }\theta\in C^{\infty}_{\textup{c}}(\RN\times (0,T)).
	\end{align}
\end{definition}
We will give, and use, an alternative definition of very weak solution in Section~\ref{section3};~see~\eqref{eq:space.alternative.definition}--\eqref{v3.10}.

\begin{definition}\label{a.def classical}
    A \emph{classical solution} of~\eqref{a.nonlocalheatequation} is a function $u\in C(\Rn\times(0,T))$ such that $\partial_t u$ and $\Lk u$ belong to $ C(\Rn\times(0,T))$ if $\mathcal{L}$ is purely nonlocal, and moreover $u\in C^{2,1}_{x,t}(\RN\times (0,T))$ if $\mathcal{L}$ has a local part, and
	\begin{equation}\label{eq:classical.equation}
        (\partial_t u+\mathcal{L}u)(x,t)=0\quad \textrm{for all }\xt\in\Rn\times (0,T).
	\end{equation}
\end{definition}

Since we are dealing with the Cauchy problem, it is also convenient to make precise what we mean by the \lq\lq initial datum'' of the problem.
Let us recall that the space of Radon measures on $\mathbb{R}^d$ is the dual space of $C_{\textup{c}}(\mathbb{R}^d)$. We will ask the solution to converge to the initial datum as $t\to0^+$ as Radon measures in the weak* topology (also known as vague topology).

\begin{definition}\label{deftrace}
	Let $\mu_{0}$ a Radon measure. We say that $u\in L^1_{\textrm{loc}}(\Rn\times (0,T))$ has (initial) trace or initial datum $\mu_{0}$ if
	\begin{equation}\label{a.1}
        \operatornamewithlimits{ess\,lim}_{t\to0^+}\intr\psi(x)u(x,t)\dx=\intr \psi\dmuo\quad\textrm{for all }\psi\in C_{\textup{c}}(\RN).
	\end{equation}
\end{definition}

%%%%%%%%%%%%%%%%%%%%%%%%%%%%%%%%%%%%%%%%%%%%%%%%%%%%%%%%%%%
\noindent \emph{Mixed polynomial growth.} Our main goal is to show that, under certain conditions, very weak and classical solutions to~\eqref{a.nonlocalheatequation} with initial datum $\mu_0$ are given by the representation formula
\begin{equation}\label{representationformula}\tag{RF}	
    u\xt=\intr P_{t}(x-y)\dmu_{0}(y)\quad\textrm{in }\RN\times (0,T).
\end{equation}
Here, $P_t(x)$ is the heat kernel of \eqref{a.nonlocalheatequation}, given by $P_{t}=\mathcal{F}(e^{-tm})$, $t>0$, where $\mathcal{F}$ denotes Fourier transform, and~$m$ is the Fourier symbol of $\Lk$. In order to accomplish this task when $\mathcal{L}=(-\Delta)^{\alpha/2}$, the papers~\cite{Soria1} and~\cite{Matteo} use in an essential way that the heat kernel associated to that operator satisfies
\begin{equation*}
	P_{t}(x)\asymp\frac{t}{\big(t^{1/s}+|x|^{2}\big)^{(d+2s)/2}}\quad\textrm{for }\xt\in\Rn\times (0,\infty).
\end{equation*}
(Here and below we denote $f\asymp g$ if the quotient $f/g$ remains bounded between two positive constants.) In particular, $P_t$ is comparable to the L\'evy kernel $K_\alpha$ of the fractional Laplacian away from the origin for every $t>0$.  This comparability between the heat and the L\'evy kernel will also be essential for our arguments in Section~\ref{section3}, when dealing with more general operators $\mathcal{L}_K$. Fortunately, there is a wide class of L\'evy kernels $K$ for which it holds: those comparable to radial functions with~\emph{mixed polynomial growth}; see~\cite{Panki}. Thus, besides~\eqref{A0}, we will also assume that
\begin{equation}\tag{A1}\label{A1}
    \frac{C^{-1}}{|x|^d\tilde{\varphi}(|x|)}\leq K(x)\leq \frac{C}{|x|^d\tilde{\varphi}(|x|)}\quad\text{for all }x\not =0
\end{equation}
for some constant $C>1$ and some nondecreasing function $\tilde{\varphi}:(0,+\infty)\to(0,+\infty)$, known as the \emph{scale function}, satisfying the global (lower and upper) \emph{weak scaling conditions}
\begin{equation}\tag{A2}\label{A2}
    \lambda_1\Big(\frac{R}{r}\Big)^{\beta_{1}}\leq\frac{\tilde{\varphi}(R)}{\tilde{\varphi}(r)}\leq \lambda_2\Big(\frac{R}{r}\Big)^{\beta_{2}}\quad\text{for all }0<r\leq R<\infty
\end{equation}
for constants $\lambda_1,\,\lambda_2>0$ and exponents $\beta_{2}\ge\beta_{1}>0$.

\begin{remark}
    Since the L\'evy kernel $K$ is assumed to satisfy~\eqref{A0}, then necessarily $\beta_2\le 2$.  Note, however, that we are admitting the borderline case $\beta_2=2$, which is not common in the literature.
\end{remark}

For some results we will also require the scale function $\tilde\varphi$ to satisfy
\begin{equation}\tag{A3}\label{A3}
    \tilde\varphi\in C^{2}((1,\infty)),\qquad |\tilde{\varphi}'/\tilde{\varphi}|\leq c,\quad|\tilde{\varphi}''/\tilde{\varphi}|\leq c\quad\textrm{in } [2,\infty)\quad\text{for some constant }c>0.
\end{equation}

Here is a table of examples satisfying assumptions \eqref{A0}--\eqref{A3}, with its lower and upper exponents.

\medskip

\begin{center}
\begin{tabular}{|c l|c |c|}
	\hline
	\multicolumn{2}{|c|}{ Examples of operators $\Lk$ or L\'evy kernels $K$}& $\beta_{1}$&$\beta_{2}$\\ \hline
	%&&&\\
	\multirow{2}{*}{$\displaystyle\Lk=(-\Delta)^{\alpha/2}$,}&\multirow{2}{*}{$0<\alpha<2$}& \multirow{2}{*}{$\alpha$}&\multirow{2}{*}{$\alpha$}\\
	\multirow{2}{*}{}&\multirow{2}{*}{}&\multirow{2}{*}{}&\multirow{2}{*}{}\\ \hline
\multirow{2}{*}{$ \displaystyle\Lk=(-\Delta)^{\alpha_{1}/2}+(-\Delta)^{\alpha_{2}/2}$},&\multirow{2}{*}{$0<\alpha_{1}<\alpha_{2}<1$}&\multirow{2}{*}{$\alpha_{1}$}&\multirow{2}{*}{$\alpha_{2}$}\\
	\multirow{2}{*}{}&\multirow{2}{*}{}&\multirow{2}{*}{}&\multirow{2}{*}{}\\ \hline
\multirow{3}{*}{$ \displaystyle K(x)=\frac{1}{|x|^{d+\alpha}(2+\cos\langle\theta,x\rangle)}$,}&\multirow{3}{2cm}{$\theta\in \mathbb{S}^{d-1}$, $0<\alpha<2$}& \multirow{3}{*}{$\alpha$}&\multirow{3}{*}{$\alpha$}\\
\multirow{2}{*}{}&\multirow{2}{*}{}&\multirow{2}{*}{}&\multirow{2}{*}{}\\ \multirow{2}{*}{}&\multirow{2}{*}{}&\multirow{2}{*}{}&\multirow{2}{*}{}\\  \hline
\multirow{3}{*}{$\displaystyle K(x)=\frac{\log^{\varepsilon}(1+|x|)}{|x|^{d+2}}$,}& \multirow{3}{*}{ $0<\varepsilon<2$}&\multirow{3}{*}{$2-\varepsilon$}&\multirow{3}{*}{ $2$}\\
\multirow{2}{*}{}&\multirow{2}{*}{}&\multirow{2}{*}{}&\multirow{2}{*}{}\\
	\multirow{2}{*}{}&\multirow{2}{*}{}&\multirow{2}{*}{}&\multirow{2}{*}{}\\ \hline	
	\multirow{4}{*}{$\displaystyle K(x)=\frac{1}{|x|^{d+2}(1+ \log^{\alpha}_{+}(\frac{1}{|x|}))}$,} &\multirow{4}{*}{$1<\alpha<2$}&\multirow{4}{2 cm}{$\;\;\;\;\;2-\varepsilon\;\;\;\;\;$ $\;\;\;\;$for all  $\;\;\varepsilon\in(0,2)$}&\multirow{4}{*}{$2$}\\
	\multirow{2}{*}{}&\multirow{2}{*}{}&\multirow{2}{*}{}&\multirow{2}{*}{}\\\multirow{2}{*}{}&\multirow{2}{*}{}&\multirow{2}{*}{}&\multirow{2}{*}{}\\ \multirow{2}{*}{}&\multirow{2}{*}{}&\multirow{2}{*}{}&\multirow{2}{*}{}\\  \hline
\end{tabular}
\end{center}

\medskip

\noindent \emph{Main results and organization of the paper.} In Section~\ref{section2}, we first review some known background  about the relationship between L\'evy processes and heat kernels. We then obtain  some fundamental  properties of the  heat kernel when the L\'evy operator is of the form~$\mathcal{L}_K$, with $K$ satisfying assumptions~\eqref{A0}--\eqref{A2}. These properties will be useful in proving our main results in~Section~\ref{section3}. Finally, by means of Bernstein's theorem for completely monotone functions, we establish additional properties for the heat kernel when $\mathcal{L}=(-\Delta)^{\alpha/2}$, $\alpha\in(0,2)$, that will be used to prove the existence of classical solutions in this case. For a complete description of  Bernstein's theorem see the monograph~\cite{Zoran}.

The core of the paper is Section~\ref{section3}, where we analyze operators of the form~\eqref{a.Lk}, assuming that $K$ satisfies~\eqref{A0}--\eqref{A2} (plus~\eqref{A3} for some results). Let $P_t$ be the associated heat kernel. Under these assumptions, $P_{t_1}\asymp P_{t_2}$ for every $t_1,t_2>0$, an important fact in our proofs. This is a significant difference with respect to the local heat equation, since $G_{t_1}$ is never comparable to $G_{t_2}$ if $t_1\neq t_2$.

In the first part of the section we focus on very weak solutions. It turns out that the integrability condition~\eqref{a.2} can be substituted by an integrability condition involving the heat kernel; see Proposition~\ref{prop1}. Thus, a natural space for very weak solutions is $\luloc$, where
\begin{equation*}
    L^{1}_{P_1}:=\big\{v\in L^{1}_{\textrm{loc}}(\Rn):\, \intr |v(x)|P_{1}(x)\dx<\infty\big\}.
\end{equation*}
There is nothing special about $P_1$, since all the functions $P_t$, $t>0$, are comparable. We could have picked any other reference time instead of $t=1$.

We start by proving uniqueness for very weak solutions to the Cauchy problem in the sense of the alternative definition~\eqref{eq:space.alternative.definition}--\eqref{v3.10}; this is Theorem~\ref{UniquenessSoria}. But very weak solutions in the sense of Definition~\ref{defWS} having $\mu_0$ as initial trace satisfy~\eqref{eq:space.alternative.definition}--\eqref{v3.10} if they belong to $\li$, see~Proposition~\ref{equivDefvwProp}, and uniqueness for very weak solutions in that class follows; see Corollary~\ref{UniquenessAcotadas}.

We next prove that nonnegative very weak solutions belong to $\li$. As a consequence we obtain a uniqueness result for them; see Corollary~\ref{UniquenessNonnegative}.

The third step, Theorem~\ref{theoremtraces}, is to show that every nonnegative very weak solution has an initial trace which is a Radon measure satisfying the integral growth condition
\begin{equation}\label{eq:growth.initial.trace}\tag{GC}
    \intr P_{1}\dmuo < \infty.
\end{equation}

The final step is to prove that if $\mu_0$ is a (signed) Radon measure satisfying
\begin{equation}\label{q4}\tag{GCS}	
    \intr P_{1}\,{\rm d}|\mu_{0}|<\infty,
\end{equation}
then, the function
\begin{equation}\label{q5}\tag{$\textup{S}_{\mu_0}$}
        U(x,t):=\intr P_{t}(x-y)\dmuo(y)\quad\textrm{for all }\xt\in\Rn\times(0,\infty)
	\end{equation}
is a very weak solution to~\eqref{a.nonlocalheatequation} in $\li$ \emph{for every} $T>0$, having $\mu_0$ as initial trace; see Theorem~\ref{prop2}. Here comes a second important difference with the local heat equation (intimately related to the first one): the integrability of the initial trace against the heat kernel at any time gives existence \emph{for all times}. In contrast, in the local framework, if the initial datum satisfies~\eqref{eq:size.condition.measure.local.heat.equation}, we can only guarantee existence for $t\in (0,T)$.

\begin{remark}
    Our existence result admits Radon measures with sign changes as initial data, though this is not necessary for Widder's theory.
\end{remark}

The four steps that we have just described imply in particular that a very weak solution to~\eqref{a.nonlocalheatequation} in $\li$ with initial trace $\mu_0$ is given by the representation formula~\eqref{representationformula}. Since nonnegative very weak solutions belong to $\li$, we have completed Widder's theory for this kind of solutions; see Corollary~\ref{CorRepresentationFormulaVW}.

In the second part of Section~\ref{section3} we consider classical solutions to~\eqref{a.nonlocalheatequation}, still for operators of the form~\eqref{a.Lk}. We give two conditions that together are sufficient for a classical solution to be a very weak solution; see Proposition~\ref{cprop1}. The first one, that concerns the regularity of the function, is
\begin{equation}\tag{H}\label{energyH}	
    \ints|\Lambda u(x,y,t)|K(y)\dy\in L^{1}_{\textup{loc}}(\Rn\times (0,T)),
\end{equation}
where $\Lambda u(x,y,t):=u(x,t)-\frac12\big(u(x+y,t)+u(x-y,t)\big)$. This would correspond to the regularity requirement $u\in C^{2,1}_{x,t}$ in Widder's theory for the classical heat equation, and is used to ensure that $\Lk u$ is integrable against any test function. The second one, $u\in\luloc$, is required by the definition of very weak solution for a nonlocal operator. Nonnegative classical solutions always satisfy the second condition;~see Corollary~\ref{a.cor2}.

Thanks to Proposition~\ref{cprop1}, we can take profit of the results for very weak solutions to prove results for classical solutions, always under the assumption~\eqref{energyH}: uniqueness and a representation formula in the class $\li$, Theorem~\ref{CorRepresentativeFormulaClassicalAcotadas}; and uniqueness, the existence of an initial trace and a representation formula for nonnegative solutions, Theorem~\ref{thm:main.K}.

What about the existence of classical solutions with a given Radon measure $\mu_{0}$ as initial trace? If $\mu_0$ satisfies the integral growth condition~\eqref{q4}, we have a natural candidate: the function $U$ given by~\eqref{q5}. However, proving that this very weak solution is a classical solution is not a trivial task, and requires some assumptions, that we write in terms of the heat kernel associated to the operator;~see Theorem~\ref{Uclassical}. These assumptions are satisfied, for instance, when the operator is the fractional Laplacian.

An inspection of the proofs in Section~\ref{section3} shows that they do not use assumptions~\eqref{A0}--\eqref{A3} directly, but their implications for the heat kernel, that could be satisfied by more general operators.  In Section~\ref{section4} we rewrite the theorems in Section~\ref{section3} in terms of such conditions for the heat kernel, so that they may encompass a wider class of purely nonlocal operators $\mathcal{L}_\nu$ of the form~\eqref{a.Lnu} with $\nu$ satisfying~\eqref{eq:Levy.condition}. To show the usefulness of these results, we check that they may be applied to the already mentioned family of anisotropic operators $\Lk_{\nu}=\sum_{j=1}^\ell\Lk_{K_j}$, $\ell\in\{1,\dots,d\}$, where each $\mathcal{L}_{K_j}$ is of the form~\eqref{a.Lk} and acts only on the subspace $\mathbb{R}^{d_j}$, with $\mathbb{R}^d=\mathbb{R}^{d_1}\times\cdots\times\mathbb{R}^{d_\ell}$. The key here, besides the special form of the operator, is that the heat kernel associated to $\mathcal{L}_\nu$ is the product of the heat kernels associated to each of the operators $\Lk_{K_j}$.

Finally, in Section~\ref{section5}, we consider operators having both a local and a nonlocal part, namely $\mathcal{L}_{\Delta,K}:=-\Delta+\mathcal{L}_K$, with $\mathcal{L}_K$ as in~\eqref{a.Lk}, and $K$ satisfying~\eqref{A0}--\eqref{A2} (and~\eqref{A3} in some cases). The results, and the proofs, go in parallel to those in Section~\ref{section3}. The clue to deal with this kind of operators is that the heat kernel is the convolution of the heat kernels of $-\Delta$ and $\mathcal{L}_K$, which helps in several ways, and even allows to remove some of the restrictions that we had in Section~\ref{section3}.

\noindent\emph{Some notations.} Throughout the paper we will often use the following notations:
\begin{itemize}
    \item Fourier transform: $\displaystyle\mathcal{F}(f)(\xi)=\frac{1}{(2\pi)^{d}}\intr f(x)e^{-i\langle x,\xi\rangle}\,{\rm d}\xi$.
    \item Inverse Fourier transform: $\displaystyle\mathcal{F}^{-1}(f)(x)=\intr f(\xi)e^{i\langle x,\xi\rangle}\,{\rm d}\xi$.
    \item Second order incremental quotients:
    \begin{equation}
    \label{SecondOrderIncrementalQuotiens}\Lambda v(x,y):=v(x)-\frac{1}{2}(v (x+y)+v(x-y)).
    \end{equation}
    \item Spatial cutoff functions: $\{\psi_R\}_{R\ge R_0}$, $R_0>0$, where
        \begin{equation}\label{cutoffspace}
            \psi_{R}(x)=\psi(x/R),\quad x\in\Rn,\quad\text{for some }\psi\in C^{\infty}_{\textrm{c}}(\Rn)\text{ such that } \mathcal{X}_{B_{1}(0)}\leq\psi\leq\mathcal{X}_{B_{2}(0)}.
        \end{equation}
    \item Set of Lebesgue points of $h\in L^{1}_{\textrm{loc}}((0,T))$ with respect to a summability kernel $\{\eta_{k}\}_{k\in\mathbb{N}}$:
        $L(h;\{\eta_k\}_{k\in\mathbb{N}}):=\{t\in (0,T):\,|h(t)|<\infty, \, \lim\limits_{k\to\infty}(\eta_{k}*h)(t) =h(t)\}$.
\end{itemize}

%%%%%%%%%%%%%%%%%%%%%%%%%%%%%%%%%%%%%%%%%%%%%%%%%%%%%%%%%%%%%%%%%%%%%%%%%%%%%%%%%%%%%%%%%
%%%%%%%%%%%%%%%%%%%%%%%%%%%%%%%%%%%%%%%%%%%%%%%%%%%%%%%%%%%%%%%%%%%%%%%%%%%%%%%%%%%%%%%%%
\section{The heat kernel}\label{section2}
\setcounter{equation}{0}

After reviewing some known background concerning the relation between L\'evy processes and heat kernels in subsection~\ref{2.1}, we obtain in subsection~\ref{2.2} several fundamental  properties of the  heat kernel associated to an operator $\mathcal{L}_K$ with $K$ satisfying~\eqref{A0}--\eqref{A2} that will be later used to prove our main results. When the operator is the fractional Laplacian we obtain further properties in subsection~\ref{2.3}, using the theory of completely monotone functions.

%%%%%%%%%%%%%%%%%%%%%%%%%%%%%%%%%%%%%%%%%%%%%%%%%%
\subsection{L\'evy processes and heat kernels}\label{2.1}

In this paragraph we explain the relation between L\'evy processes and heat kernels, that allows us to use well-known results for the former field to deal with the latter. Our main references for this preliminary material are~\cite{Apple,CRM}.

\noindent\emph{L\'evy process.}
A stochastic process on $\Rn$ is said to be a L\'evy process if it  starts at zero (i.e., $X_{0}=0$~a.s.), has stationary and independent increments, and is continuous in probability.

A L\'evy process $\big(X_{t}\big)_{t\geq 0}$  has a characteristic function, that determines its distribution, given by
\begin{equation}\label{n2.3}	
    \mathbb{E}[e^{i\xi X_{t}}]=e^{-t\psi(\xi)}\quad\textrm{for all }\xi\in\Rn,\,t\geq 0,
\end{equation}
where $\psi:\Rn\to\mathbb{C}$ is the so-called characteristic exponent, \cite[Corollary 2.5]{CRM}. By~L\'evy-Khintchine's theorem~\cite[Theorem 6.8]{CRM}, $\psi$ is a continuous function given by L\'evy-Khintchine's formula
\begin{equation}\label{n2.L-Kformula}	
    \psi(\xi)=\frac{1}{2}\langle \xi,Q\xi\rangle-\langle b,\xi\rangle+\intr \big(1-e^{i \langle \xi,y\rangle}+i\langle \xi,y\rangle\mathcal{X}_{B_{1}(0)}(y)\big)\dnu(y).
\end{equation}
Here $Q$ is a positive semi-definite  $d\times d$ matrix with real coefficients, $b\in\Rn$, and $\nu$ is a L\'evy measure, that is, a nonnegative Radon (that is, Borel regular and locally finite) measure  on $\Rn\setminus\{0\}$ such that
\begin{equation*}
	\intr (1\wedge |y|^{2})\dnu(y)<\infty.
\end{equation*}
A triplet $(Q,b,\nu)$ that satisfies the above properties is known as a L\'evy triplet.

Conversely, by~L\'evy-Khintchine's theorem~\cite[Theorem 7.5]{CRM}, a L\'evy triplet $(Q,b,\nu)$ has univocally  associated a L\'evy process whose characteristic exponent is given by \eqref{n2.L-Kformula}.

%%%%%%%%%%%%%%%%%%%%%
\noindent \emph{Markov transition density.} By \cite[Lemma 4.4]{CRM}, if $\big(X_{t}\big)_{t\geq 0} $  is a L\'evy process, then
\begin{equation}\label{n2.2}
    p_{t}(x,B)=\int_{B}p_t(x,\textup{d}y):=\mathbb{P}(x+X_{t}\in B),\quad t\geq0,\ x\in\Rn,\ B\in\mathbb{B}(\Rn),
\end{equation}
is a Markov transition density:  (i) $B\mapsto p_{t}(x,B)$ is a probability for every $t\geq 0$ and $x\in \Rn$; (ii) $(x,t)\mapsto p_{t}(x,B)$ is a Borel measurable function for every $B\in\mathbb{B}(\Rn)$;  and (iii) the Chapman-Kolmogorov equations hold,
\begin{equation*}
	p_{t+s}(x,B)=\intr p_{t}(y,B)p_{s}(x,\textup{d}y)\quad\textrm{for all }t,s\geq0,\ x\in\Rn,\ B\in\mathbb{B}(\Rn).
\end{equation*}
Besides, by definition, this Markov transition density is invariant by translations,
\begin{equation*}\label{v3.18}
    p_{t}(x,B)=p_{t}(0,B-x)\quad \textrm{for all }t\geq0,\quad x\in\Rn\quad \textrm{and }B\in\mathbb{B}(\Rn).
\end{equation*}
In slight abuse of notation, we write $p_{t}(B-x):=p_{t}(0,B-x)$.

\noindent\emph{Feller semigroup.} Associated with the above Markov transition density, we define a family of linear operators indexed by $t\geq0$ acting on the set of bounded Borel functions $\Bd$,
\begin{equation}\label{n2.Pt}	
    \mathcal{P}_{t}:\Bd\to\Bd,\quad \mathcal{P}_{t}f(x):=\mathbb{E}[f(x+X_{t})]=\intr f(y)p_{t}(x,\textup{d}y),\quad f\in\Bd,\ x\in\Rn.
\end{equation}
Since $p_t$ is invariant under translations, it follows that $\mathcal{P}_{t}$ is a convolution operator for all $t\geq 0$,
\begin{equation}\label{n2.5}	
    \mathcal{P}_{t}f(x)=\intr f(x+y)p_{t}(\textup{d}y)=f*P^{K}_{t}(x), \quad f\in\Bd,\ x\in\Rn,
\end{equation}
where $P^{K}_{t}(B):=p_{t}(-B)$ for all Borel sets $B$ and $t\geq 0$. Moreover, the family of operators $\big(\mathcal{P}_{t}\big)_{t\geq 0}$ is a Feller semigroup; that is, $\big(\mathcal{P}_{t}\big)_{t\geq 0}$:
\begin{itemize}
	\item is a semigroup: $\mathcal{P}_{t+s}=\mathcal{P}_{t}\circ\mathcal{P}_{s}$ for all $t,s\geq 0$ and $\mathcal{P}_{0}=\textrm{Id}$;
	\item is  sub-Markovian: if $0\leq f\leq 1$ then $0\leq \mathcal{P}_{t}f\leq 1$;
	\item is contractive: $\|\mathcal{P}_{t}f\|_{\infty}\leq \|f\|_{\infty}$;
	\item has the  Feller property: $\mathcal{P}_{t}(C_{0}(\Rn))\subset C_{0}(\Rn)$;
	\item is strongly continuous on $C_{0}(\Rn)$: $	\lim\limits_{t\to0^{+}}\|\mathcal{P}_{t}f-f\|_{\infty}=0$ for all $C_{0}(\Rn)$.
\end{itemize}
The strong continuity on $C_{0}(\Rn)$ implies in particular that
\begin{equation}\label{eq:strong.continuity.initial.trace}
    p_t\to \delta_0\quad\text{as }t\to0^+\quad\text{in the weak* topology in }(C_{0}(\Rn))'.
\end{equation}
Let us recall that $(C_{0}(\Rn))'$ is the set of finite Radon measures.

\noindent\emph{Infinitesimal generator.}
Since $\big(\mathcal{P}_{t}\big)_{t\geq 0}$ is a Feller semigroup, it has an associated infinitesimal generator $\mathcal{A}$, given by
\begin{gather*}
	\mathcal{A}f:=\lim\limits_{t\to0^{+}}\frac{\mathcal{P}_tf-f}{t}\quad\textrm{(uniform limit) for all $f$ in the domain of $\mathcal{A}$},\\
	\mathcal{D}(\mathcal{A}):=\Big\{f\in C_{0}(\Rn):\exists\,g\in C_{0}(\Rn)\textrm{ such that }
    \lim\limits_{t\to 0^+}\Big\|\frac{\mathcal{P}_{t}f-f}{t}-g\Big\|_{\infty}=0\Big\}.
\end{gather*}
It holds that $\PP_{t}(\DA)\subset\DA$ for all $t\geq 0$, and
\begin{equation}\label{n2.1}
    \frac{d}{dt}\PP_{t}f=\mathcal{A}(\PP_{t}f)\quad\textrm{for all }f\in \DA,
\end{equation}
so that  $u(t)=\PP_t  f$ solves the initial value problem $u'= \mathcal{A}u$, $u(0)=f$ in the Banach space $C_{0}(\Rn)$. In addition,
\begin{equation*}
	\PP_{t}f-f=\mathcal{A}\int_{0}^{t}\PP_{s}f\,{\rm d}s\quad\textrm{if } f\in C_{0}(\Rn),\qquad
	\PP_{t}f-f=\int_{0}^{t}\PP_{s}\mathcal{A}f\,{\rm d}s\quad\textrm{if } f\in \DA.
\end{equation*}
As a consequence of the above properties, $\DA$ is dense in $C_{0}(\Rn)$, $\mathcal{A}$ is a closed operator and $\mathcal{A}$ determines $(\PP_t)_{t\geq 0}$ uniquely; see~\cite[Lemma 5.4, Remark  5.5]{CRM}. Moreover, since $(\PP_{t})_{t\geq 0}$ comes from a L\'evy process, $\mathcal{S}(\Rn)\subset\DA$, where $\mathcal{S}(\Rn)$ is the Schwartz class in $\Rn$.

%%%%%%%%%%%%%%%%%%%%%%%%%%%%%%%
\noindent\emph{Fourier symbol.} By \cite[Theorem 3.3.3]{Apple}, the infinitesimal generator $\mathcal{A}$ of a  Feller semigroup $(\PP_t)_{t\geq 0}$ associated with a L\'evy process $(X_t)_{t\geq 0}$ can be expressed in terms of Fourier transform as
\begin{equation*}
	\mathcal{A}f(x)=-\intr \mathcal{F}(f)(\xi)\psi(\xi)e^{i\langle \xi,y\rangle}\,{\rm d}\xi,\quad f\in \mathcal{S}(\Rn),\,x\in\Rn,
\end{equation*}
where $\psi$ is the characteristic exponent of $(X_t)_{t\geq 0}$.

The operators in the Feller semigroup $(\PP_{t})_{t>0}$  can also be rewritten via Fourier transform. Indeed,   $\PP_{t}(\mathcal{S}(\Rn))\subset\mathcal{S}(\Rn)$ and, due to \eqref{n2.Pt} and  \eqref{n2.3},
\begin{align*}
    \PP_{t}f&=\PP_{t}\intr \mathcal{F}(f)(\xi)e^{i\langle \xi,\cdot\rangle}\,{\rm d}\xi=\intr \mathcal{F}(f)(\xi)\PP_{t}e^{i\langle \xi,\cdot\rangle}\,{\rm d}\xi=\intr \mathcal{F}(f)(\xi)e^{i\langle \xi,\cdot\rangle}e^{-t\psi(\xi)}\,{\rm d}\xi\\
    &=\langle \mathcal{F}^{-1}(f(\cdot-\xi)),e^{-t\psi(\xi)}\,\textup{d}\xi\rangle_{\mathcal{S}(\mathbb{R}^d)\times\mathcal{S}'(\mathbb{R}^d)}
    \quad\text{for all }f\in\mathcal{S}(\Rn),\ t\geq 0.
\end{align*}
As a consequence, using also~\eqref{n2.5}, we have that
\begin{equation}\label{b1}
    p_t(\textup{d}y)=\mathcal{F}^{-1}(e^{-t\psi(-y)}\textup{d}y)\quad\text{in }\mathcal{S}'(\mathbb{R}^d).
\end{equation}

%%%%%%%%%%%%%%%%%%%%%%%%%%%%%
\noindent\emph{Heat kernel.} A Feller semigroup $(\PP_{t})_{t\geq 0}$ is a strong Feller semigroup if  $\PP_{t}(B_{\textup{b}}(\Rn))\subset C_{\textup{b}}(\Rn)$. We have the following important result due to Hawkes, \cite[Lemma 4.9]{CRM}. The semigroup $(\PP_{t})_{t\geq 0}$ defined by~\eqref{n2.Pt} associated to a L\'evy process $(X_{t})_{t\geq 0}$ is a strong Feller semigroup if and only if its Markov transition density $p_{t}$, $t\geq 0$ (defined in~\eqref{n2.2}) is absolutely continuous with respect to the Lebesgue measure for all $t>0$; that is, if
\begin{equation*}
	p_{t}(\textup{d}y)=P_{t}(y)\dy\quad\text{for all }t>0
\end{equation*}
for some $P_{t}\in L^{1}(\Rn)$, known as the heat kernel of $(X_{t})_{t\geq 0}$. Note that $\|P_t\|_{L^{1}(\Rn)}=1$, since $p_t$ is a probability measure. In what follows we will use several times the alternative notation
\begin{equation*}
	P(y,t)=P_{t}(y),\quad(y,t)\in\Rn\times(0,\infty).
\end{equation*}

If the characteristic exponent of a L\'evy process $(X_{t})_{t\geq 0}$ satisfies $e^{-t\psi}\in L^{1}(\Rn)$ for all  $t>0$,  then, by  Riemann-Lebesgue's  Theorem, $\mathcal{F}^{-1}(e^{-t\psi})\in C_{0}(\Rn)$. Moreover, if $\psi$ is symmetric, by~\eqref{b1},  $(X_{t})_{t\geq 0}$ has a heat kernel given by
\begin{equation*}
	P_{t}=\mathcal{F}^{-1}(e^{-t\psi}),\quad t>0.
\end{equation*}
Notice that for each $t>0$ we have $\displaystyle|P_t(x)|\le \int_{\mathbb{R}^d}e^{-t\psi(\xi)}\,\textup{d}\xi=P_t(0)$.

Thanks to~\eqref{eq:strong.continuity.initial.trace} we know that the heat kernel has an initial trace:
\begin{equation*}
    P_t\to \delta_0\quad\text{as }t\to0^+\text{ as finite Radon measures in the weak* topology}.
\end{equation*}
Besides, by Fourier transform, $(P_{t})_{t\geq 0}$ satisfies the semigroup property
\begin{equation}\label{eq:semigroup.property}
    P_{t}* P_{s}=P_{s}*P_{t}=P_{t+s}\quad\text{for all }t,s\geq0.
\end{equation}

%%%%%%%%%%%%%%%%%%%%%%%%%%%%%%%%%%%%%%%%%
\subsection{Heat kernel properties for operators \texorpdfstring{$\mathcal{L}_K$}{L\_K}}\label{2.2}

Our next goal is to obtain further properties of the heat kernel when the L\'evy operator $\mathcal{L}_K$ has the form~\eqref{a.Lk} with $K$ satisfying~\eqref{A0}--\eqref{A2}.

\noindent\emph{Fourier symbol of the operator.} By Fourier transform,
\begin{equation*}\label{g2}	
    \mathcal{L}_K f(x)=\intr m(\xi)\mathcal{F}(f)(\xi)e^{i\langle x,\xi\rangle}\,{\rm d}\xi,\quad f\in C_{\textup{c}}^{\infty}(\Rn),\ x\in\Rn,
\end{equation*}
where  $m$ is the Fourier symbol of $\mathcal{L}_K$, given by
\begin{equation}\label{p6}
	m(\xi)=\intr(1-\cos\langle\xi,y\rangle)K(y)\dy,\quad\xi\in\RN.
\end{equation}
We know from the properties of the cosine function that there is a constant $C>0$ such that
\begin{equation}\label{p7}
	0\leq 1-\cos\langle\xi,y\rangle\leq C(1\wedge |\xi|^{2}|y|^{2})\leq C(1+|\xi|^{2})(1\wedge|y|^{2}).
\end{equation}
Then, since $K$ is a L\'evy kernel,
\begin{equation}\label{q9}
	0\leq m(\xi)\leq C(1+|\xi|^{2}),\quad \xi\in\Rn;
\end{equation}
see~\eqref{A0}. Moreover, because $K$ is symmetric and $\langle \xi,\cdot \rangle\mathcal{X}_{B_{1}(0)}$ is an odd function,
\begin{equation*}
    m(\xi)=\textrm{P.V.}\intr (1-e^{i\langle \xi,y\rangle})K(y)\dy= \intr (1-e^{i\langle \xi,y\rangle}+i\langle \xi,y\rangle\mathcal{X}_{B_{1}(0)}(y))K(y)\dy.
\end{equation*}
Since $|1-e^{i \langle \xi,y\rangle}+i\langle \xi,y\rangle\mathcal{X}_{B_{1}(0)}(y)|\leq C(1+|\xi|^{2})(1\wedge|y|^{2})$, the last integral is finite, thanks to the L\'evy condition~\eqref{A0}. Therefore, $m$ is the characteristic exponent of the L\'evy process $(Y_{t})_{t\geq 0}$ that has $(0,0,K(y)\dy)$ as L\'evy triplet; see L\'evy-Khintchine's formula~\eqref{n2.L-Kformula}. As a consequence,
\begin{equation}\label{g5}
    -\mathcal{L}_K\big|_{\mathcal{S}(\Rn)}=\mathcal{A}\big|_{\mathcal{S}(\Rn)},
\end{equation}
where $\mathcal{A}$ is the infinitesimal generator associated to the process.

We now prove the continuity of the Fourier symbol and estimate its size.
\begin{proposition} \label{Pprop2}
Let $m$ be the Fourier symbol of $\mathcal{L}_K$, with $K$ satisfying assumptions~\eqref{A0}--\eqref{A2}. Then, $m\in C(\mathbb{R}^d)$ and there exist constants $0<C_{1}\leq C_{2}$ such that
\begin{equation}\label{eq:size.m}
	C_{1}|\xi|^{\beta_{1}}\leq m(\xi)\leq C_{2}|\xi|^{\beta_{2}}\quad\textrm{for all }|\xi|\geq 1,
\end{equation}
where  $0<\beta_{1}\le\beta_{2}$ are the exponents in the global weak scaling conditions~\eqref{A2}.
\end{proposition}
\begin{proof}
Thanks to \eqref{p7}, we can apply the Dominated  Convergence Theorem, (DCT from now on), to obtain the desired continuity of $m$,
\begin{equation*}
	\lim\limits_{\xi_{1}\to \xi_{2}}(m(\xi_1)-m(\xi_2))= \intr  \lim\limits_{\xi_{1}\to \xi_{2}}
    (\cos\langle y,\xi_2\rangle-\cos\langle y,\xi_1\rangle)K(y)\dy=0.
\end{equation*}

As for the lower bound for $m$ when $|\xi|\ge1$, performing the change of variables $y\mapsto y/|\xi|$ in~\eqref{p6},
\begin{equation*}
	m(\xi)=\frac{1}{|\xi|^d}\intr (1-\cos\langle \xi/|\xi|,y\rangle)K(y/|\xi|)\dy.
\end{equation*}
Using hypothesis \eqref{A1} and the global lower weak scaling condition in~\eqref{A2},
\begin{equation*}
	K(y/|\xi|)\geq\frac{|\xi|^d}{C|y|^d\tilde{\varphi}(y/|\xi|)}=\frac{|\xi|^d}{C|y|^d\tilde{\varphi}(y)}
    \frac{\tilde{\varphi}(y)}{\tilde{\varphi}(y/|\xi|)}
	\geq \frac{\lambda}{C}\frac{|\xi|^{d+\beta_{1}}}{|y|^d\tilde{\varphi}(y)}\geq \frac{\lambda}{C^{2}}|\xi|^{d+\beta_{1}}K(y),
\end{equation*}
where $C$ is the same constant as in \eqref{A1} and $\lambda$ and $\beta_{1}$ come from \eqref{A2}. Therefore,
\begin{equation*}
m(\xi)\geq\frac{\lambda}{C^{2}}|\xi|^{\beta_1} \intr (1-\cos\langle\xi/|\xi|,y\rangle)K(y)\dy=\frac{\lambda}{C^{2}}|\xi|^{\beta_1}m(\xi/|\xi|).
\end{equation*}
Since $K$ is assumed to be  positive, then $m(\xi)\geq 0$, with equality  if and only if $\xi=0$; see~\eqref{p6}. Therefore, the continuity of $m$ implies that
\begin{equation*}
    m(\xi)\geq C_1|\xi|^{\beta_1}, \quad\text{where }  C_{1}=\frac{\lambda}{C^{2}}\min_{\theta\in\mathbb{S}^{n}}m(\theta)>0.
\end{equation*}

The proof of the upper bound is completely analogous.
\end{proof}

\noindent\emph{The heat kernel and its regularity.} Let $(Y_{t})_{t\geq 0}$ be the L\'evy process with L\'evy triplet $(0,0,K(y)\dy)$. We have just seen that its characteristic exponent is $m$. The lower bound for $m$ in~\eqref{eq:size.m} implies that  $e^{-tm}\in L^{1}(\Rn)$ for all $t>0$, whence, since $m$ is symmetric, $(Y_{t})_{t\geq 0}$ has a heat kernel given by
\begin{equation}\label{Aheatkernel}	
    P_{t}=\mathcal{F}^{-1}(e^{-tm}),\quad t>0;
\end{equation}
see~\eqref{b1}. The lower bound also allows us to prove that the heat kernel is smooth.

\begin{corollary}\label{pcor1}
     Let $K$ satisfying~\eqref{A0}--\eqref{A2} and let $P$ be the heat kernel associated to~$\mathcal{L}_K$. Then:
     \begin{itemize}
         \item[\textup{(i)}] $\partial_{t}^{k}P(\cdot ,t), D^{\alpha}P_t\in C_{0}(\Rn)$ for all $t>0$, $k\in \mathbb{N}\cup \{0\}$, and $\alpha=(\alpha_{1},\dots,\alpha_{d})\in (\mathbb{N}\cup\{0\})^{d}$;
         \item[\textup{(ii)}] $\mathcal{L}_K P_t\in C_0(\mathbb{R}^d)$ for all $t>0$;
         \item[\textup{(iii)}] $\partial_{t}^{k}P\in C(\Rn\times (0,\infty))$ for all $k\in \mathbb{N}\cup \{0\}$.
     \end{itemize}
\end{corollary}

\begin{proof}
(i) By the definition~\eqref{Aheatkernel} of the heat kernel via Fourier transform,
\begin{align*}\label{p10}
    &\partial_{t}^{k}P(\cdot,t)=\partial_{t}^{k}\mathcal{F}^{-1}(e^{-tm})=\mathcal{F}^{-1}\big(\partial_{t}^{k}e^{-tm}\big)
    =\mathcal{F}^{-1}\big((-1)^k m^ke^{-tm}\big), \\
    &D^{\alpha}P_{t}=D^{\alpha}\mathcal{F}^{-1}(e^{-tm})=\mathcal{F}^{-1}\big(pe^{-tm}\big)\quad\text{for some $\mathbb{C}$-polynomial $p$}.
\end{align*}
Hence, the result will follow from Riemann-Lebesgue's Lemma if we prove that
\begin{equation}\label{eq:smoothness.P}
	m^{k}e^{-tm},\, pe^{-tm}\in L^{1}(\Rn)\quad\textrm{for all }t>0
\end{equation}
for any $k\in\mathbb{N}\cup\{0\}$ and any $\mathbb{C}$-polynomial $p=p(\xi_1,\dots,\xi_d)$.

By the  lower bound for $m$ in~\eqref{eq:size.m}, for any $k\in \mathbb{N}\cup \{0\}$,
\begin{equation}\label{p14}	
    \begin{aligned}
    	\intr (1+|\xi|^{k})e^{-tm(\xi)}{\rm d}\xi&=\int_{\{|\xi|\leq 1\}} (1+|\xi|^k)e^{-tm(\xi)}{\rm d}\xi+\int_{\{|\xi|\geq 1\}}
        (1+|\xi|^{k})e^{-tm(\xi)}{\rm d}\xi\\
    	&\leq 2|B_1(0)|+2\int_{\{|\xi|\geq 1\}}|\xi|^{k}e^{-tC_{1}|\xi|^{\beta_{1}}}{\rm d}\xi<\infty\quad\text{for all }t>0,
    \end{aligned}
\end{equation}
from where~\eqref{eq:smoothness.P} follows, since $0\leq m^k(\xi)\leq C(1+|\xi|^{2k})$ for some constant $C>0$ for all $\xi\in\Rn$,
by~\eqref{q9}, and $|p(\xi_1,\dots,\xi_d)|\leq C(1+|\xi|^{k})$ for all $\xi\in\Rn$ for some constants $C>0$ and $k\in \mathbb{N}\cup\{0\}$.

\noindent(ii) We have already proved in (i) that $P_{t}\in C^{2}_{0}(\Rn)$. Hence, by Taylor's expansion,
\begin{equation}\label{z10}	
    \begin{aligned}
        |\Lambda P_{t}(x,y)|&\leq \frac12 \|D^{2}P_{t}\|_{\infty}|y|^{2}\mathcal{X}_{B_{1}(0)}(y) +2\|P_{t}\|_{\infty}\mathcal{X}_{B^{\textup{c}}_{1}(0)}(y)\\
        &\le (\frac12\|D^{2}P_t\|_{\infty}+2\|P_t\|_{\infty})(1\wedge|y|^{2})
    \end{aligned}
\end{equation}
where $\Lambda$ denotes the second order incremental quotients (see~\eqref{SecondOrderIncrementalQuotiens}).
Hence, since $K$ satisfies the L\'evy property~\eqref{A0}, we can apply the DCT to obtain
\begin{equation*}
    |\Lk P_{t}(x_{1})-\Lk P_{t}(x_{2})|\to 0\quad\textrm{as }|x_{1}-x_{2}|\to 0,\quad	
    |\Lk P_{t}(x)|\leq \intr |\Lambda P_{t}(x,y)|K(y)\dy\to 0\quad\textrm{as }|x|\to\infty.
\end{equation*}

\noindent(iii) It is enough to check that $\partial_{t}^{k}P(x,\cdot)\in C((0,\infty))$ for all $x\in\Rn$, since we already now that $\partial_{t}^{k}P(\cdot,t)\in C_{0}(\Rn)$ for all $t>0$.  Consider $x\in\Rn$ and $t_{1},t_{2}>0$. By \eqref{p14}, using the DCT, we have
\begin{equation*}
    |\partial_{t}^{k}P(x,t_{1})-\partial_{t}^{k}P(x,t_{2})|\leq\intr m^k(\xi)\big|e^{-t_{1}m(\xi)}-e^{-t_{2}m(\xi)}\big|\,\textup{d}\xi\to0
    \quad\textrm{as }|t_{1}-t_{2}|\to 0.\qedhere
\end{equation*}
\end{proof}

\noindent\emph{Size estimates.} If $K$ satisfies~\eqref{A0}--\eqref{A2},~\cite[Theorem 1.2]{Panki} implies that the tails of the heat kernel and the L\'evy kernel are comparable uniformly in bounded times away from zero: for each $\varepsilon>0$ and $T>\varepsilon$ there exists a constant $C_{\varepsilon,T}\geq1$ depending only on $\varepsilon$ and $T$ such that
\begin{equation}\label{Atotal}\tag{$\textup{C}_{PK}$}
	C_{\varepsilon,T}^{-1}K(x)\leq P_{t}(x)\leq C_{\varepsilon,T}K(x)\quad\textrm{for all }|x|\geq 1\textrm{ and }t\in[\varepsilon,T].
\end{equation}
Remember that $P_t\in C_0(\mathbb{R}^d)$ for all $t>0$. In particular, $P_t$ is bounded. On the other hand, $P_t>0$ for all $t>0$; see~\cite[Theorem 1.2]{Panki}. Combining these two facts with~\eqref{Atotal}, we have that there is some constant $C_{\varepsilon,T}>1$ such that
\begin{equation}\label{Atotal2}\tag{$\textup{C}_{P_t}$}
	 C_{\varepsilon,T}^{-1 }P_{t}(x)\leq P_{s}(x)\leq C_{\varepsilon,T}P_{t}(x)\quad\textrm{for all }x\in\Rn\textrm{ and } t,s\in[\varepsilon,T].
\end{equation}
Besides, since $P_{1}>0$ and continuous, for each $\tilde{\theta}\in C_{\textup{c}}(\RN\times[0,T])$ there is $C_{\tilde{\theta}}>0$ such that
\begin{equation}\label{a.4}	
    |\tilde{\theta}(x,t)|\leq C_{\tilde{\theta}}P_{1}(x)\quad\textrm{for all }(x,t)\in\RN\times[0, T],
\end{equation}
which together with the semigroup property~\eqref{eq:semigroup.property} and~\eqref{Atotal2} yields that there is $C_{\tilde{\theta}, T}>0$ such that
\begin{equation}\label{a.5}
    |\big(P_{t}*\tilde{\theta}(\cdot,t)\big)(x)|\leq C_{\tilde{\theta}} P_{t}*P_{1}(x)\leq C_{\tilde{\theta}, T}P_{1}(x)\quad\text{for all }\xt\in\RN\times[0,T].
\end{equation}

Thanks to the positivity and continuity of $P_1$, we can control the image of test functions under the action of $\mathcal{L}_K$.

\begin{lemma}\label{a.17}
     Let $\mathcal{L}=\mathcal{L}_K$ with $K$ satisfying assumptions~\eqref{A0}--\eqref{A2} and let $P$ be the corresponding heat kernel. Let $0\leq \delta<\tau<\infty$. For each $\tilde{\theta}\in C^{2}_{\textup{c}}(\RN\times[\delta,\tau])$ there is a constant $C_{\tilde{\theta}}>0$ such that
	\begin{equation*}        	
        |\mathcal{L}_K\tilde{\theta}(x,t)|\leq\intr |\Lambda\tilde{\theta}(x,y,t)|K(y)\dy\leq C_{\tilde{\theta}}P_{1}(x)\quad \text{for all }\xt\in\RN\times[\delta,\tau].
	\end{equation*}
\end{lemma}

\begin{proof}
The first inequality is immediate. As for the second, we split the integration domain in two sets, to take profit of the different properties of $K$ in different sets,
\begin{equation*}
	\intr |\Lambda\tilde{\theta}(x,y,t)|K(y)\dy=\underbrace{\ints|\Lambda\tilde{\theta}(x,y,t)|K(y)\dy}_{\textup{I}(x,t)}
    +\underbrace{\intb|\Lambda\tilde{\theta}(x,y,t)|K(y)\dy}_{\textup{II}(x,t)}.
\end{equation*}

By Taylor's expansion, for each $x,y\in\mathbb{R}^d$ and $t\in[0,\tau]$ there is some $\lambda$, $|\lambda|\le 1$, such that
\begin{equation*}
 |\Lambda\tilde{\theta}(x,y,t)|\le \frac{|y|^2}2\sum_{j,k=1}^{d}|\partial_j\partial_k\tilde\theta(x+\lambda y,t)|.
\end{equation*}
Let $R_{\tilde\theta}$ be such that $\operatorname{supp}(\tilde\theta(\cdot,t))\subset B_{R_{\tilde\theta}}(0)$ for all $t\in[0,\tau]$. If
$|\lambda|,|y|\le1$, then
\begin{equation*}
    \operatorname{supp}(\partial_j\partial_k\tilde\theta(\cdot+\lambda y,t))\subset B_{1+R_{\tilde\theta}}(0),\quad j,k\in\{1,\dots,d\},
\end{equation*}
whence, since $P_1$ is continuous and positive, there exists a constant $C_{\tilde\theta}$ such that
\begin{equation*}
    \sum_{j,k=1}^d|\partial_j\partial_k\tilde\theta (x+\lambda y,t)|\le C_{\tilde\theta} P_1(x).
\end{equation*}
The L\'evy condition \eqref{A0} implies then that $\displaystyle \textup{I}(x,t)\leq C'_{\tilde{\theta}}P_{1}(x)$ for some constant $C'_{\tilde{\theta}}>0$.

As for $\textup{II}$, since $K$ is symmetric,
\begin{equation*}
	\textup{II}(x,t)\leq|\tilde{\theta}(x,t)|\int_{\{|y|\geq1\}} K(y)\dy+\int_{\{|y|\geq1\}}
    |\tilde{\theta}(x-y,t)|K(y)\dy.
\end{equation*}
By \eqref{a.4},  there is a constant $C_{\tilde{\theta}}>0$ such that $|\tilde{\theta}(\cdot,t)|\leq C_{\tilde\theta}P_{1}$ for all $t\in[0,\tau]$. Then,
\begin{equation*}
	\textup{II}(x,t)\leq C_{\tilde{\theta}}P_{1}(x)\int_{\{|y|\geq1\}}
    K(y)\dy+C_{\tilde{\theta}}\int_{\{|y|\geq1\}}P_{1} (x-y)K(y)\dy.
\end{equation*}
The integral in the first term of the right-hand side is finite, thanks to the L\'evy condition~\eqref{A0}. To handle the second term we use that we can control  the tail of $K$ by the tail of $P_1$, see  \eqref{Atotal}. Thus,
\begin{equation*}
	\textup{II}(x,t)\leq C_{\tilde{\theta}}\Big(P_{1}(x)+ \int_{\{|y|\geq 1\}}P_{1} (x-y)P_{1}(y)\dy\Big)\leq C_{\tilde{\theta}}\Big(P_{1}(x)+\intr P_{1} (x-y)P_{1}(y)\dy\Big).
\end{equation*}
The integral on the rightmost term is equal to $P_{2}(x)$, due to the semigroup property~\eqref{eq:semigroup.property}. Moreover, thanks to \eqref{Atotal2}, there exists $C>0$ such that $P_{2}\leq CP_{1}$. We conclude that
\begin{equation*}
    \textup{II}(x,t)\leq C_{\tilde{\theta}}(P_{1}(x)+ P_{2}(x))\leq C_{\tilde{\theta}}P_{1}(x).\qedhere
\end{equation*}
\end{proof}

Our next result shows that we can compare $P_1$ with its translations.

\begin{lemma}\label{v3.19}
    Let $\mathcal{L}=\mathcal{L}_K$ with $K$ satisfying assumptions~\eqref{A0}--\eqref{A2} and let $P$ be the corresponding heat kernel. Given $R>0$, there exists a constant $C_R>0$ such that
	\begin{equation}\label{eq:comparison.translations}
			P_{1}(x-y)\leq C_R P_{1}(y)\quad\textrm{for all }y\in\Rn\text{ and }x\in B_R(0).
	\end{equation}
\end{lemma}

\begin{proof}
Let $x\in B_R(0)$ and $y\in A_x:=\{y:|y|\leq 2\}\cup\{y:|x-y|\leq |y|/2,\, |y|\geq 2\}$. Thus we have $A_x\subset\{y: |y|\leq 2(R+1)\}$. Since $P_1\in C(\mathbb{R}^d)$ and $0<P_1\leq P_1(0)$, then $ P_1(x-y)\leq  P_1(0)\le C_R P_1(y)$, where $C_R=P_1(0)/\tilde{C}_R$ with $\tilde{C}_R=\min\{P_1(y):|y|\leq 2(R+1)\}>0$.

On the other hand, if $y\in (A_{x})^{\textup{c}}=\{y:|x-y|\geq |y|/2\geq 1\}$, using the estimates for $P_1$ following from~\eqref{Atotal} and~\eqref{A1}, and the fact that $\tilde{\varphi}$ is a nondecreasing function that satisfies the upper weak scaling condition in~\eqref{A2}, we have, for constants not depending on $x$,
\begin{equation*}
    P_{1}(x-y)\leq \frac{C}{|x-y|^{d}\tilde{\varphi}(|x-y|)}\leq \frac{C}{|y|^{d}\tilde{\varphi}(|y|/2)}\leq \frac{C}{|y|^{d}\tilde{\varphi}(|y|)}    \leq CP_{1}(y).\qedhere
\end{equation*}
\end{proof}

\noindent {\emph{Rate of change.} Our next aim is to prove that $P_1$ is essentially constant on dyadic annuli, that is, on sets of the form $\{x: r\le|x|\le 2r\}$, with a rate of change that is  independent of $r$. To be more precise, we formulate the following.
\begin{definition}\label{defSlowly}
    We will say that a  function $f:\Rn\to(0,\infty)$ has a \emph{bounded rate of change} (or simply, that $f$ is \emph{changing slowly}) if for all $0<\rho_{1}\leq \rho_{2}$ there exists a constant $C_{\rho_{1},\rho_{2}} \geq 1$ such that
    \begin{equation*}\label{eq:slowly.varying}
    	C^{-1}_{\rho_{1},\rho_{2}}\leq \frac{f(x)}{f(y)}\leq C_{\rho_{1},\rho_{2}}\quad\textrm{if }\rho_{1}\leq \frac{|y|}{|x|}\leq \rho_{2}.
    \end{equation*}
\end{definition}
A function with this property is one that does not exhibit  rapid oscillations nor a rapid decay or growth. A typical example is that where $f$ is monotone decreasing with a certain polynomial decay. On the contrary, the Gaussian represents a function which does  not satisfy the condition. As we will see, the above property  will be essential throughout this work.}

The {slow change} of $P_1$ {near zero} stems from its continuity and positivity. Away from the origin it is inherited directly from $K$, which has this property thanks to assumptions \eqref{A1} and \eqref{A2}. In fact, we know that the tails of $P_1$ and $K$ are comparable. {We are now ready to prove the following.}

\begin{proposition}\label{Pslowly}
    Let $\mathcal{L}=\mathcal{L}_K$ with $K$ satisfying assumptions~\eqref{A0}--\eqref{A2} and let $P$ be the corresponding heat kernel. Then, $P_{1}$ is a {slowly changing} function.
\end{proposition}

\begin{proof}
Let $0<\rho_1\le \rho_2$. We consider two cases, depending on the size of $|x|$.

Let $|x|\leq 1\vee \rho_{1}^{-1}$, where $a\vee b:=\max\{a,b\}$. If $\rho_1\le |y|/|x|\leq \rho_{2}$, then  $|y|\leq \rho_{2}(1\vee \rho_{1}^{-1})$. Since $P_{1}$ is positive and continuous,
\begin{equation*}
	C^{-1}\leq \frac{P_{1}(x)}{P_{1}(y)}\leq C\quad\textrm{for all }|x|\leq 1\vee \rho_{1}^{-1}\textrm{ and }|y|\leq \rho_{2}(1\vee \rho_{1}^{-1}),
\end{equation*}
for some constant $C\ge 1$ depending only on $\rho_1$ and $\rho_2$, and we are done.

Suppose now that $|x|\geq 1\vee \rho_{1}^{-1}$.  Then $|x|,|y|\ge1$ whenever $\rho_1\le |y|/|x|\leq \rho_{2}$. By \eqref{Atotal},  $K$ and $P_1$ are comparable outside the ball of radius 1, whence,  by hypothesis \eqref{A1},
\begin{equation*}
	\frac{C^{-1}}{|x|^{d}\tilde{\varphi}(|x|)}\leq P_{1}(x)\leq \frac{C}{|x|^{d}\tilde{\varphi}(|x|)}\quad\text{for all }|x|\geq 1
\end{equation*}
for some $C>1$. As a consequence,
\begin{equation*}
	C^{-2}\rho_{1}^{d}\frac{\tilde{\varphi}(|y|)}{\tilde{\varphi}(|x|)}\leq \frac{P_{1}(x)}{P_{1}(y)}\leq C^{2}\rho_{2}^d\frac{\tilde{\varphi}(|y|)}{\tilde{\varphi}(|x|)}\quad\textrm{whenever }\rho_{1}\leq\frac{|y|}{|x|}\leq \rho_{2}\textrm{ and }|x|\geq1\vee \rho_{1}^{-1},
\end{equation*}
from where the result will follow if  there is a constant  $C\ge1$ depending only on $\rho_{1}$ and $\rho_{2}$ such that
\begin{equation}\label{eq:phitilde.sv}
	C^{-1}\leq \frac{\tilde{\varphi}(|y|)}{\tilde{\varphi}(|x|)}\leq C \quad\textrm{whenever }\rho_{1}\leq\frac{|y|}{|x|}\leq \rho_{2}.
\end{equation}

Assume that $\rho_1\le |y|/|x|\leq \rho_{2}$. To prove~\eqref{eq:phitilde.sv} we consider two cases, depending on the size of $|y|$ compared to $|x|$. If $|y|\geq |x|$, then $(1\vee\rho_{1})\leq|y|/|x|\leq \rho_{2}$. Hence,  by~\eqref{A2},
\begin{equation*}
	\lambda_1 (1\vee\rho_{1})^{\beta_{1} }\leq \frac{\tilde{\varphi}(|y|)}{\tilde{\varphi}(|x|)}\leq \lambda_2 \rho_{2}^{\beta_{2}}.
\end{equation*}
On the other hand, if $|y|\leq |x|$, then $ \rho_{1}|\leq|y|/|x|\leq (1\wedge\rho_{2})$. Therefore,  again by \eqref{A2},
\begin{equation*}
	\lambda_2 ^{-1}\rho_{1}^{\beta_{2} }\leq \frac{\tilde{\varphi}(|y|)}{\tilde{\varphi}(|x|)}\leq \lambda_1^{-1}(1\wedge\rho_{2})^{\beta_{1}}.\qedhere
\end{equation*}
\end{proof}

As a corollary, we get the following result that will show to be useful in next section.

\begin{lemma}
     Let $\mathcal{L}=\mathcal{L}_K$ with $K$ satisfying assumptions~\eqref{A0}--\eqref{A2} and let $P$ be the corresponding heat kernel. There is a constant $C>0$ such that
    \begin{equation}\label{eq:Pz.Px}
        P_1(z)\le C P_1(x)\quad\text{for all }x\in \mathbb{R}^d\text{ and }z\in B_1(x).
    \end{equation}
\end{lemma}

\begin{proof}
If $|x|\ge2$, then $|x|/ 2\le  |z|\le 3/2 |x|$, whence, since $P_1$ is {slowly changing}, $P_1(z)\le C P_1(x)$ for some constant $C>0$ independent of $x$ and $z$. Besides, since $P_1\in C(\mathbb{R}^d)$ and $0<P_1\le P_1(0)$, then
\begin{equation*}
    P_1(z)\le \frac{P_1(0)}{\min_{B_2(0)}P_1}\,P_1(x)\quad\text{for all }|x|\le 2.\qedhere
\end{equation*}
\end{proof}

The function $P_1$ is also \emph{almost decreasing}.

\begin{proposition}\label{prop:P1.almost.decreasing}
    Let $\mathcal{L}=\mathcal{L}_K$ with $K$ satisfying assumptions~\eqref{A0}--\eqref{A2} and let $P$ be the corresponding heat kernel. Then, $P_{1}$ is almost decreasing: there exists a constant $C>0$ such that
    \begin{equation*}
        P_1(z)\le CP_1(x)\quad\text{for all }x,z\in\mathbb{R}^d\text{ such that }|x|\le |z|.
    \end{equation*}
\end{proposition}

\begin{proof}
If $1\le|x|\le |z|$ the result follows easily, since in this region $P_1$ is comparable to $K$, see~\eqref{Atotal}, and $K$ is almost decreasing, see~\eqref{A1}. Besides, since $P_1\in C(\mathbb{R}^d)$ and $0<P_1\le P_1(0)$,
\begin{equation*}
        P_1(z)\le \frac{P_1(0)}{\min_{B_1(0)}P_1}\,P_1(x)\quad\text{if }|x|\le |z|\le1.
\end{equation*}
Finally, if $|x|\le 1\le |z|$, we consider $\xi$ such that $|\xi|=1$, and, using the two previous cases we conclude that
$P_1(z)\le C P_1(\xi)\le CP_1(x)$.
\end{proof}

\begin{remark}\label{rk:Pt.sv.and.ad}
    Since, by~\eqref{Atotal2}, $P_t$ is comparable to $P_1$ if $t>0$, and the latter is {slowly changing} and almost decreasing, so is the former.
\end{remark}

\noindent\emph{The heat kernel solves the equation.} By assumption~\eqref{A0}, $K$ is symmetric, whence the associated Markov transition density is also symmetric. Then,  by \eqref{n2.5},
\begin{equation}\label{n2.6}	
    \mathcal{P}_{t}f=P_{t}*f\quad \textrm{for all } f\in \mathcal{S}(\Rn)\textrm{ and } t>0.
\end{equation}
Since $\mathcal{S}(\Rn)\subset \DA$ and $\PP_{t} (\mathcal{S}(\Rn))\subset\mathcal{S}(\Rn)$, using~\eqref{n2.1},~\eqref{g5} and~\eqref{n2.6} we get that for every $f\in \mathcal{S}(\Rn)$ the function $u(t)=P_t*f$ is a (classical) solution of $\partial_{t}u+\mathcal{L}_K u=0$ with initial data $f$.

The operators $\mathcal{L}_K$ and  $-\mathcal{A}$ in principle only coincide in $\mathcal{S}(\Rn)$, see \eqref{g5}, and the heat kernel at time $t$, $P_t$, does not belong in general to this class (see for instance the case of the fractional Laplacian). However, we are able to prove that $P$ does not only solve the  heat equation associated to $-\mathcal{A}$, but also the one associated to $\mathcal{L}_K$.

\begin{proposition}\label{pprop2}
    Let $\mathcal{L}=\mathcal{L}_K$ with $K$ satisfying assumptions~\eqref{A0}--\eqref{A2}. The corresponding heat kernel $P$ is a classical solution of equation~\eqref{a.nonlocalheatequation} for every $T>0$.
\end{proposition}

\begin{proof}
By Corollary~\ref{pcor1}, we already know that $P,\partial_tP,\mathcal{L}P_t\in C(\Rn\times (0,\infty))$. Hence it is enough to check that $P_t$ satisfies~\eqref{eq:classical.equation} for all $T>0$.

Recall that $P_t\in C^2_0(\mathbb{R}^d)$; see~Corollary~\ref{pcor1}. Hence, there is a sequence $\{f_n\}\subset C^{\infty}_{\textup{c}}(\Rn)$ such that $f_n\to P_t$ in $W^{2,\infty}_{0}(\Rn)$ as $n\to\infty$.  Then, since $K$ satisfies the L\'evy condition~\eqref{A0},
\begin{align*}
    \sup_{x\in\Rn}&\big\{|\Lk f_n(x)-\Lk P_t(x)|\big\}\leq \sup_{x\in\Rn}\big\{\intr |\Lambda(f_n-P_t)(x,y)|K(y)\dy\big\}\\
    &\leq \frac12\|D^{2}(f_n-P_t)\|_{\infty}\ints |y|^2K(y)\dy+2\|f_n-P_t\|_{\infty}\intb K(y)\dy\to 0 \quad\textrm{as }n\to\infty.
\end{align*}
That is, $\Lk f_n\to \Lk P_t$ uniformly  as $n\to \infty$. But, since $f_n\in C^{\infty}_{\textup{c}}(\Rn)\subset\mathcal{S}(\mathbb{R}^d)$, then $\mathcal{L}f_n=-\mathcal{A}f_n$ for all $n\in\mathbb{N}$; see~\eqref{g5}. Thus, we have proved that $-\mathcal{A}f_n\to \Lk P_t$ uniformly  as $n\to \infty$ converges uniformly to $ \Lk P_{t}$ as $n\to\infty$. Since $\mathcal{A}$ is a closed operator (see \cite[Remark 5.5]{CRM}), we conclude that $P_{t}\in\DA$ and $-\mathcal{A}P_{t}=\Lk P_{t}$.

By~\eqref{n2.1}, since $P_{h}\in\DA$ for all $h>0$, then $\frac{d}{dt}\PP_{t}P_{h}=\mathcal{A}\PP_{t}P_{h}$ for all $h>0$. But, thanks to~\eqref{n2.6} and the semigroup property~\eqref{eq:semigroup.property}, we have that $\PP_{t}P_{h}=P_{t+h}$ for all $t\geq 0$ and $h>0$, whence we conclude that $P_t$ satisfies~\eqref{eq:classical.equation}.
\end{proof}

\begin{corollary}\label{PisaVeryWeaksolution.cor}
    Let $\mathcal{L}=\mathcal{L}_K$ with $K$ satisfying assumptions~\eqref{A0}--\eqref{A2}. The corresponding heat kernel $P$ is a very weak solution to equation~\eqref{a.nonlocalheatequation} for every $T>0$.
\end{corollary}

\begin{proof}
By Proposition~\ref{pcor1}, $P\in C(\Rn\times (0,\infty))$, whence $P\in L^1_{\textrm{loc}}(\Rn\times (0,T))$.

Let us check that $P$ satisfies~\eqref{a.2}. Given $\theta\in C^{\infty}_{\textup{c}}(\Rn\times(0,T))$, let $0<\varepsilon<\tau<T$ such that $\operatorname{supp}(\theta)\subset\Rn\times(\varepsilon,\tau)$. Using Lemma~\ref{a.17}, the symmetry of $P_t$, the semigroup property~\eqref{eq:semigroup.property}, and~\eqref{Atotal2},
\begin{equation}\label{z11}
    \begin{aligned}
        \int_0^T\intr P_t(x)|\Lk\theta\xt|\dx\dt&\leq C_\theta\int_\varepsilon^\tau\intr P_t(x)P_1(x)\dx\dt\\
        &\leq C_\theta\int_\varepsilon^\tau P_{1+t}(0)\dt\leq C_{\theta,\tau}\tau P_1(0)<\infty.
    \end{aligned}
\end{equation}

Let us prove now that $P$ satisfies~\eqref{a.3}. By~\eqref{z11}, using first Fubini-Tonelli's Theorem and then changing variables, we obtain that
\begin{equation}\label{z4}
    \begin{aligned}	
        \int_0^T\intr P(x,t)\Lk\theta\xt\dx\dt&=\int_\varepsilon^\tau\intr\intr P(x,t)\Lambda\theta(x,y,t)K(y)\dy\dx\dt\\ &=\int_\varepsilon^\tau\intr K(y)\intr  P(x,t)\Lambda\theta(x,y,t)\dx\dy\dt\\
        &=\int_\varepsilon^\tau\intr K(y)\intr\Lambda P(x,y,t)\theta(x,t)\dx\dy\dt.
    \end{aligned}
\end{equation}

On the other hand, since $0<P_t(x)\le P_t(0)$ for all $x\in\mathbb{R}^d$, $t>0$, and $P\in C(\mathbb{R}^d\times(0,\infty))$, see Corollary~\ref{pcor1} (iii), then $\|P_t\|_{\infty}\le \max_{t\in[\varepsilon,\tau]} P_t(0)=c_{\varepsilon,\tau}<\infty$ if $t\in[\varepsilon,\tau]$. Besides, using Fourier transform and the estimate from below for the symbol $m$ in~\eqref{eq:size.m},
\begin{equation*}
    |\partial_i\partial_j P_t(x)|\le\int_{\mathbb{R}^d}|\xi_i||\xi_j|e^{-tm(\xi)}\,\textup{d}\xi
    \le\int_{\mathbb{R}^d}|\xi|^2e^{-\varepsilon|\xi|^{\beta_1}}\,\textup{d}\xi=c_\varepsilon<\infty,\quad x\in\mathbb{R}^d,\ t\in[\varepsilon,\tau],
\end{equation*}
for all $i,j\in\{1,\dots,d\}$. These estimates together with~\eqref{z10} yield $|\Lambda P(x,y,t)|\leq C_{\varepsilon,\tau} (1\wedge|y|^{2})$ for some constant $C_{\varepsilon,\tau}>0$ depending only on $\varepsilon$ and $\tau$. Therefore,
\begin{align*}
    \int_{\varepsilon}^{\tau}\!\intr\! \intr &|\theta(x,t)| |\Lambda P(x,y,t)|K(y)\dy\dx\dt\leq C_{\varepsilon,\tau}\!\int_{\varepsilon}^{\tau}\!\intr\!\intr|\theta(x,t)|(1\wedge|y|^{2})K(y)\dy\dx\dt<\infty,
\end{align*}
since $\theta$ is a test function. As a consequence, using Fubini-Tonelli's Theorem again, we have that 	
\begin{equation}\label{z5}
    \begin{aligned}
        \int_0^T\intr\theta(x,t)\Lk P(x,t)\dx\dt&=\int_\varepsilon^\tau\intr \intr \theta(x,t) \Lambda P(x,y,t)K(y)\dy\dx\dt\\ &=\int_\varepsilon^\tau\intr K(y)\intr\theta(x,t) \Lambda P(x,y,t)\dx\dy\dt.
    \end{aligned}
\end{equation}
Combining~\eqref{z4} and~\eqref{z5},  we get that
\begin{equation*}
    \int_0^T\intr P(x,t)\Lk \theta\xt\dx\dt=\int_0^T\intr\theta\xt \Lk P(x,t)t\dx\dt.
\end{equation*}
Thus, since $P$ is a classical solution to~\eqref{a.nonlocalheatequation}, (see Proposition~\ref{pprop2}), and integrating by parts in time,  we conclude that
\begin{align*}
    \int_0^T\intr P(x,t)\Lk \theta\xt\dx\dt&=-\int_0^T\intr\theta\xt\partial_t P(x,t)\dx\dt=\int_0^T\intr P(x,t)\partial_t\theta\xt \dx\dt.\qedhere
\end{align*}
\end{proof}

%%%%%%%%%%%%%%%%%%%%%%%%%%%%%%%%%%%
\subsection{Completely monotone functions and  fractional heat kernels}\label{2.3}

We now devote our attention to the fractional Laplacian of order $\alpha$, $\alpha\in (0,2)$, denoted by $(-\Delta)^{\alpha/2}$. It is defined as the operator $\mathcal{L}_{K_\alpha}$ with L\'evy kernel, satisfying hypotheses \eqref{A0}--\eqref{A3},
\begin{equation}\label{eq:definition.fractional.laplacian}	
    K_{\alpha}(y)=\frac{C_{d,\alpha}}{|y|^{d+\alpha}}\quad\textrm{and }C_{d,\alpha}=\left(\intr \frac{1-\cos(\xi_{1})}{|\xi|^{d+\alpha}}\,{\rm d}\xi\right)^{-1}.
\end{equation}
The Fourier multiplier of this operator is $m(\xi)=|\xi|^\alpha$, whence the corresponding heat kernel, that we will denote by $P^{\alpha}_{t}(x)$ or $P^{\alpha}(x,t)$, is given by $\mathcal{F}(P_{t}^{\alpha})(\xi)=e^{-t|\xi|^{\alpha}}$ for all $t>0$. Since the multiplier is homogeneous, the heat kernel has a self-similar structure,
\begin{equation}\label{p5}
    P^{\alpha}(x,t)=t^{-d/\alpha}P^{\alpha}(xt^{-\alpha},1).
\end{equation}

We will obtain estimates for~$P^\alpha$ by means of Bernstein's Theorem for \emph{completely monotone} functions;  for a reference see, for instance,~\cite[Theorem 1.3.23]{Apple} {and \cite[Theorem 1.4]{Zoran}}.

\begin{definition}
    A function $g\in C^{\infty}([0,\infty))$ is \emph{completely monotone} if~$(-1)^{n}g^{(n)}\geq 0$ for all $n\in\mathbb{N}\cup \{0\}$.
\end{definition}

\noindent\emph{Example. } The function $g_{\alpha}(y)=e^{-y^{\alpha/2}}$,  $y\in[0,\infty)$,  is completely monotone for any $\alpha\in(0,2]$.
		
\begin{theorem}[Bernstein's Theorem]\label{BerTheorem}
    If $g$ is a completely monotone function, it is the Laplace transform of some nonnegative Borel measure $\mu$,
	\begin{equation*}
		g(y)=\int_{0}^{\infty}e^{-sy}\dmu(s),\quad y\in[0,\infty).
	\end{equation*}
\end{theorem}

\begin{remark}\label{remarkp1}
    (i) The Borel measure $\textup{d}\mu_{\alpha}$ associated to the function $g_{\alpha}$ in the above example is a probability measure since, by Bernstein's Theorem,
	\begin{equation*}\label{p11}	
        g_{\alpha}(y)=e^{-y^{\alpha/2}}=\int_{0}^{\infty}e^{-sy}\dmu_{\alpha}(s),\quad y\in[0,\infty),
	\end{equation*}
	which implies that $\displaystyle 1=g_{\alpha}(0)=\int_{0}^{\infty}\dmu_{\alpha}(s)$.
		
	\noindent (ii) If $\alpha=2$, then $\mu_{2}=\delta_{1}$, and if $\alpha=1$,
    $\textup{d}\mu_{1}(s)=(4\pi)^{-1/2}s^{-3/2}e^{-1/4s}\,\textup{d}s$.
\end{remark}

\begin{corollary}
	Let $P^{\alpha}$ be the heat kernel associated to the operator $(-\Delta)^{\alpha/2}$, $\alpha\in(0,2)$. Then,
	\begin{equation}\label{p1}
        P^{\alpha}(x,1)=\int_{0}^{\infty}G(x,s)\dmu_{\alpha}(s),\quad x\in\Rn,
	\end{equation}
    where $\mu_{\alpha}$ is the Borel measure given by Bernstein's Theorem and $G$ is the classical heat kernel associated to $-\Delta$, given by \eqref{p12}.
\end{corollary}

\begin{proof}
By Bernstein's Theorem, $\displaystyle e^{-|\xi|^{\alpha}}=g_{\alpha}(|\xi|^{2})=\int_{0}^{\infty}e^{-s|\xi|^{2}}\dmu_{\alpha}(s)$, whence
\begin{equation}\label {p2}	
    P^{\alpha}(x,1)=\frac{1}{(2\pi)^{d}}\intr e^{-|\xi|^{\alpha}}e^{-i\langle x,\xi\rangle}\,{\rm d}\xi
    =\frac{1}{(2\pi)^{d}}\intr\int_{0}^{\infty}e^{-s|\xi|^{2}}e^{-i\langle x,\xi\rangle}\dmu_{\alpha}(s)\,{\rm d}\xi.
\end{equation}
By  Tonelli's Theorem
\begin{equation*}
    \int_{0}^{\infty} G(0,s)\dmu_{\alpha}(s)=\int_{0}^{\infty}\intr\frac{e^{-s|\xi|^{2}}}{(2\pi)^{d}}\,{\rm d}\xi\dmu_{\alpha}(s)=
    \frac{1}{(2\pi)^{d}}\intr\int_{0}^{\infty}e^{-s|\xi|^{2}}\dmu_{\alpha}(s)\,{\rm d}\xi= P^{\alpha}(0,1)<\infty.
\end{equation*}
Therefore, since $G(x,s)\leq G(0,s)$ for all $x\in\Rn$, $s>0$,
\begin{equation*}
	\int_0^\infty G(x,s)\dmu_{\alpha}(s)<\infty\quad\textrm{for all }x\in\Rn.
\end{equation*}
As a consequence, by Fubini's Theorem, we can interchange the integrals in \eqref{p2} to get
\begin{equation*}
	P^{\alpha}(x,1)=\int_{0}^{\infty}\intr\frac{e^{-s|\xi|^{2}}}{(2\pi)^{d}}e^{-i\langle x,\xi\rangle}
    \,{\rm d}\xi\dmu_{\alpha}(s)=\int_{0}^{\infty}G(x,s)\dmu_{\alpha}(s)\quad\text{for all }x\in\Rn.\qedhere
\end{equation*}
\end{proof}

The representation \eqref{p1} combined with the self-similar structure~\eqref{p5} of the kernel $P^{\alpha}$ allows us to obtain estimates for its derivatives. Let us state some of them explicitly.

\begin{proposition}\label{pprop1}
    Let $P^{\alpha}$ be the heat kernel associated to the operator $(-\Delta)^{\alpha/2}$, $\alpha\in(0,2)$. There exists a positive constant $C=C(\alpha,d)$ such that for all $x\in\Rn$:
    \begin{itemize}
        \item[\rm (i)] $\displaystyle|\partial_{j}P^{\alpha}_1(x)|\leq  C\Big(1\wedge\frac{1}{|x|}\Big)P_1^{\alpha}(x)$  for all $j=1,\dots,d$;
        \item[\rm (ii)] $\displaystyle|\partial_{i}\partial_{j}P^{\alpha}_1(x)|\leq  C\Big(1\wedge\frac{1}{|x|^{2}}\Big)P^{\alpha}_1(x)$ for all $i,j=1,\dots,d$;
        \item[\rm (iii)] $|\langle x,\nabla P^{\alpha}_1(x)\rangle|\leq C(|x|\wedge1)P^{\alpha}_1(x)$;
        \item[\rm (iv)] $\displaystyle|\partial_{t}P^{\alpha}(x,t)|=|\Lk P^{\alpha}_t(x)|\leq \frac{C}{t}P^{\alpha}_t(x)$ for all $t\in (0,\infty)$;
        \item[\rm (v)] $\displaystyle|\partial_j P^{\alpha}_t(x)|\leq C\Big(\frac1{t^\alpha}\wedge\frac{1}{|x|}\Big) P^{\alpha}_1(x)$  for all $j=1,\dots,d$ and $t\in (0,\infty)$;
        \item[\rm (vi)] $\displaystyle|\partial_{i}\partial_{j} P^{\alpha}_{t}(x)|\leq C\Big(\frac1{t^{2\alpha}}\wedge\frac{1}{|x|}\Big)P_{1}^{\alpha}(x)$ for all $i,j=1,\dots,d$ and $t\in (0,\infty)$.
    \end{itemize}
\end{proposition}

\begin{proof}
(i)--(ii)  These properties are clearly satisfied for any bounded domain, since $P^{\alpha}_1>0$ and $P^{\alpha}_1\in C^{\infty}(\Rn)$; see Corollary \ref{pcor1}.  Therefore, it is enough to prove them for $|x|\geq 2$.
	
A direct computation of the spatial derivatives of the classical heat kernel $G$ yields
\begin{equation*}
    \begin{array}{ll}
    \displaystyle\partial_{j}G(x,s)=-\frac{x_{j}}{2s}G(x,s),\qquad	&\displaystyle\partial_{i}\partial_{j}G(x,s)=\Big(\frac{-\delta_{ij}}{2s}+\frac{x_{i}x_{j}}{8s^2}\Big)G(x,s),
    \\[6pt]
    \displaystyle|\partial_{j}G(x,s)|\le\frac{|x|}{2s}G(x,s),\qquad&\displaystyle|\partial_{i}\partial_{j}G(x,s)|\le\Big(\frac{1}{2s}
    +\frac{|x|^{2}}{8s^2}\Big)G(x,s)
	\end{array}
\end{equation*}
for all $(x,s)\in\Rn\times (0,\infty)$.  Then, since $a^{k}\leq k!e^{a}$ for all $k\in\mathbb{N}$ and $a\geq 0$, we have that
\begin{align}
	\label{p3} &|x||\partial_{j}G(x,s)|\leq 4e^{\frac{|x|^{2}}{8s}}G(x,s)=4G(x/2,s),\\
	\label{p4} &|x|^{2}|\partial_{i}\partial_{j}G(x,s)|
    \leq 20e^{\frac{|x|^{2}}{8s}}G(x,s)=20G(x/2,s).
\end{align}
Since $G(x,s)\leq G(0,s)$ for all $(x,s)\in\Rn\times (0,\infty)$, using the MVT and these estimates we get
\begin{align*}
	&\Big|\frac{G(x+he_{j},s)-G(x,s)}{h}\Big|=|\partial_{j}G(x'(x,h,s),s)|\leq \frac{4}{|x'|}G(0,s)\leq\frac{C}{|x|}G(0,s), \\
	&\Big|\frac{\partial_{j}G(x+he_{i},s)-\partial_{j}G(x,s)}{h}\Big|=|\partial_{i}\partial_{j}G(x''(x,h,s),s)|\leq \frac{20}{|x''|^{2}}G(0,s)
    \leq \frac{C}{|x|^{2}}G(0,s),
\end{align*}
if $|h|\leq 1$,	because $x',x''\in B_{1}(x)$ for all $(x,s)\in\Rn \times (0,\infty)$.
	
Since $G(0,\cdot)\in L^{1}((0,\infty),\textup{d}\mu_{\alpha})$, see~\eqref{p1}, we are allowed to interchange the derivatives and the integral in the representation formula~\eqref{p1} for the heat kernel via the DCT and, by \eqref{p3} and~\eqref{p4}, we obtain
\begin{align*}
	&|\partial_{j}P^{\alpha}_1(x)|=\Big|\intoii \partial_{j} G(x,s)\dmu_{\alpha}(s)\Big|
    \leq \frac{C}{|x|}\intoii G(x/2,s)\dmu_{\alpha}(s)=\frac{C}{|x|}P^{\alpha}_1(x/2),\\
	&|\partial_{i}\partial_{j}P^{\alpha}_1(x)|=\Big|\intoii \partial_{i} \partial_{j} G(x,s)\dmu_{\alpha}(s)\Big|
    \leq \frac{C}{|x|^{2}}\intoii G(x/2,s)\dmu_{\alpha}(s)=\frac{C}{|x|^{2}}P^{\alpha}_1(x/2).
\end{align*}
Finally, since $P^{\alpha}$ is slowly changing, see Proposition \ref{Pslowly}, there is $C>0$ such that $P^{\alpha}_1(x/2)\leq CP^{\alpha}_1(x)$ for all $x\in \Rn$, and the result follows.

\noindent (iii) It is a direct consequence of (i).

\noindent (iv) Using the self-similar form~\eqref{p5} of $P^\alpha$ and (iii), we get
\begin{equation*}
    |\partial_{t}P^{\alpha}(x,t)|=t^{-(1+\frac d\alpha)}\Big|\frac{d}{\alpha}P^{\alpha}_1(xt^{-\alpha})+\frac{1}{\alpha}\langle xt^{-\alpha}, \nabla P^{\alpha}_1(xt^{-\alpha})\rangle\Big|\leq\frac{C}{t}P^{\alpha}_1(x)
\end{equation*}
for all $t>0$. The estimate for $|\Lk P^{\alpha}_t|$ follows immediately, since $P^\alpha$ solves the fractional heat equation; see~Proposition~\ref{pprop2}.

\noindent(v) Using again the self-similarity of $P^{\alpha}$, now in combination with (i), we obtain
\begin{equation*}
    |\partial_j P^\alpha_t(x)|=t^{-(\alpha+\frac d\alpha)}|\partial_j P^\alpha_1(xt^{-\alpha})|\le C
    \Big(\frac1{t^\alpha}\wedge\frac{1}{|x|}\Big) P^{\alpha}_1(x).
\end{equation*}

\noindent(vi) Using once more the self-similarity of $P^{\alpha}$, now together with (ii),
\begin{equation*}
    |\partial_i\partial_jP^\alpha_t(x)|= t^{-(2\alpha+\frac d\alpha)}|\partial_i\partial_jP^\alpha_1(xt^{-\alpha})|\leq C\Big(\frac1{t^{2\alpha}}\wedge\frac{1}{|x|}\Big)P_{1}^{\alpha}(x).\qedhere
\end{equation*}
\end{proof}

%%%%%%%%%%%%%%%%%%%%%%%%%%%%%%%%%%%%%%%%%%%%%%%%%%%%%%%%%%%%%%%%%%%%%%%%%%%%%%%%%%%%%%%%%
%%%%%%%%%%%%%%%%%%%%%%%%%%%%%%%%%%%%%%%%%%%%%%%%%%%%%%%%%%%%%%%%%%%%%%%%%%%%%%%%%%%%%%%%%
\section{Operators \texorpdfstring{$\mathcal{L}_K$}{L\_K}}\label{section3}
\setcounter{equation}{0}

This section constitutes the main body of the paper. It is devoted to develop a complete Widder-type theory for~\eqref{a.nonlocalheatequation} when $\mathcal{L}=\mathcal{L}_K$, with $K$ satisfying~\eqref{A0}--\eqref{A2}, plus~\eqref{A3} for some results. It is divided into two parts. In the first one we focus on very weak solutions, addressing three key topics: the  existence of an  initial trace when the solutions are nonnegative, uniqueness under various conditions, and existence. In the second part, we turn our attention to classical solutions. We begin by exploring uniqueness. Later, we give conditions on the heat kernel guaranteing the existence of classical solutions. Such conditions are satisfied by the heat kernel of the fractional Laplacian.

%%%%%%%%%%%%%%%%%%%%%%%%%%%%%%%%%%%%%%%%%%%%%%%%%%%%%%%
\subsection{Very weak solutions}

We now give conditions ensuring uniqueness for very weak solutions. In addition, we prove, on the one hand, that nonnegative very weak solutions have an initial trace which is a Radon measure satisfying certain integral growth condition and, on the other hand, the existence of a solution for any initial data satisfying a growth condition that coincides with the above mentioned one when the data are nonnegative.  As a corollary, we obtain a representation formula. As a first step, we discuss a useful alternative definition of very weak solution.

%%%%%%%%%%%%%%%%%%%%%%%%%%%%%%%%%%%%%%%%%%%%%%%%%%%%%%%
\subsubsection{An alternative definition of very weak solution}

Our notion of very weak solution requires $u$ to be integrable against $\Lk\theta$ for any test function $\theta$. We will see now that this is equivalent to requiring $u$ to be integrable against the heat kernel at any positive time, say $t=1$. This equivalence relies on the L\'evy kernel $K$ and $P_{1}$ being comparable far from the origin (see~\eqref{Atotal}) and on $P_{1}$ being {slowly changing} by Proposition~\ref{Pslowly}.

\begin{proposition}\label{prop1}
    Let $\mathcal{L}=\mathcal{L}_K$, with $K$ satisfying~\eqref{A0}--\eqref{A2}.

    \noindent{\rm (i)} Let  $0\leq\delta<\tau\leq T$. Then $u\in L^{1}((\delta,\tau);\lp)$ if and only if $u\in L^{1}_{\rm{loc}}(\Rn\times[\delta,\tau])$ and
    \begin{equation}\label{v3.9}	
        \int_{\delta}^{\tau}\intr|u(x,t)||\Lk \theta(x,t)|\dx\dt<\infty\quad\textrm{for all }\theta\in C^{\infty}_{\textup{c}}(\RN\times [\delta,\tau]).
    \end{equation}

    \noindent{\rm (ii)}  $u\in\luloc$ if and only if  $u\in L^{1}_{\rm{loc}}(\Rn\times(0,T))$ and~\eqref{a.2} holds.
\end{proposition}

\begin{proof}
(i) $\Rightarrow$) Since $u\in L^{1}((\delta,\tau);\lp)$, the local integrability in $\Rn\times[\delta,\tau]$ follows immediately from the positivity of $P_1$. On the other hand, by Lemma~\ref{a.17}, for a given $\theta\in C^{\infty}_{\textup{c}}(\RN\times[\delta,\tau])$,
\begin{equation*}\label{v3.20}
    \int_{\delta}^{\tau}\intr|u(x,t)||\Lk \theta(x,t)|\dx\leq C_{\theta} \int_{\delta}^{\tau}\intr|u(x,t)|P_{1}(x)\dx\dt<\infty.
\end{equation*}

\noindent $\Leftarrow$) Consider $\psi\in C^{\infty}_{\textup{c}}(\Rn)$ such that $\mathcal{X}_{B_{1}(0)}\leq \psi\leq \mathcal{X}_{B_{2}(0)}$. Since $\psi\ge0$,  $K$ is symmetric and $\textrm{supp}(\psi)\subset B_{2}(0)$, if $|x|\geq4$, then  $	|\Lk\psi(x)|=\int_{B_{2}(x)}\psi(x-y)K(y)\dy$. Notice that we are avoiding the singularity of $K$, since $|y|\ge 2$ in the domain of integration. On the other hand, if $|x|\geq 4$ and $|y-x|\leq 2$, then $|y|\leq3 |x|/2$, whence
\begin{equation*}
    K(y)\geq\frac{C}{|y|^d\tilde{\varphi}(|y|)}\geq  \frac{C}{\big(\frac{3^d|x|^d}{2^{n}}\big)\tilde{\varphi}\big(\frac{3|x|}{2}\big)}\geq CK\big(\frac{3x}{2}\big)\geq  CP_{1}\big(\frac{3x}{2}\big)\geq CP_{1}(x)\quad\textrm{whenever }y\in B_{2}(x),
\end{equation*}
due to \eqref{A1}, the fact that $\tilde\varphi$ is nondecreasing, \eqref{Atotal}, and the fact that $P_{1}$ is {slowly changing}, see Proposition~\ref{Pslowly}. Then,
\begin{equation}\label{v3.24}
	|\Lk\psi(x)|=\int_{B_{2}(x)}\psi(x-y)K(y)\dy\geq CP_{1}(x)\int_{B_{2}(0)}\psi(y)\dy = CP_{1}(x)\quad\textrm{for all }|x|\geq 4.
\end{equation}
As a consequence, using also that $P_1\le P_1(0)$, $u\in L^{1}_{\textrm{loc}}(\Rn\times[\delta, \tau])$ and \eqref{v3.9},
\begin{equation*}\label{v3.22}	
    \int_{\delta}^{\tau}\intr |u\xt|P_{1}(x)\dx\dt\leq P_{1}(0) \int_{\delta}^{\tau}\int_{B_{4}(0)} |u\xt|\dx\dt+C\int_{\delta}^{\tau}\intr |u\xt||\Lk\psi|(x)\dx\dt <\infty.
\end{equation*}

\noindent (ii) $\Rightarrow$) This follows directly from (i) $\Rightarrow$).

\noindent $\Leftarrow$) Let $0<\varepsilon<T'<T$. We consider $\psi$ as in the proof of (i) $\Leftarrow$) and  $\varphi:[0,T]\to\mathbb{R}$ such that $\mathcal{X}_{[\varepsilon,T']}\leq \varphi\leq \mathcal{X}_{[\varepsilon/2, T'+(T-T')/2]}$. Then, $\varphi\psi\in C^{\infty}_{\textup{c}}(\Rn\times(0,T))$. Hence, since $P_1\le P_1(0)$, by~\eqref{a.2} and~\eqref{v3.24}, with $\varphi\psi$ replacing $\psi$, we have
\begin{align*}
    \int_{\varepsilon}^{T'}&\intr |u\xt|P_{1}(x)\dx\dt\leq P_{1}(0)\int_{\varepsilon}^{T'}\int_{B_{4}(0)}|u\xt|\dx\dt+C\int_{\varepsilon}^{T'}\intr |u\xt||\Lk\psi|(x)\dx\dt\\
    &\leq  P_{1}(0) \int_{\varepsilon}^{T'}\int_{B_{4}(0)} |u\xt|\dx\dt+C\int_{\varepsilon}^{T'}\intr |u\xt||\Lk(\varphi\psi)|\xt\dx\dt
	<\infty.\qedhere
\end{align*}
\end{proof}

\begin{remark}\label{remark3}
	Thanks to Proposition~\ref{prop1} (ii), if $u$ is a very weak solution of equation \eqref{a.nonlocalheatequation}, then $u\in\luloc$.
\end{remark}

When dealing with very weak solutions to the Cauchy Problem for~\eqref{a.nonlocalheatequation}, a common choice is asking $u$ to be a very weak solution of the equation in the sense of Definition~\ref{defWS} and to take the data (in general a Radon  measure $\mu_{0}$) in the sense of Definition~\ref{deftrace}; see for instance~\cite{Matteo}. This is our choice in the present paper. Another option is asking $u$ to satisfy
\begin{align}
    \label{eq:space.alternative.definition}
    &u\in L^1_{\rm loc}([0,T);\lp);\\
    \label{v3.10}
    &\int_{0}^{T}\!\intr u\xt(\partial_{t}\tilde{\theta}\xt-\Lk\tilde{\theta}\xt)\dx\dt=-\intr\tilde{\theta}(x,0)\dmu_{0}(x)\quad\textrm{for all }\tilde{\theta}\in C^{\infty}_{\textup{c}}(\Rn\times [0,T));
\end{align}
see for instance~\cite{Soria1} for a definition in this spirit. Both choices have advantages and disadvantages. Formulation \eqref{eq:space.alternative.definition}--\eqref{v3.10} facilitates obtaining uniqueness without any requirement of positivity or integrability, excepting~\eqref{eq:space.alternative.definition}, although it requires such conditions to establish existence. On the contrary, proving existence is more straightforward with our choice, but you pay the price of needing additional assumptions to achieve uniqueness.

If $\mathcal{L}=\mathcal{L}_K$ with $K$ satisfying~\eqref{A0}--\eqref{A2} and $u\in\li$, both notions are equivalent. We will take profit of this equivalence later.

\begin{proposition}\label{equivDefvwProp}
    Let $\mathcal{L}=\mathcal{L}_K$ with $K$ satisfying~\eqref{A0}--\eqref{A2}. Let $\mu_{0}$ be a Radon measure and $u\in\li$.
    	
    \noindent{\rm (i)} If $u$ satisfies~\eqref{v3.10}, then it is a very weak solution of~\eqref{a.nonlocalheatequation} with initial trace~$\mu_0$.
    	
    \noindent{\rm(ii)} If $u$ is a very weak solution to~\eqref{a.nonlocalheatequation} with initial trace $\mu_{0}$, then it  satisfies~\eqref{eq:space.alternative.definition}--\eqref{v3.10}.
\end{proposition}

\begin{proof}
(i) Condition~\eqref{a.2} follows from Proposition~\ref{prop1}, since $L^\infty((0,T);L^1_{P_1})\subset L^1_{\rm loc}((0,T);L^1_{P_1})$.
Moreover, since $C_{\textup{c}}^{\infty}(\Rn\times(0,T))\subset C_{\textup{c}}^{\infty}(\Rn\times[0,T))$, \eqref{v3.10} implies~\eqref{a.3}.

It remains to prove that $\mu_{0}$ is the initial trace of $u$. To this aim, given $\psi\in C_{\textup{c}}(\mathbb{R}^d)$, we would like to take $H(\tau-\cdot)\psi$, $\tau\in(0,T)$, where $H$ is the Heaviside function, as test function in~\eqref{v3.10}, and then pass to the limit $\tau\to0^+$. However, this test is not admissible, since it does not have the required smoothness; hence, we have to proceed through an approximation argument.

Approximating $\psi$ is easy: we take a sequence $\{\psi_j\}_{j\in\mathbb{N}}\subset C^\infty_{\textup{c}}(\mathbb{R}^d)$ converging to $\psi$ uniformly in~$\mathbb{R}^d$ such that $|\psi_j|\le M \mathcal{X}_{B_R(0)}$ for some constants $M>0$, $R>0$.

In order to approximate the time factor, we consider $\displaystyle\varphi_{k}(t)= \int_{2^kt}^{\infty}\eta(s) \,{\rm d}s$, $k\in\mathbb{N}$, $t\in\mathbb{R}$, where
\begin{equation}\label{etaParaCutOff}
		\eta(s):=
		\begin{cases}\displaystyle
			C_{1}\textup{e}^{-\frac{1}{1-s^{2}}}&\text{if }|s|<1,\\
			0&\textrm{otherwise},
		\end{cases}
        \qquad C_{1}^{-1}:=\int_{-1}^{1}\textup{e}^{-\frac{1}{1-s^{2}}}\,{\rm d}s.
\end{equation}
Then $\{(\varphi_{k})'\}_{k\in\mathbb{N}}=\{-\eta_{k}\}_{k\in\mathbb{N}}$, with $\eta_{k}(s):=2^{k}\eta(2^{k}s)$, $s\in\mathbb{R}$. Notice that $\{\eta_{k}\}_{k\in\mathbb{N}}$ is a summability kernel. The desired approximating sequence is then $\{\varphi_k^\tau\}_{k\in\mathbb{N}}$, where $\varphi_k^{\tau}(t)=\varphi_k(t-\tau)$ for $\tau\in (0,T)$, $k\geq k_0>-\ln (T-\tau)/\ln 2$ and $t\in\mathbb{R}$.

Let $\displaystyle g_j(t)=\intr \psi_j(x)u(x,t)\dx$. Take $\varphi^{\tau}_{k}\psi_j\in C_{\textup{c}}^{\infty}(\Rn\times[0,T))$ as test function in~\eqref{v3.10}. Then,
\begin{equation}\label{eq:weak.for.trace}	
    \int_{0}^{T}(\varphi_{k}^{\tau})'(t)g_j(t)\dt-\int_{0}^{T}\intr\varphi_{k}^{\tau}(t)u(x,t)\Lk\psi_j(x)\dx\dt
    =-\intr\varphi^{\tau}_{k}(0)\psi_j\dmu_{0}.
\end{equation}
Since $|\psi_j|\le M \mathcal{X}_{B_R(0)}$ for all $j\in\mathbb{N}$ and $P_1>0$, then $|\psi_j|\le C P_1$ for some $C>0$. Hence, $g_j\in L^1((0,T))$,  as $u\in L^\infty((0,T);L^1_{P_1})\subset L^1((0,T);L^1_{P_1})$. Therefore, $(\varphi_{k})'*g_j(\tau)\to -g_j(\tau)$ if $\tau\in (0,T)\cap L(g_j;\{\eta_k\}_{k\in\mathbb{N}})$, whence
\begin{equation*}\label{v3.12}	
    \lim_{k\to\infty}\int_0^T(\varphi^\tau_{k})'(t)g_j(t)\dt=-g_j(\tau)\quad\textrm{for all }\tau\in(0,T)\cap L(g_j;\{\eta_k\}_{k\in\mathbb{N}}).
\end{equation*}
Note that, by Lebesgue's Differentiation Theorem, $|(0,T)\setminus L(g_j;\{\eta_k\}_{k\in\mathbb{N}})|=0$.

Since $\mathcal{X}_{(-\infty,-2^{-k}]}\leq \varphi_{k}\leq \mathcal{X}_{(-\infty,2^{-k})}$, then $0\le\varphi_k^\tau\le 1$. Besides, $|\Lk\psi_j|\leq C_jP_{1}$ (see Lemma~\ref{a.17}). Hence,
\begin{equation*}
    \begin{aligned}
        \limsup_{k\to\infty}\Big| \int_0^T\intr\varphi_k^\tau(t)u\xt\Lk\psi_j(x)\dx\dt\Big|&\leq C_j\lim_{k\to\infty}\int_0^{\tau+2^{-k}}\int_{\mathbb{R}^d}u(x,t)P_1(x)\dx\dt\\
        &=C_j\int_0^\tau\int_{\mathbb{R}^d}u(x,t)P_1(x)\dx\dt.
    \end{aligned}
\end{equation*}

On the other hand, $\varphi^\tau_k(0)=\varphi_k(-\tau)=1$ if $\tau\in (0,T)$ and $k\ge -\log\tau/\log 2$, whence
\begin{equation*}
    \lim_{k\to\infty}\intr\varphi^\tau_k(0)\psi_j\dmu_{0}=\intr\psi_j\dmu_0\quad\textrm{for all }\tau\in(0,T)\cap L(g_j;\{\eta_k\}_{k\in\mathbb{N}}).
\end{equation*}
Thus, letting $k\to\infty$ in~\eqref{eq:weak.for.trace}, we obtain
\begin{equation*}
    \Big|\intr\psi_j(x)u(x,\tau)\dx-\intr\psi_j\dmu_0\Big|\le C_j\int_0^\tau\int_{\mathbb{R}^d}u(x,t)P_1(x)\dx\dt\quad\textrm{for all }\tau\in(0,T)\cap L(g_j;\{\eta_k\}_{k\in\mathbb{N}}),
\end{equation*}
whence
\begin{equation}\label{eq:trace.approximates}
    \operatornamewithlimits{ess\, lim}_{\tau\to0^+}\Big|\intr\psi_j(x)u(x,\tau)\dx-\intr\psi_j\dmu_0\Big|=0.
\end{equation}
This will be enough to complete the approximation argument. Indeed, on the one hand, since $|\psi_j-\psi|=|\psi_j-\psi|\mathcal{X}_{B_R(0)}$, $j\in\mathbb{N}$, then
\begin{equation*}
    \intr|\psi_j-\psi|\,\textup{d}|\mu_0|\le \|\psi-\psi_j\|_{L^\infty(\mathbb{R}^d)} |\mu_0|(B_R(0)).
\end{equation*}
On the other hand, as $P_1>0$, there is a constant $C>0$ such that $\mathcal{X}_{B_R(0)}\le CP_1$, so that
\begin{align*}
    \intr|\psi-\psi_j|(x)u(x,\tau)\dx&\le C\|\psi-\psi_j\|_{L^\infty(\mathbb{R}^d)}\intr u(x,\tau)P_1(x) \dx\\
    &\le C\|\psi-\psi_j\|_{L^\infty(\mathbb{R}^d)}\|u\|_{\li}.
\end{align*}
Therefore,
\begin{align*}
    \Big|\intr\psi(x)u(x,\tau)\dx-\intr\psi\dmu_0\Big|&\le \Big|\intr\psi_j(x)u(x,\tau)\dx-\intr\psi_j\dmu_0\Big|\\
    &\quad+\|\psi-\psi_j\|_{L^\infty(\mathbb{R}^d)}\big(|\mu_0|(B_R(0))+C\|u\|_{\li}\big).
\end{align*}
Since $|(0,T)\setminus\cap_{j=1}^{\infty}L(g_j;\{\eta_k\}_{k\in\mathbb{N}})|=0$, and thanks to~\eqref{eq:trace.approximates}, letting first $\tau\to0^+$ and then $j\to\infty$, we arrive at~\eqref{a.1}, as desired.

\noindent (ii) Condition~\eqref{eq:space.alternative.definition} follows directly from $u\in\li$.

Let us check that $u$ satisfies~\eqref{v3.10}. We define
\begin{equation*}
	\tilde\varphi_k(t)=\int^{2^kt}_{-\infty}\eta(s)\,{\rm d}s,\quad k\in\mathbb{N},\ t\in\mathbb{R}.
\end{equation*}
It satisfies that $\mathcal{X}_{[\frac{1}{2^{k}},\infty)}\leq \tilde{\varphi}_{k}\le\mathcal{X}_{[-\frac{1}{2^{k}},\infty)}$ for all $k\in\mathbb{N}$ and $(\tilde{\varphi}_{k})'=\eta_{k}$. For fixed $\tau\in (0,T)$, let $\tilde{\varphi}^{\tau}_{k}(t)=\tilde{\varphi}_{k}(t-\tau)$. Given  $\tilde{\theta}\in C^{\infty}_{\textup{c}}(\Rn\times [0,T))$, the function $\tilde{\varphi}^{\tau}_{k}\tilde{\theta}\in C^{\infty}_{\textup{c}}(\Rn\times (0,T))$ for all $k>-\ln \tau/\ln 2$. Hence, because $u$ is a very weak solution,
\begin{equation}\label{eq:approximate.weak}
    \underbrace{\int_{0}^{T}(\tilde{\varphi}_{k}^{\tau})'(t)\intr u\xt\tilde{\theta}\xt\dx\dt}_{\textup{I}_k(\tau)}
    +\underbrace{\int_{0}^{T}\tilde{\varphi}_{k}^{\tau}(t)
    \intr u\xt(\partial_{t}\tilde{\theta}\xt-\Lk\tilde{\theta}\xt)\dx\dt}_{\textup{II}_k(\tau)}=0.
\end{equation}
Let us analyze the double limit, first a $k\to\infty$ and then as $\tau\to0^+$, of $\textup{I}_k(\tau)$ and $\textup{II}_k(\tau)$.

As $P_1$ is strictly positive, we can make the decomposition $\textup{I}_k(\tau)=\textup{I}_{1,k}(\tau)+\textup{I}_{2,k}(\tau)$, where
\begin{align*}	
    \textup{I}_{1,k}(\tau)&:=\int_0^T(\tilde{\varphi}_k^\tau)'(t)\intr u\xt\tilde{\theta}(x,0)\dx\dt,\\ \textup{I}_{2,k}(\tau)&:=\int_0^T(\tilde{\varphi}_k^\tau)'(t)\intr u\xt\frac{\tilde{\theta}(x,t)-\tilde{\theta}(x,0)}{P_1(x)}P_1(x)\dx\dt.
\end{align*}
Since $|\tilde{\theta}(\cdot,0)|\leq CP_{1}$ and $u\in \li$, then $\intr u(x,\cdot)\tilde{\theta}(x,0)\dx\in L^1((0,T))$. Therefore, because $\{\eta_k\}_{k\in\mathbb{N}}$ is a summability kernel,
\begin{equation*}
    \lim_{k\to\infty}\textup{I}_{1,k}(\tau)=\intr u(x,\tau)\tilde{\theta}(x,0)\dx\quad\textrm{for a.e. }\tau\in (0,T),
\end{equation*}
whence, as $\mu_{0}$ is the initial trace of $u$,
\begin{equation*}
    \operatornamewithlimits{ess\, lim}_{\tau\to0^+} \lim_{k\to\infty}\textup{I}_{1,k}(\tau)= \operatornamewithlimits{ess\, lim}_{\tau\to0^+} \intr u(x,\tau)\tilde{\theta}(x,0)\dx=\intr \tilde{\theta}(x,0)\dmuo(x).
\end{equation*}

On the other hand, since $P_{1}$ is strictly positive and $\tilde{\theta}\in C_{\textup{c}}(\Rn\times[0,T))$, by the MVT,
\begin{equation*}
    \frac{|\tilde{\theta}(x,s)-\tilde{\theta}(x,0)|}{P_{1}(x)}\leq \frac{\|\partial_{t}\theta\|_{L^{\infty}(\Rn\times(0,T))}}{\underset{(y,s) \in\textrm{supp}(\tilde{\theta})}{\min}\{P_{1}(y)\}}t=C_{\tilde{\theta},P_1}t\quad\textrm{for all }\xt\times\Rn\times[0,T).
\end{equation*}
Besides, $(\tilde{\varphi}_{k}^{\tau})'=\eta_{k}(\cdot-\tau)\geq 0$ and $\operatorname{supp}(\eta_{k}(\cdot-\tau))\subset[\tau-2^{-k},\tau+2^{-k}]$, so that
\begin{align*}
	|\textup{I}_{2,k}(\tau)|&\leq C_{\tilde{\theta},P_1}\int_{\tau-2^{-k}}^{\tau+2^{-k}} t\eta_{k}(t-\tau)\intr|u(x,t)|P_{1}(x)\dx\dt\\
    &\leq C_{\tilde{\theta},P_1}\|u\|_{L^\infty((0,T);L^{1}_{P_1})}\int_{\tau-2^{-k}}^{\tau+2^{-k}} t\eta_{k}(t-\tau)\dt.
\end{align*}
Since $\varphi^{\tau}_{k}(\tau-2^{-k})=0$ and $\varphi^{\tau}_{k}(\tau+2^{-k})=1$, integrating by parts we get
\begin{equation*}
    0\le \int_{\tau-2^{-k}}^{\tau+2^{-k}} t\eta_{k}(t-\tau)\dt=\tau+2^{-k}- \int_{\tau-2^{-k}}^{\tau+2^{-k}}\varphi_{k}^{\tau}(t)\dt\le\tau+ 2^{1-k},
\end{equation*}
because  $\varphi_k^\tau\ge0$. Thus,
\begin{equation*}
    \operatornamewithlimits{ess\, lim}_{\tau\to0^+} \limsup_{k\to\infty}|\textup{I}_{2,k}(\tau)|\le C_{\tilde{\theta},P_1}\|u\|_{L^\infty((0,T);L^{1}_{P_1})}\operatornamewithlimits{ess\, lim}_{\tau\to0^+}\limsup_{k\to\infty}(\tau+ 2^{1-k})=0.
\end{equation*}

Concerning $\textup{II}_k(\tau)$, since $\li\subset L^{1}((0,T); L^{1}_{P_1})$, Proposition~\ref{prop1} (i) implies that
\begin{equation*}
    \int^T_0\intr |u\xt|(|\partial_{t}\tilde{\theta}\xt|+|\Lk\tilde{\theta}\xt|)\dx\dt<\infty.
\end{equation*}
As a consequence, since $|\tilde{\varphi}_k^\tau|\le 1$, we may apply the DCT twice to obtain
\begin{align*}
    \operatornamewithlimits{ess\, lim}_{\tau\to0^+} \lim_{k\to\infty}\textup{II}_k(\tau)&=\operatornamewithlimits{ess\, lim}_{\tau\to0^+}\int_\tau^T\intr u\xt(\partial_{t}\tilde{\theta}\xt-\Lk\tilde{\theta}\xt)\dx\dt\\
    &=\int_0^T\intr u\xt(\partial_{t}\tilde{\theta}\xt-\Lk\tilde{\theta}\xt)\dx\dt.
\end{align*}

Summarizing: by letting first $k\to\infty$ and then $\tau\to0^+$ in \eqref{eq:approximate.weak} we obtain~\eqref{v3.10}.
\end{proof}
\begin{remark}
	{An equivalence result similar to Proposition~\ref{equivDefvwProp} is given in~\cite{Felix} for a family of nonlocal operators wider than the one considered here, although within a class of solutions more restrictive than~$\li$.}
\end{remark}

%%%%%%%%%%%%%%%%%%%%%%%%%%%%%%%%%%%%%%%%%%%%%%%%%%%%%%%
\subsubsection{Uniqueness of very weak solutions}

We will prove uniqueness of very weak solutions within  two different classes: $\li$ and the class of the ones that are nonnegative.

The strategy is as follows. We first  prove uniqueness for solutions in the sense of the alternative definition~\eqref{eq:space.alternative.definition}--\eqref{v3.10}, whence, thanks to Proposition~\ref{equivDefvwProp}, we obtain uniqueness for very weak solutions within $\li$. Next, we prove a  smoothing result for nonnegative very weak solutions, Lemma~\ref{lemmatrace} below, showing that they belong to $\li$, from where uniqueness follows.

We start with the uniqueness result for solutions in the sense of~\eqref{eq:space.alternative.definition}--\eqref{v3.10}.

\begin{theorem} \label{UniquenessSoria}
    Let $\mathcal{L}=\mathcal{L}_K$ with $K$ satisfying~\eqref{A0}--\eqref{A2}. Let $\mu_{0}$ be a Radon measure. Let $u_1$ and $u_2$ satisfy~\eqref{eq:space.alternative.definition}--\eqref{v3.10}. Then, $u_1=u_2$ a.e.\,in $\mathbb{R}^d\times(0,T)$.
\end{theorem}

\begin{proof}
Let $u=u_1-u_2$. Then $u\in L^{1}_{\rm loc}([0,T);\lp)$ and
\begin{equation}\label{v3.15}	
    \int_{0}^{T}\intr u(x,t)(\partial_{t}\tilde{\theta}\xt-\Lk\tilde{\theta}\xt)\dx\dt=0\quad\textrm{for all }\tilde{\theta}\in C^\infty_{\textup{c}}(\Rn\times [0,T)).
\end{equation}
We have to prove that $u=0$ for a.e.~in $\Rn\times(0,T)$. For that purpose, it is enough to show that
\begin{equation}\label{a.22}	
    \int_0^T\intr u(x,t)\tilde{\theta}(x,t)\dx\dt=0\quad\textrm{for all }\tilde{\theta}\in C^\infty_{\textup{c}}(\RN\times[0,T)).
\end{equation}
To this aim we will use Hilbert's duality method.  We consider a solution $\varphi$ for the backward problem
\begin{equation}\label{a.6}
	\partial_{t}\varphi-\Lk\varphi=\tilde{\theta}\quad\textrm{in }\Rn\times(0,T).
\end{equation}
If $\varphi$ belonged to $C^\infty_{\textup{c}}(\Rn\times [0,T))$, we could use it as a test function in~\eqref{v3.15} to obtain
\begin{equation*}
	0=\int_{0}^{T}\intr u\xt(\partial_t\varphi\xt-\Lk\varphi\xt)\dx\dt=\int_{0}^{T}\intr u(x,t)\tilde{\theta}(x,t)\dx\dt,
\end{equation*}
and we would be done. Let $0<\too<T$ and $R_0>0$ be such that $\operatorname{supp}(\tilde{\theta})\subset B_{R_0}\times[0,\too]$. By Duhamel's formula,
\begin{equation}\label{eq:solution.backward.problem}
	\varphi\xt=\begin{cases}
		\displaystyle
		\int_{0}^{\too-t}\intr P_{\too-t-s}(x-y)\tilde{\theta}(y,s)\dy\,{\rm d}s,\quad&\xt\in\Rn\times[0,\too),\\
		\varphi=0,& \xt\in\Rn\times[\too,T)
	\end{cases}
\end{equation}
is a solution to~\eqref{a.6}. However, though it has the required smoothness, it is not compactly supported in space. Despite this fact, we will show that it is anyway an admissible test function for \eqref{v3.15}.

Let $\{\psi_R\}_{R\ge R_0}$ be the family of cutoff functions given by~\eqref{cutoffspace}. Since $\psi_{R}\varphi\in C_{\textup{c}}^\infty(\Rn\times[0,T))$, we can use it as a test function in~\eqref{v3.15}, and we have
\begin{equation}\label{a.20}
    \int_{0}^{T}\intr u\xt(\psi_{R}\partial_t\varphi-\Lk(\psi_{R}\varphi))\xt \dx\dt=0.
\end{equation}
The operator $\Lk$ satisfies the nonlocal Leibniz's formula
\begin{equation}\label{leibniz}	
    \Lk(uv)=u\Lk v +v\Lk u -B(u,v),
\end{equation}
where  $B(u,v)$ is the bilinear form given by
\begin{equation}\label{51.2}	
    B(u,v)(x)=\intr(u(x)-u(x-y))(v(x)-v(x-y))K(y)\dy.
\end{equation}
Then, because of~\eqref{a.6}, if $R\ge R_0$ we have  that
\begin{equation*}\label{a.21}
    \psi_{R}\partial_t\varphi-\Lk(\psi_{R}\varphi)=\tilde{\theta}-\varphi\Lk\psi_{R}+B(\psi_{R},\varphi),
\end{equation*}
since  $\psi_{R}\tilde{\theta}=\tilde{\theta}$ for all $R\geq R_{0}$, whence we may rewrite \eqref{a.20} as
\begin{align}\label{eq:with.R}
	0&=\int_{0}^{T}\intr u(x,t)\tilde{\theta}(x,t)\dx\dt-\int_{0}^{T}\intr u\xt\varphi\xt\Lk\psi_{R}\xt\dx\dt\\
	\notag&\quad+\int_{0}^{T}\intr u\xt B(\phi_{R},\varphi)\xt\dt.
\end{align}
If the integrals
\begin{gather*}
    I^1_R=\int_0^{\too}\intr|u||\varphi\Lk\psi_{R}|,\quad I^2_R=\int_0^{\too}\intr|u|B^1_R,\quad I^3_R=\int_0^{\too}\intr|u|B^2_R,\quad\text{where}\\
    B^{1}_{R}(x,t):=\ints |\psi_{R}(x)-\psi_{R}(x-y)||\varphi\xt-\varphi(x-y,t)|K(y)\dy,\\
    B^{2}_{R}\xt:=\intb |\psi_{R}(x)-\psi_{R}(x-y)||\varphi\xt-\varphi(x-y,t)|K(y)\dy,
\end{gather*}
go to zero as $R\to\infty$, the desired result \eqref{a.22} will follow immediately from \eqref{eq:with.R}.

As for $I_R^1$, since $\psi_R(x)=\psi(x/R)$, then  $\Lambda \psi_{R}(x,y)\to 0$ for all $x,y\in\mathbb{R}^d$ as $R\to \infty$. Moreover, there exists $C>0$ depending on the uniform norm  of $\psi$ and $D^{2}\psi$, but not on $R$,  such that
\begin{equation}\label{z16}	
    |\Lambda\psi_{R}(x,\cdot)|K\leq C (1\wedge|\cdot|^{2})K\in L^{1}(\Rn)\quad\textrm{ for all }x\in\Rn,\ R\geq R_{0},
\end{equation}
since $K$ is a L\'evy kernel. Then, by the DCT, $\Lk\psi_{R}\to 0$ as $R\to0$. On the other hand, by \eqref{a.4},  the semigroup property~\eqref{eq:semigroup.property} and \eqref{Atotal2}, we have
\begin{equation}\label{a.23}
	\begin{aligned}
		|\varphi\xt|&\leq C_{\tilde{\theta}}\int_{0}^{\too-t}\intr P_{\too-t-s}(x-y)P_{1}(y)\dy{\rm d}s\\
		&=C_{\tilde{\theta}}\int_{0}^{\too-t}P_{1+\too-t-s}(x)\dx\leq  C_{\tilde{\theta}, T}\too P_{1}(x).
	\end{aligned}	
\end{equation}
Hence,  since $u\in L^{1}_{\rm loc}([0,T);\lp)$, we may use the DCT to conclude that $I^{1}_{R}\to 0$ as $R\to\infty$.
	
Let us proceed with the analysis of $I_R^2$. By the MVT,
\begin{equation*}
    |\psi_{R}(x)-\psi_{R}(x-y)|\leq\frac{1}{R}\|\nabla\psi\|_{\infty}|y|,\quad |\varphi\xt-\varphi(x-y,t)|\leq|\nabla\varphi(z,t)||y|,
\end{equation*}
where $z\in B_{1}(x)$. Since
\begin{equation*}
	\partial_{i}\varphi(\cdot,t)=\int_{0}^{\too-t}\intr P_{\too-t-s}(\cdot-y)\partial_{i}\theta(y,t)\dy\,{\rm d}s\quad\textrm{for all }i=1,\dots,d,
\end{equation*}
arguing as when we obtained~\eqref{a.23}, we get $|\partial_{i}\varphi|\le C_{\partial_{i}\tilde{\theta},T}P_{1}$. Hence,  $|\nabla\varphi(z,t)|\leq C_{\tilde{\theta},T}P_{1}(z)$ for some $z\in B_{1}(x)$. Using now~\eqref{eq:Pz.Px} we conclude that
\begin{equation*}
	|\varphi\xt-\varphi(x-y,t)|\leq C_{\tilde{\theta},T}|y|P_{1}(x).
\end{equation*}
Thus, since $K$ is a L\'evy kernel,
\begin{equation*}\label{g1}	
    B^{1}_{R}(x,t) \leq \frac{1}{R}CP_{1}(x)\|\nabla \psi\|_{\infty}\ints |y|^{2}K(y)\dy\leq \frac{1}{R}C P_{1}(x)\quad \textrm{for all }\xt\in\Rn\times[0, \too)
\end{equation*}
for some constant $C>0$. Therefore, since $u\in L^{1}([0,T);\lp)$, by the DCT, $I^{2}_{R}\to0$ as $R\to\infty$.
	
Finally, we deal with $I^{3}_{R}$. Since $\|\psi_{R}\|_\infty=1$, $\|\varphi\|_\infty\le \|\tilde\theta\|_\infty$ and $K$ is a L\'evy kernel, then
\begin{equation*}
	|\psi_{R}(x)-\psi_{R}(x-y)||\varphi\xt-\varphi(x-y,t)|K(y)\leq 4\|\tilde{\theta}\|_{\infty}K(y)\in L^{1}(\RN\setminus B_{1}(0)).
\end{equation*}
Besides, $|\psi_{R}(x)-\psi_{R}(x-y)|\to 0$ a.e.~as $R\to\infty$, whence, by the DCT,  $B^{2}_{R}(x,t)\to 0$  for all $x\in\Rn$ and $0\leq t\leq \too$ as $R\to\infty$. On the other hand, $\|\psi_{R}\|_{\infty}=1$,
\begin{equation*}
	B^{2}_{R}\xt\leq2|\varphi\xt|\intb K(y)\dy+2\intb|\varphi(x-y,t)|K(y)\dy.
\end{equation*}
The first term on the right-hand side is easily seen to be less or equal than $CP_1(x)$,  using again that $K$ is a L\'evy kernel and~\eqref{a.23}.
Besides, by \eqref{a.23}, \eqref{Atotal}, the semigroup property~\eqref{eq:semigroup.property}, and~\eqref{Atotal2},
\begin{align*}\label{z14}	
    \intb|\varphi(x-y,t)|K(y)\dy&\leq C_{\tilde{\theta}}\intb P_{1}(x-y)K(y)\dy\leq C_{\tilde{\theta}}\intb P_{1}(x-y)P_{1}(y)\dy\\
    &\leq C_{\tilde{\theta}}\intr P_{1}(x-y)P_{1}(y)\dy =C_{\tilde{\theta}}P_{2}(x)\leq C_{\tilde{\theta},T}P_{1}(x).
\end{align*}
Thus, $B^2_R(x,t)\le C_{\tilde{\theta},T}P_{1}(x)$. Therefore, since $u\in L^{1}_{\rm loc}([0,T);\lp)$, we may use the DCT and we get $I^{3}_{R}\to 0$ as $R\to\infty$.
\end{proof}
As an immediate consequence of Theorem~\ref{UniquenessSoria} and Proposition~\ref{equivDefvwProp}, we obtain uniqueness for very weak solutions within the class $\li$.

\begin{corollary}[Uniqueness for  very weak solutions within $\li$] \label{UniquenessAcotadas}
    Let $\mathcal{L}=\mathcal{L}_K$ with $K$ satisfying~\eqref{A0}--\eqref{A2}. There is  at most one very weak solution of \eqref{a.nonlocalheatequation} belonging to $\li$ with a given initial trace.
\end{corollary}

Now we will obtain uniqueness in the class of nonnegative very weak solutions. If $u$ is a very weak solution, then $u\in L^{1}_{\rm loc}((0,T);\lp)$; see~Remark~\ref{remark3}. Hence, if we are able to prove that nonnegative solutions in that class belong to  $u\in L^{\infty}((0,T);\lp)$,  uniqueness will be inherited from Corollary~\ref{UniquenessAcotadas}. In order to prove this smoothing result we follow  ideas from \cite{Matteo}.

Let us start with a heuristic proof. Let $u$ be a nonnegative very weak solution to~\eqref{a.nonlocalheatequation} and $f(t):=\intr u\xt P_1(x)\dx$.  Then, formally ($f$ is not necessarily smooth; it may even be infinite for some times, so we will have to take some extra care when performing the rigorous proof),
\begin{equation*}\label{v3.6}	
    |f'(t)|=\Big|\intr \partial_t u\xt P_1(x)\dx\Big|=\Big|\intr \mathcal{L}u\xt P_1(x)\dx\Big|=\Big|\intr u\xt\Lk P_1(x)\dx\Big|.
\end{equation*}
Therefore, if $|\Lk P_{1}|\le c P_{1}$ for some constant $c>0$, then $|f'(t)|\leq cf(t)$, since $u$ is nonnegative, which yields, upon integration,
\begin{equation}\label{z20}
     e^{-c|t-\tau|}f(t)\leq f(\tau)\leq e^{c|t-\tau|}f(t)\quad\text{for almost every }0<t,\tau<T.
\end{equation}
By Remark~\ref{remark3}, $u\in L^{1}_{\rm loc}((0,T);\lp)$; hence $f(\too)<\infty$ for some time $\too\in (0,T)$. Therefore, by~\eqref{z20}, $u\in\li$.

Besides the already mentioned regularity issues, we find a second difficulty when trying to follow the above programme: up to now, the inequality $|\Lk P_{1}|\le c P_{1}$ is only known to hold when $\Lk$ is the fractional Laplacian; see Proposition~\ref{pprop1} (iv).  To overcome this hitch, the strategy  is to construct a positive $\phi\in C^{2}(\Rn)$ comparable to $P_1$ such that $|\Lk \phi|\leq C\phi$ for some $C>0$  and use it instead of~$P_1$, multiplied by a convenient smooth cutoff function of time, as a test function  in the definition of very weak solution. This will yield an integrated version of the ODE inequality $|g'|\le Cg$, with $g(t):=\intr u(x,t)\phi(x)\dx$, that will imply that
\begin{equation}\label{va.34}
    e^{-C|t-\tau|}\|u(\cdot,\tau)\|_{L^{1}(\RN,\,\phi\dx)}\leq\|u(\cdot,t)\|_{L^{1}(\RN,\,\phi\dx)}\leq e^{C|t-\tau|}\|u(\cdot,\tau)\|_{L^{1}(\RN,\,\phi\dx)}
\end{equation}
for almost every $0<t,\tau<T$. Since $P_{1}$ and $\phi$ are comparable, this is equivalent to~\eqref{z20}.

The first step is then to construct the desired function $\phi$. It is here that we need assumption~\eqref{A3}.

\begin{lemma}\label{lemmaGettingPhi}
    Let $\mathcal{L}=\mathcal{L}_K$ with $K$ satisfying~\eqref{A0}--\eqref{A3}, and $B$ the bilinear form defined in~\eqref{51.2}. There exist $\phi\in C^{2}(\Rn)$ and a constant $c>1$ such that
	\begin{gather}
		\label{a.46} c^{-1}P_{1}\leq \phi\leq cP_{1},\\
        \label{Lphi}|\Lk\phi|\leq c\phi,\qquad|B(\psi,\phi)|\leq c\max\{\|\psi\|_{\infty},\|\nabla\psi\|_{\infty}\}\phi\quad\text{for all } \psi\in C^{\infty}_{\textup{c}}(\Rn).
	\end{gather}
\end{lemma}

\begin{proof} Let $\tilde{\varphi}$  as  in \eqref{A1}--\eqref{A3}. Let $\phi\in C^2(\Rn)$ such that $\phi>0$ and
\begin{equation}\label{a.33}	
    \phi(x) =
    \begin{cases}
		1,&|x|\leq1,\\
		\frac{1}{|x|^d\tilde{\varphi}(|x|)},&|x|\geq 2.
	\end{cases}
\end{equation}
Note that $\phi$ satisfies \eqref{a.46}, due to the way in which it is defined.

Let us start by proving that $\phi$ satisfies
\begin{equation}\label{a.47}
	\ints|\Lambda\phi(x,y)|K(y)\dy\leq c\phi(x),\quad \ints|y||\phi(x)-\phi(x-y)|K(y)\dy\leq c\phi(x),\quad x\in\Rn.
\end{equation}
A direct calculation shows that there is a constant $C>0$ such that
\begin{equation*}
	\begin{array}{l}
        \displaystyle|\nabla \phi(x)|\leq C\phi (x)\big(1+\big|\frac{\tilde{\varphi}'(x)}{\tilde{\varphi}(x)}\big|\big),\\
        \displaystyle|D^{2}\phi(x)|\leq C\phi (x)\big(1+\big|\frac{\tilde{\varphi}'(x)}{\tilde{\varphi}(x)}\big|+\big(\frac{\tilde{\varphi}'(x)}{\tilde{\varphi}(x)}\big)^{2}+ \big|\frac{\tilde{\varphi}''(x)}{\tilde{\varphi}(x)}\big| \big),
    \end{array}
	\qquad\textrm{for all }|x|\geq 3,
\end{equation*}
whence, thanks to~\eqref{A3}, there is a constant $C>0$ such that
\begin{equation}\label{v3.26}	
    |\nabla \phi(x)|\leq C\phi (x),\quad |D^{2}\phi(x)|\leq C\phi (x)\quad\textrm{for all }|x|\geq 3.
\end{equation}
Then, applying the MVT once or twice as appropriate, if $|x|\geq 3$ and $|y|\leq 1$,
\begin{gather*}
    |\Lambda\phi(x,y)|K(y)\leq |D^{2}\phi(z)||y|^{2}K(y)\leq C\phi(z)|y|^{2}K(y),\\
    |y||\phi(x)-\phi(x-y)|K(y)\leq|\nabla\phi(z)||y|^{2}K(y)\leq C\phi(z)|y|^{2}K(y),
\end{gather*}
where $z=z(x,y)\in B_{1}(x)$. {Moreover, since $P_{1}$ is {slowly changing} (see Proposition~\ref{Pslowly}) and  $\phi$  satisfies~\eqref{a.46}, we have that $\phi$ is  also {slowly changing}.} Then, since $z\in B_1(x)$,
\begin{equation}\label{z26}
    \begin{array}{l}
        |\Lambda\phi(x,y)|K(y)\leq C\phi(x)|y|^{2}K(y),\\[6pt]
        |y||\phi(x)-\phi(x-y)|K(y)\leq C\phi(x)|y|^{2}K(y)
    \end{array}
    \quad\text{if } |x|\geq 3,\ |y|\leq 1,
\end{equation}
and~\eqref{a.47} follows for $|x|\geq 3$, since $K$ is a L\'evy kernel. Besides, since $\phi>0$, applying the MVT,
\begin{equation}\label{z25}
    \begin{array}{l}
        |\Lambda\phi(x,y)|K(y)\leq \|D^{2}\phi\|_{L^{\infty}(B_{4}(0))}|y|^{2}K(y) \leq C\phi(x)|y|^{2}K(y),\\[6pt] |y||\phi(x)-\phi(x-y)|K(y)\leq\|\nabla\phi\|_{L^{\infty}(B_{4}(0))}|y|^{2}K(y) \leq C\phi(x)|y|^{2}K(y),	
    \end{array}
    \qquad\text{if }|x|\leq 3,
\end{equation}
whence~\eqref{a.47}  also holds for all $|x|\leq 3$, because of $K$ being a L\'evy kernel.

On the other hand, since $\phi$ is positive, $K$ is a  symmetric L\'evy kernel,  \eqref{Atotal}, the semigroup property~\eqref{eq:semigroup.property}, ~\eqref{Atotal2} and ~\eqref{a.46}, there is $C>1$ such that
\begin{align*}
    \max&\Big\{	\intb |\Lambda \phi(x,y)|K(y)\dy,\intb |\phi(x)-\phi(x-y)|K(y)\dy\Big\}\\
    &\leq \phi(x)\intb K(y)\dy+\intb\phi(x-y)K(y)\dy\\
    &\leq C\Big(\phi(x)+\intb P_{1}(x-y)P_{1}(y)\dy\Big)\leq  C(\phi(x)+P_{2}(x))\leq C(\phi(x)+P_{1}(x))\leq C\phi(x),
\end{align*}
 for all $x\in\Rn$. Then,
\begin{equation}\label{z.17}
	\intb|\Lambda\phi(x,y)|K(y)\dy\leq c\phi(x),\quad \intb|\phi(x)-\phi(x-y)|K(y)\dy\leq c\phi(x),\quad x\in\Rn.
\end{equation}
Thus, combining \eqref{a.47} and \eqref{z.17}, we obtain $	|\Lk\phi|\leq c\phi$, and by the MVT, for a given $\psi\in C^{\infty}_{\textup{c}}(\Rn)$,
\begin{align*}
	|B(\psi,\phi)(x)|&\leq \|\nabla\psi\|_{\infty}\ints|y||\phi(x)-\phi(x-y)|K(y)\dy\\
    &\quad+2\|\psi\|_{\infty}\intb|\phi(x)-\phi(x-y)|K(y)\dy\\
    &\le c\max\{\|\psi\|_{\infty},\|\nabla\psi\|_{\infty}\}\phi(x),\quad x\in\Rn.\qedhere
\end{align*}
\end{proof}

\begin{remark}
    The key points to obtain~\eqref{Lphi} are inequalities~\eqref{a.47} and~\eqref{z.17}. The latter is guaranteed if $\phi$ has a tail comparable to $P_1$. The former  may be understood as a weak requirement of boundedness for the Hessian and the gradient of $\phi$  respectively, in relationship with the L\'evy kernel~$K$. It follows from~\eqref{A3}. We could have used the existence of a function $\phi$ satisfying~\eqref{a.47}--\eqref{z.17} as an assumption instead of~\eqref{A3}. However, we prefer~\eqref{A3}, since it is more easily verified.
\end{remark}

We can proceed now with a rigorous proof.

\begin{lemma} \label{lemmatrace}
    Let $\mathcal{L}=\mathcal{L}_K$ with $K$ satisfying~\eqref{A0}--\eqref{A3}.

    \noindent\textup{(i)} If $u$ is a nonnegative very weak solution to~\eqref{a.nonlocalheatequation}, there is a constant $c>0$ such that
	\begin{equation}\label{lematecnicodesigualdad}
        e^{-c|t-\tau|}\|u(\cdot,\tau)\|_{L^{1}_{P_1}}\leq\|u(\cdot,t)\|_{L^{1}_{P_1}}\leq e^{c|t-\tau|}\|u(\cdot,\tau)\|_{L^{1}_{P_1}}\quad\text{for a.e. }t,\tau\in(0,T).
	\end{equation}

    \noindent\textup{(ii)} If $u$ is a nonnegative very weak solution to~\eqref{a.nonlocalheatequation}, then $u\in\li$.
\end{lemma}

\begin{proof}
(i) As we have just explained, it is enough to prove~\eqref{va.34}.

Since $u$ is a very weak solution, if $g(t):=\intr u\xt\phi(x)\dx$ were well defined for all $t\in (0,T)$ and $\phi\mathcal{X}_{[t_0,t_1]}$ were an admissible test function, we would have
\begin{gather*}
	\intr u(x,t_0)\phi(x)\dx-	\intr u(x,t_1)\phi(x)\dx=\int_{0}^{T}\intr u\xt\partial_{t}(\phi\mathcal{X}_{[t_0,t_1]})\xt\dx\dt\\
	=\int_0^T\intr u\xt\Lk(\phi\mathcal{X}_{[t_0,t_1]})\xt\dx\dt=\int_{t_0}^{t_1}\intr u\xt\Lk\phi(x)\dx\dt
\end{gather*}
 for all $0<t_0<t_1<T$.  Then,  by the first inequality in~\eqref{Lphi}, since $u$ is nonnegative, there would be $C'>0$ such that
\begin{equation}
	\label{z21}	\Big|\intr u(x,t_0)\phi(x)\dx-	\intr u(x,t_1)\phi(x)\dx\Big|\leq C'\int_{t_0}^{t_1}\intr u\xt\phi(x)\dx\dt, \quad 0<t_0<t_1<T,
\end{equation}
whence, using a Gronwall type argument,  $e^{-C'|t_0-t_1|}g(t_0)\leq g(t_1)\leq e^{C'|t_0-t_1|}g(t_0)$,  $t_0,t_1\in (0,T)$.

However, $\phi\mathcal{X}_{[t_0,t_1]}$ is not an admissible test function, and we do not know if $g$ is well defined for all $t\in (0,T)$, only   a.e.~in $(0,T)$, since $u\in\luloc$  (see Remark~\ref{remark3}), and $P_1$ and $\phi$ are comparable; see~\eqref{a.46}. Therefore, we will need to construct suitable approximations of $\phi$ and $\mathcal{X}_{[t_0,t_1]}$ such that the product is a suitable test function that in the limit yields~\eqref{z21} for a.e.\,$t_0,t_1\in (0,T)$.
	
We approximate $\phi$ by $\phi\psi_n$, $n\in\mathbb{N}$ with $\psi_n$ a cutoff function of the form~\eqref{cutoffspace} with $R=n$.

As for $\mathcal{X}_{[t_0,t_1]}$, let $\eta:\mathbb{R}\to\mathbb{R}$ as in~\eqref{etaParaCutOff}. The family $\{\eta_{k}\}_{k\in\mathbb{N}}$ given by  \mbox{$\eta_{k}(t):=2^{k}\eta(2^{k}t)$} is a summability kernel. For $0<t_{0}<t_{1}<T$ and  \mbox{$k_{0}>-\max\{\ln(t_{1}-t_{0}),\ln t_{0}, \ln (T-t_{1})\}/\ln 2$}, let
\begin{equation*}
	\varphi_{k}^{t_{0},t_{1}}(t)=\int_{2^k(t_0-t)}^{2^k(t_1-t)}\eta(s)\,{\rm d}s,\quad k\geq k_{0},\; t>0.
\end{equation*}
Then, since $\eta$ is symmetric,
\begin{equation*}
    \mathcal{X}_{[t_{0}+\frac{1}{2^{k}}, t_{1}-\frac{1}{2^{k}}]}\leq \varphi^{t_{0},t_{1}}_{k}\leq\mathcal{X}_{[t_{0}-\frac{1}{2^{k}}, t_{1}+\frac{1}{2^{k}}]},\quad (\varphi^{t_{0},t_{1}}_{k})'(t)=\eta_{k}(t-t_{0})-\eta_{k}(t-t_{1})
\end{equation*}
for all $k\geq k_{0}$ and $t>0$. In particular, $(\varphi^{t_{0},t_{1}}_{k})'\to \delta_{t_0}-\delta_{t_1}$ as $k\to \infty$ in the sense of distributions.
	
Let $g_n(t):=\intr  u(x,t)\psi_n(x)\phi(x)\dx$, $n\in\mathbb{N}$. For notational convenience, let $g_{\infty}(t)=g(t)$ for \mbox{$t>0$}.
Since $u$ is a very weak solution, then  $u\in\luloc$ (see Remark~\ref{remark3}). On the other hand,  as $\phi$ and $P_{1}$ are comparable, see~\eqref{a.46}, then $g_\infty\in L^{1}_{\textrm{loc}}((0,T))$ and  $g_n\in L^{1}_{\textrm{loc}}((0,T))$ for all $n\in\mathbb{N}$, whence, by Lebesgue's Differentiation Theorem, $|(0,T)\setminus L(g_\infty;\{\eta_k\}_{k\in\mathbb{N}})|=0$ and $|(0,T)\setminus L(g_n;\{\eta_k\}_{k\in\mathbb{N}})|=0$ for all $n\in\mathbb{N}$. As a straightforward consequence,
\begin{equation}\label{v3.2}	
    |(0,T)\setminus \bigcap\limits_{n=1}^{\infty}L(g_{n};\{\eta_k\}_{k\in\mathbb{N}})|=0.
\end{equation}
Finally, consider two times $0<\tau_{0}<\tau_{1}<T$ such that $\tau_{0},\tau_{1}\in \cap_{n=1}^{\infty}L(g_{n};\{\eta_k\}_{k\in\mathbb{N}})$.
	
Since $u$ is a very weak solution and $\varphi^{\tau_{0},\tau_{1}}_{k} \psi _{n}\phi\in C_{\textup{c}}^{\infty}(\Rn\times(0,T))$ for all $k>k_{0}$ and $n\in\mathbb{N}$,
\begin{equation}\label{v3.1}	
    \int_0^T(\varphi^{\tau_{0},\tau_{1}}_{k})'(t)\intr \psi_{n}(x)\phi(x)u(x,t)\dx\dt
	=\int_0^T\varphi^{\tau_{0},\tau_{1}}_{k}(t)\intr u(x,t)\Lk(\psi_{n}\phi)(x)\dx\dt.
\end{equation}
Besides, thanks to~\eqref{z16}, there is $C>0$ depending on  $\|\psi\|_{\infty}$ and $\|D^2\psi\|_{\infty}$ such that
$|\Lk\psi_{n}|\leq C$ for all $n\in\mathbb{N}$.  Moreover, by~\eqref{Lphi} there exists $C>$ such that $|\Lk\phi|\leq C\phi$ and $|B(\psi_{n},\phi)|\leq C\phi$ for all $n\in\mathbb{N}$. Therefore, by Leibniz's formula~\eqref{leibniz}, there is $C_{2}>0$ such that
\begin{equation*}\label{claimLemaTecnico}	
    |\Lk(\psi_{n}\phi)(x)|\leq C_{2}\phi(x)\quad\textrm{for all }x\in \RN\textrm{ and }n\in\mathbb{N}.
\end{equation*}
Then, since $u$ and  $\varphi^{\tau_0,\tau_1}_k$ are nonnegative,~\eqref{v3.1} yields
\begin{equation*}
    \big|\int_0^T(\varphi^{\tau_0,\tau_1}_k)'(t)g_n(t)\dt\big|\leq C_2\int_0^T\varphi^{\tau_0,\tau_1}_k(t)g_\infty(t)\dt\quad\textrm{for all } k>k_0\text{ and }n\in\mathbb{N}.
\end{equation*}
Then, since $\tau_0,\tau_1\in \cap_{n=1}^{\infty}L(g_{n};\{\eta_k\}_{k\in\mathbb{N}})$, taking $k\to\infty$,
\begin{equation*}\label{v3.3}
    |g_{n}(\tau_{1})-g_{n}(\tau_0)|\leq C_{2}\int_{\tau_{0}}^{\tau_1}g_{\infty}(t)\dt \quad\textrm{for all }n\in\mathbb{N}.
\end{equation*}
Since $\psi_{n}$ converges monotonically to 1 as $n\to\infty$ and $u$ is nonnegative, the Monotone Convergence Theorem (MCT from now on) yields $g_{n}(\tau_i)\to g_{\infty}(\tau_i)$ as $n\to\infty$ for $i=0,1$. Therefore,
\begin{equation}\label{eq:Gronwall.absolute.value}
	|g_{\infty}(\tau_1)- g_{\infty}(\tau_0)|\le C_{2}\int_{\tau_0}^{\tau_1} g_{\infty}(t)\dt.
\end{equation}
Since $C_{2}$  does not depend on $\tau_{0}$ and $\tau_{1}$, we have
\begin{equation}\label{eq:Gronwall}
    g_{\infty}(t)\leq g_{\infty}(\tau_{1})+C_{2}\int_{t}^{\tau_1} g_{\infty}(s)\,{\rm d}s \quad\textrm{for all }t\in [\tau_0,\tau_1)\cap\bigcap\limits_{n=1}^{\infty}L(g_{n};\{\eta_k\}_{k\in\mathbb{N}}).
\end{equation}
We will prove now, using an induction argument, a Gronwall-type lemma: if $g_\infty$ satisfies the integral estimate~\eqref{eq:Gronwall}, it satisfies a pointwise estimate.

For $k=1$, by~\eqref{eq:Gronwall} and~\eqref{v3.2}, we have
\begin{equation*}
	g_\infty(t)\leq g_\infty(\tau_1)+C_2\int_t^{\tau_1}\Big(g_\infty(\tau_1)+C_2\int_s^{\tau_1} g_\infty(r)\dr\Big)\,{\rm d}s.
\end{equation*}
Therefore, since
\begin{equation*}
    \int_t^{\tau_1}\int_s^{\tau_1}g_\infty(r)\dr\,{\rm d}s=\int_t^{\tau_1}\int^r_tg_{\infty}(r)\,{\rm d}s\dr=\int_t^{\tau_1}g_\infty(r)(r-t)\dr,
\end{equation*}
we have that
\begin{equation*}
	g_{\infty}(t)\leq g_{\infty}(\tau_1)(1+C_{2}(\tau_{1}-t))+C_{2}^2\int_{t}^{\tau_1}g_{\infty}(s)(s-t)\,{\rm d}s
\end{equation*}
for all $t\in\cap_{n=1}^{\infty}L(g_{n})$ such that $\tau_0\leq t<\tau_1$. Iterating, we obtain that
\begin{equation*}
    g_{\infty}(t)\leq g_{\infty}(\tau_{1})\Big(\sum_{j=0}^{k}\frac{C_{2}^j(\tau_{1}-t)^j}{j!}\Big)+ C_{2}^{k+1}\int_{t}^{\tau_1}g_{\infty}(s)\frac{(s-t)^{k}}{k!}\,{\rm d}s
\end{equation*}
for all $t\in [\tau_0,\tau_1)\cap\bigcap\limits_{n=1}^{\infty}L(g_{n})$, $k\in\mathbb{N}$. Note that
\begin{equation*}
    \int_t^{\tau_1}g_\infty(s)\frac{(s-t)^k}{k!}\,{\rm d}s\leq\frac{(\tau_1-t)^k}{k!}\int_t^{\tau_1}g_\infty(s)\,{\rm d}s\leq c \frac{T^k}{k!}\int_{\tau_0}^{\tau_1}\intr u\xt P_{1}(x)\dx
\end{equation*}
where $c$ is  the constant in~\eqref{a.46}. Therefore,
\begin{equation*}
    g_{\infty}(t)\leq g_{\infty}(\tau_{1})\Big(\sum_{j=0}^{k}\frac{C_{2}^j(\tau_{1}-t)^j}{j!}\Big)
    +c\frac{C_{2}^{k+1}T^{k}}{k!}\int^{\tau_1}_{\tau_0}\intr u\xt P_{1}(x)\dx,
\end{equation*}
for all $t\in \cap_{n=1}^{\infty}L(g_{n})$ such that $t\in [\tau_{0},\tau_{1})$ and $ k\in\mathbb{N}$.  By Remark~\ref{remark3}, $u\in L^{1}((\tau_0,\tau_1);L^{1}_{P_{1}})$ and therefore taking $k\to\infty$ in the above inequality, we obtain
\begin{equation*}
	g_{\infty}(t)\leq g_{\infty}(\tau_1) e^{C_{2}(\tau_{1}-t)}\quad\textrm{a.e. }t\in (\tau_{0},\tau_{1})
\end{equation*}
due to \eqref{v3.2}. Since $C_{2}$ does not depend on $\tau_{0}$ or $\tau_{1}$, the result can be extrapolated and we obtain
\begin{equation}\label{v3.4}	
    g_{\infty}(t)\leq g_{\infty}(\tau) e^{C_{2}(\tau-t)}\quad\textrm{a.e. }t,\tau \in (0,T),\ t\leq\tau.
\end{equation}

Repeating the induction argument, starting now from
\begin{equation*}
    g_{\infty}(\tau_1)\leq g_{\infty}(t)+C_{2}\int_{t}^{\tau_1} g_{\infty}(s)\,{\rm d}s \quad\textrm{for all }t\in [\tau_0,\tau_1)\cap\bigcap\limits_{n=1}^{\infty}L(g_{n};\{\eta_k\}_{k\in\mathbb{N}}),
\end{equation*}
which is also implied by~\eqref{eq:Gronwall.absolute.value}, we arrive at
\begin{equation*}\label{v3.5}	
    g_{\infty}(\tau)\leq g_{\infty}(t) e^{C_{2}(\tau-t)}\quad\textrm{a.e. }t,\tau \in (0,T),\ t\leq\tau,
\end{equation*}
which combined with~\eqref{v3.4} gives~\eqref{va.34}.

\noindent (ii) Since $u$ is a very weak solution, $u\in\luloc$; see Remark~\ref{remark3}. Hence, $\|u(\cdot,\tau)\|_{\lp}<\infty$ for some $\tau\in(0,T)$. The result now follows from the second inequality in~\eqref{lematecnicodesigualdad}.
\end{proof}

The combination of Corollary~\ref{UniquenessAcotadas} and Lemma~\ref{lemmatrace} (ii) yields the desired uniqueness result for nonnegative very weak solutions.

\begin{corollary}[Uniqueness for nonnegative very weak solutions]\label{UniquenessNonnegative}
    Let $\mathcal{L}=\mathcal{L}_K$ with $K$ satisfying~\eqref{A0}--\eqref{A3}. There is at most one nonnegative very weak solution of~\eqref{a.nonlocalheatequation} with a given initial trace.
\end{corollary}

We finish this paragraph proving that nonnegative solutions have some regularity in time.

\begin{proposition}\label{prop:regularity.in.time}
    Let $\mathcal{L}=\mathcal{L}_K$ with $K$ satisfying~\eqref{A0}--\eqref{A3}, and let $u$ be a nonnegative very weak solution to~\eqref{a.nonlocalheatequation}.

    \noindent\textup{(i)} Let $\phi\in C^2(\mathbb{R}^d)$ satisfying~\eqref{a.46}--\eqref{Lphi}. There is a constant $C>0$ such that
    \begin{equation}\label{eq:continuity.phi}
        \Big|\intr u(x,\tau_1)\phi(x)\dx-\intr u(x,\tau_2)\phi(x)\dx\Big|\leq C\|u\|_{\li}|\tau_1-\tau_2|
    \end{equation}
for a.e. $\tau_1,\tau_2\in (0,T).$

    \noindent\textup{(ii)} Given $\psi\in C^2_{\textup{c}}(\mathbb{R}^d)$, there is a constant $C_\psi>0$ such that
    \begin{equation}\label{eq:continuity.psi}
        \Big|\intr u(x,\tau_1)\psi(x)\dx-\intr u(x,\tau_2)\psi(x)\dx\Big|\leq C_\psi\|u\|_{\li}|\tau_1-\tau_2|
    \end{equation}
    for a.e. $\tau_1,\tau_2\in (0,T).$
\end{proposition}

\begin{proof}
(i) We have seen in the proof of Lemma~\ref{lemmatrace} that there is a constant $C>0$ such that
\begin{equation*}
    \Big|\intr u(x,\tau_1)\phi(x)\dx-\intr u(x,\tau_2)\phi(x)\dx\Big|\leq C\int_{\tau_1\wedge\tau_2}^{\tau_1\vee\tau_2}u(x,t)P_1(x)\dx \quad\text{for a.e. }\tau_1,\tau_2\in (0,T);
\end{equation*}
see~\eqref{eq:Gronwall.absolute.value} and~\eqref{a.46}. The result now follows from Lemma~\ref{lemmatrace} (ii).

\noindent (ii) By Lemma~\ref{a.17}, there is a constant $C_\psi>0$ such that $|\Lk\psi|\leq C_{\psi}P_{1}$. Hence, we can proceed as in the proof of Lemma~\ref{lemmatrace}, with $\phi$ replaced by $\psi$, to obtain that
\begin{equation*}
    \Big|\intr u(x,\tau_1)\psi(x)\dx-\intr u(x,\tau_2)\psi(x)\dx\Big|\leq C_\psi\int_{\tau_1\wedge\tau_2}^{\tau_1\vee\tau_2}u(x,t)P_1(x)\dx \quad\text{for a.e. }\tau_1,\tau_2\in (0,T),
\end{equation*}
whence the result, thanks again to Lemma~\ref{lemmatrace} (ii).
\end{proof}

If $\Lk=\Ls$, the smooth function $\phi=P^{\alpha}_1$,  where $P^{\alpha}_1$ denotes the heat kernel associated to~$\Ls$ (see Subsection~\ref{2.3}),  satisfies (trivially)~\eqref{a.46}. Next lemma implies that it also satisfies~\eqref{Lphi}. Therefore, nonnegative very weak solutions are Lipschitz as functions in time with values in $L^1 _{P_1}$.

\begin{remark}
    We already know by Proposition~\ref{pprop1}~(iv) that the first inequality in~\eqref{Lphi} holds. However, we give an alternative proof that might be useful to deal with other operators.
\end{remark}

\begin{lemma} \label{Remark-P-satisface}
    Let $\Lk=\Ls$, $\alpha\in(0,2)$, and $0<\varepsilon<T$. Then, there is $C>0$ such that
	\begin{equation}\label{LPalpha}	
        |\Ls P^{\alpha}_{t}|\leq\intr|\Lambda P^{\alpha}(\cdot,y,t)|K(y)\dy\leq CP^{\alpha}_{1},\quad
        |B(\psi,P^{\alpha}_{t})|\leq C\max\{\|\psi\|_{\infty},\|\nabla\psi\|_{\infty}\}P^{\alpha}_{1}
	\end{equation}
    for all $t\in(\varepsilon,T)$ where $B$ is the bilinear form given by~\eqref{51.2} with $K=K_{\alpha}$.
\end{lemma}

\begin{proof}
The key to get~\eqref{LPalpha} is to prove that there is $c>1$ such that
\begin{align}
    \label{z.18}
	&\ints|\Lambda P^\alpha_t(x,y)|K(y)\dy\leq cP^\alpha_1(x),\quad \ints|y||P^\alpha_t(x)-P^\alpha_t(x-y)|K(y)\dy\leq cP^\alpha_1(x),\\
	\label{z.19}
	&\intb|\Lambda P^\alpha_t(x,y)|K(y)\dy\leq cP^\alpha_1(x),\quad \intb|P^\alpha_t(x)-P^\alpha_t(x-y)|K(y)\dy\leq cP^\alpha_1(x)
\end{align}
for all $x\in\Rn$; see the proof of Lemma~\ref{lemmaGettingPhi}. The estimates~\eqref{z.19} are obtained easily, following closely the proof of Lemma~\ref{lemmaGettingPhi}, using that $K_{\alpha}$ is a symmetric L\'evy kernel,  $K_{\alpha}$ and $P^{\alpha}_1$ have comparable tails (see~\eqref{Atotal}), the semigroup property~\eqref{eq:semigroup.property} and~\eqref{Atotal2}.
	
On the other hand, by~Proposition~\ref{pprop1} (vi), using Taylor's expansion, for  $x\in \Rn$ and $|y|\leq 1$,
\begin{gather*}
    |\Lambda P^{\alpha}(x,y,t)|\leq C|D^{2}P^{\alpha}_t(z)||y|^{2}\leq CP^{\alpha}_1(z)|y|^2,\\
    |y||P^{\alpha}_{t}(x)-P^{\alpha}_{t}(x-y)|\leq|\nabla P^{\alpha}_t(z')||y|^{2}\leq CP^{\alpha}_1(z')|y|^2
\end{gather*}
where  $z$ and $z'$ depends on $x$ and $y$ and $z,z'\in B_{1}(x)$. Since $P^{\alpha}_1(\cdot)$ is {slowly changing}, Proposition~\ref{Pslowly},
\begin{equation*}
    |\Lambda P^{\alpha}(x,y,t)|\leq CP^{\alpha}_1(x)|y|^2,\quad|y||P^{\alpha}_1(x)-P^{\alpha}_1(x-y)|\leq CP^{\alpha}_1(x)|y|^2.
\end{equation*}
Therefore, since $K_{\alpha}$ is a L\'evy kernel, we obtain \eqref{z.18}.
\end{proof}

\begin{corollary}
    Let $\Lk=\Ls$, $\alpha\in(0,2)$ and let $u$ be a nonnegative very weak solution to~\eqref{a.nonlocalheatequation}. There is a constant $C>0$ such that
    \begin{equation*}
        \big|\|u(\cdot,\tau_1)\|_{L^1_{P_1}}-\|u(\cdot,\tau_2)\|_{L^1_{P_1}}\big|\leq C\|u\|_{\li}|\tau_1-\tau_2|\quad\text{for a.e. }\tau_1,\tau_2\in T.
    \end{equation*}
\end{corollary}

%%%%%%%%%%%%%%%%%%%%%%%%%%%%%%%%%%%%%%%%%%%%%%%%%%%%%%%
\subsubsection{Existence of an initial trace}

We now prove that every nonnegative solution has an initial trace. We also have to assume~\eqref{A3} for this purpose.

\begin{theorem}[Existence of  initial trace integrable against $P_{1}$]\label{theoremtraces}
    Let $\mathcal{L}=\mathcal{L}_K$ with $K$ satisfying~\eqref{A0}--\eqref{A3}. Let $u$ be a nonnegative very weak solution of~\eqref{a.nonlocalheatequation}. There is a nonnegative Radon measure $\mu_{0}$ such that~\eqref{a.1} holds; that is, $\mu_0$ is the initial trace of $u$.  Moreover, $\mu_0$ satisfies the integral growth condition~\eqref{eq:growth.initial.trace}.
\end{theorem}

\begin{proof}
Let $u$ be a nonnegative very weak solution of~\eqref{a.nonlocalheatequation}. By Remark \ref{remark3}, $u\in\luloc$, whence there is some $\tilde{T}\in(0,T)$ such that $\intr u(x,\tilde{T})P_1(x)\dx<\infty$. Besides, by Lemma \ref{lemmatrace}, there is some exceptional set $E\subset (0,T)$ with $|E|=0$ such that
\begin{equation*}\label{v3.7}	
    \intr u(x,t)P_1(x)\dx\leq e^{cT}\intr u(x,\tilde{T})P_1(x)\dx<\infty \quad\textrm{for all }t\in(0,\tilde{T})\setminus E.
\end{equation*}
On the other hand, for each $R>0$, let $c_R:=1/\min\{P_1(x):|x|\leq R\}$. This quantity is finite for each $R$, since $P_1$ is continuous and $P_1>0$. Hence
\begin{equation*}\label{a.38}
	\sup_{t\in(0,T)\setminus E}\int_{B_{R}(0)}u\xt\dx\leq c_{R}\intr u\xt P_1(x)\dx\leq c_R e^{cT}\intr u(x,\tilde{T})P_1(x)\dx=C_R.
\end{equation*}
This implies weak* compactness in the sense of Radon measures: there exists a sequence of times $\{t_n\}_{n\in\mathbb{N}}\subset(0,T)\setminus E$ with $t_{n}\to 0^{+}$ as $n\to \infty$ and a nonnegative Radon measure $\mu_0$ such that
\begin{equation*}\label{eq:trace.limit.subsequences}
	\lim\limits_{n\to\infty}\intr \psi(x)u(x,t_{n})\dx=\intr\psi\dmu_0\quad\textrm{for all }\psi\in C_{\textup{c}}(\Rn).
\end{equation*}
For a reference, see for instance \cite[Theorem 4.4]{Simon}. Moreover, for any $R>0$, $\mu_0(B_{R}(0))\leq C_{R} $.

However, we are not done yet, since the limit might depend on the sequence. Let us discard this possibility. Suppose that for some sequence $\{\tilde t_n\}_{n\in\mathbb{N}}\subset(0,T)\setminus E$  with $\tilde t_n\to 0^+$ as $n\to\infty$ we have
\begin{equation*}\label{a.43}	
    \lim\limits_{n\to\infty}\intr \psi(x)u(x,\tilde t_n)\dx=\intr\psi\,\textup{d}\tilde\mu_0(x)\quad\text{for all }\psi\in C_{\textup{c}}(\Rn).
\end{equation*}
If moreover $\psi\in C^\infty_{\textup{c}}(\Rn)$, taking $\tau_1=t_n$ and $\tau_2=\tilde t_k$ in \eqref{eq:continuity.psi}, we obtain
\begin{equation*}
	\big|\intr u(x,t_n)\psi(x)\dx-\intr u(x,\tilde t_k)\psi(x)\dx\big|\leq C_\psi\|u\|_{\li}|t_n-t_k|\quad\textrm{for all }n,k\in\mathbb{N}.
\end{equation*}
Letting first $n\to\infty$ and then $k\to\infty$, we conclude that 	
\begin{equation}\label{eq:equality.traces.as.distributions}
    \int \psi\dmu_0= \intr\psi\,\textup{d}\tilde\mu_0(x) \quad\text{for all }\psi\in C^\infty_{\textup{c}}(\mathbb{R}^d).
\end{equation}
By a density argument,~\eqref{eq:equality.traces.as.distributions} is also true for $\psi\in C_{\textup{c}}(\mathbb{R}^d)$, hence $\mu_0\equiv\tilde\mu_0$ as Radon measures on~$\RN$.

To conclude, we have to prove that $\mu_0$ satisfies~\eqref{eq:growth.initial.trace}. To this aim, we consider a family of cut-off functions $\{\psi_{R}\}_{R>0}$ with $\psi_R$ as in~\eqref{cutoffspace}. By Lemma~\ref{lemmatrace}, for all $t\in (0,T)\setminus E$ we have
\begin{equation*}
    \intr  u(x,t)\psi_{R}(x)P_{1}(x)\dx	\leq \intr u\xt P_{1}(x)\dx\leq  e^{cT}\intr u(x,\tilde{T})P_{1}(x)\dx= C_{\tilde{T},T}
\end{equation*}
Thus, letting $t\to 0^{+}$, we get $\displaystyle\intr  \psi_{R}P_{1}\dmu_{0}\leq C_{\tilde{T},T}$, whence, by the MCT,
$\displaystyle\intr  P_{1}\dmu_{0}\leq C_{\tilde{T},T}$.
\end{proof}

%%%%%%%%%%%%%%%%%%%%%%%%%%%%%%%%%%%%%%%%%%%%%%%%%%%%%%%
\subsubsection{Existence of very weak solutions}

We have just seen that nonnegative weak solutions to equation~\eqref{a.nonlocalheatequation} have an initial trace $\mu_0$ that is a Radon measure satisfying~\eqref{eq:growth.initial.trace}. We now investigate the reciprocal. Is it true that for any nonnegative Radon measure $\mu_0$ satisfying~\eqref{eq:growth.initial.trace} there is a very weak solution to~\eqref{a.nonlocalheatequation} having $\mu_0$ as initial trace? We will answer to this question positively. In fact, we will be able to construct also solutions with sign changes, if the initial trace satisfies a suitable growth condition closely related to~\eqref{eq:growth.initial.trace}. These solutions belong to the class $\li$ for which there is uniqueness. As a corollary we will obtain a representation formula.

\begin{theorem}\label{prop2}
    Let $\mathcal{L}=\mathcal{L}_K$ with $K$ satisfying~\eqref{A0}--\eqref{A2}. Let $\mu_{0}$ be a Radon measure with the admissible growth at  infinity~\eqref{q4}. Then, for each $T>0$ the function $U$ given in~\eqref{q5} is a very weak solution of \eqref{a.nonlocalheatequation} in $C(\Rn\times(0,T))\cap\li$ with initial trace~$\mu_{0}$.
 \end{theorem}

\begin{proof}
We first check that $U$ is well defined. By \eqref{Atotal2}, $P_{t}\leq C_{t}P_{1}$ for some $C_{t}>0$ depending on~$t$. Then, using Lemma~\ref{v3.19} and~\eqref{q4}, we obtain that
\begin{equation}\label{c5}	
\begin{aligned}
    \intr P_{t}(x-y)\,{\rm d}|\mu_{0}|(y)&\leq C_{t}\intr P_{1}(x-y)\,{\rm d}|\mu_{0}|(y)\\ &\leq C_{x,t}\intr P_{1}(y)\,{\rm d}|\mu_{0}|(y)<\infty,\quad\xt\in\Rn\times (0,\infty).
\end{aligned}
\end{equation}

Now, we will prove that $U\in \li$. By Tonelli's theorem, the fact that $P_{t}$ is symmetric, the semigroup property~\eqref{eq:semigroup.property},~\eqref{Atotal2} and~\eqref{q4},
\begin{align}
	\notag 	\intr |U\xt|P_{1}(x)\dx &\leq \intr\intr P_{t}(y-x)P_{1}(y)\dx\,{\rm d}|\mu_{0}|(y)\\
	\label{v3.16}&=\intr P_{1+t}(y){\rm d}|\mu_{0}|(y) 	\leq C_{T}\intr P_{1}(y)\,{\rm d}|\mu_{0}|(y)<\infty\quad\textrm{for all } t\in (0,T).
\end{align}

Let us prove next that $\mu_{0}$ is the initial trace of $U$. Let $\psi\in C_{\textup{c}}(\Rn)$. By~\eqref{a.5} and~\eqref{q4},
\begin{equation*}
    \intr\intr P_{t}(x-y)|\psi(x)|\dx\,{\rm d}|\mu_{0}|(y)\leq C_{T,\psi}\intr P_{1}(y)\,{\rm d}|\mu_{0}|(y)<\infty.
\end{equation*}
Then, by Fubini-Tonelli's theorem,
\begin{equation*}\label{q6}	
    \intr U(x,t)\psi(x)\dx=\intr \intr P_{t}(x-y)\psi(x)\dmu_{0}(y)\dx=\intr \intr P_{t}(x-y)\psi(x)\dx\dmu_{0}(y).
\end{equation*}
Since $P_{t}$ is symmetric, we may use the DCT in $L^{1}(\mathbb{R}^d,{\rm d}|\mu_{0}|)$ to get
\begin{equation*}
	\lim_{t\to0^{+}}\intr U(x,t)\psi(x)\dx=\intr \lim_{t\to0^{+}}\intr P_{t}(y-x)\psi(x)\dx\dmu_{0}(y)= \intr\psi(y)\dmu_{0}(y),
\end{equation*}
because $P_{t}(\cdot-y)$ has the Dirac delta centered at $y$ as initial trace;~see Proposition \ref{pprop2}.

Besides, since  $U\in \li\subset\luloc$, by Proposition~\ref{prop1}, $U$~satisfies~\eqref{a.2}.
	
Finally, let us prove that $U$ satisfies~\eqref{a.3}. Let $\theta\in C^{\infty}_{\textup{c}}(\RN\times (0,T))$. Due to Lemma~\ref{a.17},~\eqref{v3.16}, the symmetry of $P_{t}$, the semigroup property~\eqref{eq:semigroup.property} and~\eqref{Atotal2},
\begin{align*}
    \int_{0}^{T}\intr \intr P_{t}(x-y)|\Lk\theta(x,t)|\dx\,{\rm d}|\mu_{0}|(y)\dt&\leq C_{\theta}\int_{0}^{T}\intr \intr P_{t}(x-y) P_{1}(x)\dx\,{\rm d}|\mu_{0}|(y)\dt\\
	&\leq C _{\theta}C_T T\intr P_{1}(y)\,{\rm d}|\mu_{0}|(y)<\infty.
\end{align*}
Since $P$ is a very weak solution, see Corollary~\ref{PisaVeryWeaksolution.cor}, so is any  spatial translation of it, due to the translation invariance of the equation. Therefore,  using also Fubini-Tonelli's Theorem,
\begin{equation}\label{z7}
    \begin{aligned}
        \int_{0}^{T}\intr U(x,t)\Lk\theta(x,t)\dx\dt&=\intr\int_{0}^{T}\intr P_{t}(x-y)\Lk\theta(x,t)\dx\dt\dmu_{0}(y)\\
        &=\intr\int_{0}^{T}\intr  P_{t}(x-y)\partial_{t}\theta(x,t) \dx\dt\dmu_{0}(y).
    \end{aligned}
\end{equation}
On the other hand, using~\eqref{a.5} and~\eqref{q4},
\begin{equation*}
	\intr\int_{0}^{T}\intr  P_{t}(x-y)|\partial_{t}\theta(x,t)| \dx\dt\,{\rm d}|\mu_{0}|(y)<\infty.
\end{equation*}
Then, we can use Fubini-Tonelli's Theorem to change the order of integration for the last integral in~\eqref{z7}, and we finally obtain that $U$ satisfies~\eqref{a.3}.
	
Now, we will prove that $U\in C(\Rn\times(0,T))$. This comes from the fact that $U$ is the limit, uniform in compact sets, of continuous functions.

Let us define the approximating functions. For each $r>0$ we define the finite Radon measure
\begin{equation*}\label{c7}	
    \mu^{r}(B):=\mu_{0}(B\cap B_{r}(0))\quad\textrm{for all Borel set }B,
\end{equation*}
and we set
\begin{equation*}\label{c8}	
    U^{r}(x,t):=\intr P_{t}(x-y)\dmu^{r}(y)=\int_{\{|y|\leq r\}}P_{t}(x-y)\dmu_{0}(y).
\end{equation*}
Let us check that $U^r\in C(\Rn\times(0,T))$. Let $(x_{1},t_{1})\in \Rn\times (0, T)$. Then
\begin{equation*}
	|U^{r}(x_{1},t_{1})-U^{r}(x_{2},t_{2})|\leq \int_{\{|y|\leq r\}}|P_{t_{1}}(x_{1}-y)-P_{t_{2}}(x_{2}-y)|\,{\rm d}|\mu_{0}|(y).
\end{equation*}
Since $P$ is uniformly continuous  in $\Rn\times(\varepsilon,T')$ if $0<\varepsilon<T'<T$ (Corollary \ref{pcor1}), given $\widetilde{\varepsilon}>0$  there exist $\delta_{1},\delta_{2}>0$ such that
\begin{equation*}
    |P_{t_{1}}(x_{1}-y)-P_{t_{2}}(x_{2}-y)|\leq \widetilde{\varepsilon}
\end{equation*}
for each $|x_{1}-x_{2}|\leq \delta_{1}$ and $|t_{1}-t_{2}|<\delta_{2}$, so that we have the desired continuity,
\begin{equation*}
    |U^{r}(x_{1},t_{1})-U^{r}(x_{2},t_{2})|\leq \widetilde{\varepsilon} |\mu_{0}|(B_{r}(0)).
\end{equation*}

Let us now prove that $U^{r}$ converges uniformly to $U$ in compact sets of $\Rn\times (0,T)$. Let $R>0$ and $0<\varepsilon<T'<T$. We define
\begin{equation*}
	I_{r}(R):=\sup\limits_{\varepsilon\leq t\leq T'}\sup\limits_{|x|\leq R}|U\xt-U^{r}\xt|\leq \sup\limits_{\varepsilon\leq t
    \leq T'}\sup\limits_{|x|\leq R}\int_{\{|y|>r\}}P_{t}(x-y)\,{\rm d}|\mu_{0}|(y).
\end{equation*}
By~\eqref{Atotal2}, there is a  positive constant $C$ only depending on $\varepsilon$ and $T$, such that $P_{t}(x)\leq CP_{1}(x)$ for all $\xt\in \Rn\times (\varepsilon,T')$.  Thus, using also~\eqref{eq:comparison.translations} and the growth condition~\eqref{q4} for the initial datum, 	
\begin{equation*}
    I_{r}(R)\leq C\int_{\{|y|>r\}}P_{1}(y)\,{\rm d}|\mu_{0}|(y)\to 0\quad\text{as }r\to+\infty,
\end{equation*}
which shows the desired uniform convergence.
\end{proof}

As a corollary to our existence and uniqueness results, Theorem~\ref{prop2}, Corollary~\ref{UniquenessAcotadas} and Corollary~\ref{UniquenessNonnegative}, there is a representation formula for very weak solutions belonging to $\li$ and for nonnegative very weak solutions whenever the initial datum $\mu_{0}$ is admissible.

\begin{corollary}[Representation formula]\label{CorRepresentationFormulaVW}
    Let $\mathcal{L}=\mathcal{L}_K$ with $K$ satisfying~\eqref{A0}--\eqref{A2} and let $\mu_{0}$ be a Radon measure satisfying~\eqref{q4}. Let $u$ be a very weak solution to~\eqref{a.nonlocalheatequation} with trace $\mu_{0}$. If either $u\in \li$, or $u$ is nonnegative and $K$ satisfies in addition~\eqref{A3}, then $u$ is given by representation formula~\eqref{representationformula}.
\end{corollary}

%%%%%%%%%%%%%%%%%%%%%%%%%%%%%%%%%%%%%%%%%%%%%%%%%%%%%%%
\subsection{Classical solutions}

We consider now classical solutions. If we are able to prove that classical solutions are very weak solutions, our uniqueness theorems for the latter class, Corollary~\ref{UniquenessAcotadas} and Corollary~\ref{UniquenessNonnegative}, will yield uniqueness for the former. In the case of nonnegative solutions, this will also guarantee the existence of an initial trace, thanks to Theorem~\ref{theoremtraces}, and a representation formula, because of Theorem~\ref{prop2}. However, proving this is not as straightforward as in the case of the classical heat equation, due to the nonlocal character of the diffusion operator~$\Lk$. This will be our first goal.

Next, we will give sufficient conditions, in terms of the heat kernel, ensuring the existence of a classical solution with initial trace any Radon measure $\mu_{0}$ satisfying the integral growth condition~\eqref{prop2}. Such conditions are satisfied, for instance, by the heat kernel of the fractional heat equation.

%%%%%%%%%%%%%%%%%%%%%%%%%%%%%%%%%%%%%%%%%%%%%%%%%%%
\subsubsection{Uniqueness of classical solutions}

The following proposition presents two conditions that together are sufficient for a classical solution, even one with sign changes, to be a very weak solution. The first condition, denoted by \eqref{energyH}, concerns the regularity of the function. We require a classical solution $u$ to possess this regularity to ensure that $\Lk u$ is integrable against any test function. This would correspond to the regularity requirement $u\in C^{2,1}_{x,t}$ in the classical result by Widder. The second one, the integrability condition $u\in\luloc$, arises from the nonlocal character of the operator $\mathcal{L}$, and is required in the definition of very weak solution; see Proposition~\ref{prop1}. However, we will prove later that it is satisfied by all nonnegative classical solutions.

\begin{proposition}\label{cprop1}
    Let $\mathcal{L}=\mathcal{L}_K$ with $K$ satisfying~\eqref{A0}--\eqref{A2} and $u$ a classical solution to~\eqref{a.nonlocalheatequation}. If $u\in\luloc$ and satisfies~\eqref{energyH}, then it is a very weak solution to the same equation.
\end{proposition}

\begin{proof}
Since $u\in\luloc$, $u$ satisfies \eqref{a.2}; see Proposition~\ref{prop1}.

It remains to prove that $u$ satisfies \eqref{a.3}. Since $u$ is a classical solution,
\begin{equation*}
    0=\int_0^T\intr (\partial_tu\xt+\Lk u\xt)\theta\xt\dx\dy=\int_0^T\intr (-u\xt\partial_t\theta\xt+\Lk u\xt\theta\xt)\dx\dy
\end{equation*}
for all $\theta\in C^{\infty}_{\textup{c}}(\RN\times(0,T))$. If
\begin{equation}\label{a.30}
    \mathcal{C}:=\int_0^T\intr\intr|\theta\xt| |\Lambda u(x,y,t)|K(y)\dy\dx\dt<\infty,
\end{equation}
we would have, via Fubini's theorem, that
\begin{equation*}
    \int_0^T\intr\Lk u\xt\theta\xt\dx\dy=\int_0^T\intr u\xt\Lk\theta\xt\dx\dy,
\end{equation*}
and we would be done. However, proving~\eqref{a.30} is not straightforward, because of the singularity of L\'evy kernel~$K$ at the origin. It is here that we need hypothesis~\eqref{energyH}.

To prove~\eqref{a.30}, we make the decomposition $\mathcal{C}=\textup{I}+\textup{II}$, where
\begin{align*}
    \textup{I}&:=\int_0^T\intr\ints|\theta\xt|\,|\Lambda u(x,y,t)|K(y)\dy\dx\dt,\\
    \textup{II}&:=\int_{0}^{T}\intr\intb|\theta\xt|\,|\Lambda u(x,y,t)|K(y)\dy\dx\dt.
\end{align*}
The term $\textup{I}$ is finite, due to~\eqref{energyH}. As for $\textup{II}$, since $K$ is symmetric, then $\textup{II}\le \textup{II}_1+\textup{II}_2$, where
\begin{align*}
    \textup{II}_1&:=\int_{0}^{T}\intr|\theta\xt||u(x,t)|\dx\dt\,\intb K(y)\dy,\\
    \textup{II}_2&:=\int_{0}^{T}\intr\intb |\theta\xt|\,|u(x+y,t)|K(y)\dy\dx\dt\\
    &=\int_{0}^{T}\intr|u(x,t)|\intb |\theta(x-y,t)|K(y)\dy\dx\dt.	
\end{align*}	
As $u\in C(\Rn\times(0,T))$ and $K$ is a L\'evy kernel, then
\begin{equation*}
    \textup{II}_1\le \|\theta u\|_{L^1(\RN\times(0,T))}\intb K(y)\dy<\infty.
\end{equation*}

Let now $0<\varepsilon<T'<T$ such that $\operatorname{supp}(\theta)\subset \Rn\times(\varepsilon,T')$. Using~\eqref{Atotal}, \eqref{a.4},~\eqref{eq:semigroup.property} and \eqref{Atotal2},
\begin{align*}
    \textup{II}_2&\leq C_{\theta,T}\int_{\varepsilon}^{T'}\intr|u(x,t)|\intr P_1(x-y)P_1(y)\dy\dx\dt
    =C_{\theta,T}\int_{\varepsilon}^{T'}\intr|u(x,t)|P_2(x)\dx\dt\\
    &\leq C_{\theta,T}\int_{\varepsilon}^{T'}\intr|u(x,t)|P_{1}(x)\dx\dt<\infty,
\end{align*}
because $u\in \luloc$.
\end{proof}

\begin{corollary}\label{CorRepresentativeFormulaClassicalAcotadas}
    Let $\mathcal{L}=\mathcal{L}_K$ with $K$ satisfying~\eqref{A0}--\eqref{A2}.

    \noindent{\rm (i)} There is at most one classical solution $u\in\li$ to~\eqref{a.nonlocalheatequation} satisfying~\eqref{energyH} with initial trace a given Radon measure.

    \noindent{\rm (ii)} Let $K$ satisfy also~\eqref{A3}. If there exists a classical solution $u\in\li$ to~\eqref{a.nonlocalheatequation} satisfying~\eqref{energyH} with initial trace a measure $\mu_0$ verifying~\eqref{q4}, then $u$ is given by the representation formula~\eqref{representationformula}.
\end{corollary}

\begin{proof}
(i) Suppose that $u\in\li$ is a classical solution to~\eqref{a.nonlocalheatequation} that satisfies~\eqref{energyH}. Since  $ \li\subset\luloc$, then $u$ is a very weak solution of \eqref{a.nonlocalheatequation}; see Proposition~\ref{cprop1}. Uniqueness now follows from Corollary~\ref{UniquenessAcotadas}, since $u\in\li$.
	
\noindent (ii) It follows immediately from Corollary~\ref{CorRepresentationFormulaVW}, since $u$ is a very weak solution of~\eqref{a.nonlocalheatequation}.
\end{proof}

We move on now to the class of nonnegative classical solutions $u$ to~\eqref{a.nonlocalheatequation} satisfying~\eqref{energyH}. We will prove that $u\in\luloc$, whence, by Proposition~\ref{cprop1}, it will be a very weak nonnegative solution, from where uniqueness and a representation formula will follow. The aforementioned integrability will be a consequence of the inequality
\begin{equation}\label{a.16}
    U(x,t):=\intr P_{t}(x-y)\dmu_{0}(y)\leq u\xt \quad\textrm{for all }\xt\in\Rn\times (0,T),
\end{equation}
satisfied by nonnegative classical solutions, even if~\eqref{energyH} does not hold. The first step to prove~\eqref{a.16} is to obtain a maximum principle, which we do next. This only requires~\eqref{A0}.

\begin{lemma}[A maximum principle] \label{a.mp}
    Let $\mathcal{L}=\mathcal{L}_K$ with $K$ satisfying~\eqref{A0}. Let $\Omega\subset\RN$ be non-empty, open, bounded and smooth, and $v\in C(\overline{\Omega}\times[0,T))$ such that $\partial_t v\in C(\Omega\times(0,T))$ and $\Lk v\xt$ is defined for all $\xt\in\Omega\times (0,T)$. If
	\begin{equation}\label{eq:ineq.MP}
		\begin{cases}
			(\partial_t v+\Lk v)(x,t) \leq 0\quad&\textrm{for all }\xt\in\Omega\times(0,T),\\
			v\xt\leq 0&\textrm{for all }(x,t)\in\big(\Rn\times[0,T)\big)\setminus\big(\Omega\times(0,T)\big),
		\end{cases}
	\end{equation}
	then $v\leq 0$ for all $x\in\mathbb{R}^d\times[0,T)$.
\end{lemma}

\begin{proof}
Choose any $T'\in (0,T)$. Let $(x_0,t_0)\in \bar{\Omega}\times[0,T']$ such that
\begin{equation*}
	v(x_0,t_0):=\max\limits_{\bar{\Omega}\times[0,T']}v\xt.
\end{equation*}	
Assume for contradiction that $v(x_0,t_0)>0$. Thanks to the \lq\lq boundary conditions'', necessarily $(x_0,t_0)\in\Omega\times(0,T']$. Observe that $\partial_t v(x_0,t_0)\geq0$ and $\Lambda v(x_0,y,t_0)\geq 0$ for all $y\in \RN$, because $v$ reaches it maximum at time $t_{0}$ at $x=x_{0}$. Hence, for any $R_0>0$,
\begin{equation*}
	0\ge-\partial_t v(x_0,t_0)\geq \Lk v (\xoo,\too)\geq \int_{B_{2R_{0}}^{\textup{c}}(0)}\Lambda v(x_{0},y,t_0)K(y)\dy.
\end{equation*}
Let $R_{0}>0$ be such that $\Omega\subset B_{R_{0}}(0)$, and $|y|> 2R_{0}$. Then, $|x_{0}\pm y|>R_0$, whence $x_0\pm y\not\in \Omega$ and therefore $v(x_0\pm y,t_0)\leq 0$. Thus,
\begin{equation*}
	\label{a.mp1}	0\ge \int_{B_{2R_{0}}^{\textup{c}}(0)}\Lambda v(x_{0},y,t_{0})K(y)\dy \geq v(\xoo,\too)\int_{B_{2R_{0}}^{\textup{c}}(0)}K(y)\dy>0,
\end{equation*}
because $K$ is positive, see \eqref{A0}, and we arrive at a contradiction.
\end{proof}

We now proceed to prove~\eqref{a.16}.

\begin{lemma}\label{a.cor4}
    Let $\mathcal{L}=\mathcal{L}_K$ with $K$ satisfying~\eqref{A0}--\eqref{A2}. If $u$ is a nonnegative classical solution of the nonlocal heat equation~\eqref{a.nonlocalheatequation} with initial data a nonnegative Radon measure $\mu_0$, then~\eqref{a.16} holds.
\end{lemma}

\begin{proof}
Let $\delta\in (0,T)$. Since $u$ is a classical solution, $u(\cdot,\delta)\in C(\Rn)$. We consider the family of cutoff functions $\{\psi_{R}\}_{R>0}$ with $\psi_R$ as in~\eqref{cutoffspace}. Let
\begin{equation*}
	v_{R}^{\delta}\xt=\intr P_{t}(x-y)\psi_{R}(y)u(y,\delta)\dy.
\end{equation*}
Since the heat kernel $P\in C^{\infty}(\Rn\times(0,\infty))$, see Corollary~\ref{pcor1},  and $\big(\mathcal{P}_{t}\big)_{t\geq 0}$ is a strongly continuous semigroup in $C_{0}(\Rn)$, then $v_{R}^{\delta}\in C^{\infty}(\Rn\times[0,T))$, and solves
\begin{equation*}
	\begin{cases}
		\partial_{t}v_{R}^{\delta}+\Lk v_{R}^{\delta}=0\quad&\textrm{everywhere in }\Rn\times (0,T),\\
		v_{R}^{\delta}\geq 0&\text{everywhere in }\Rn\times [0,T),\\
		v_{R}^{\delta}(\cdot,0)=\psi_{R}(\cdot)u(\cdot,\delta)&\textrm{everywhere in }\Rn.
	\end{cases}
\end{equation*}
By~\eqref{a.5},
\begin{equation*}
	0\leq v_{R}^{\delta}(x,t)\leq  \|u(\cdot,\delta)\|_{L^{\infty}(B_{2R}(0))}(P_{t}*\psi_{R})(x)
	\leq C_{R, T}\|u(\cdot,\delta)\|_{L^{\infty}(B_{2R}(0))}P_{1}(x).
\end{equation*}
On the other hand, $P_{1}(x)\to 0$ as $|x|\to \infty$, since $P_{1}$ is comparable to $K$ away from the origin, \eqref{Atotal}, and $K(x)\to 0$ as $|x|\to\infty$, see \eqref{A1} and \eqref{A2}. Therefore, given $\varepsilon>0$ we can choose $\rho>0$ large such that
$C_{R,T}\|u(\cdot,\delta)\|_{L^{\infty}(B_{2R}(0))}P_{1}(x)\leq \varepsilon$ for all $|x|\geq \rho$. Hence, since $u$ is nonnegative,
\begin{equation*}
	\begin{array}{ll}
		0\leq v_{R}^{\delta}(x,t)\leq\varepsilon\leq u(x,t+\delta)+\varepsilon\quad&\textrm{for any }|x|\geq\rho,\,t\in(0,T-\delta)\\[6pt]
		v_{R}^{\delta}(x,0)=\psi_{R}(x)u(x,\delta)\leq u(x,\delta)\leq u(x,\delta)+\varepsilon\quad&\textrm{for any }x\in\Rn.
	\end{array}
\end{equation*}
Thus, applying the Maximum Principle, Lemma~\ref{a.mp}, with $\Omega=B_{\rho}(0)$ to $w:=v_{R}^{\delta}-u^{\delta}-\varepsilon$, we get
\begin{equation*}
	v_{R}^{\delta}\xt\leq u(x,t+\delta)+\varepsilon\quad\textrm{for all }\xt\in\Rn\times[0,T-\delta).
\end{equation*}
Since $\varepsilon$ is arbitrary,
\begin{equation*}\label{a.15}
    v_{R}^{\delta}\xt=	\intr P_{t}(x-y)\psi_{R}(y)u(y,\delta)\dy\leq u(x,t+\delta)\quad\textrm{for all }\xt\in\Rn\times[0,T-\delta).
\end{equation*}
Using that the initial trace of $u$ is $\mu_0$, for all $(x,t)\in\Rn\times (0,T)$ we have
\begin{equation*}
	\intr P_{t}(x-y)\psi_{R}(y)\dmuo(y)=\operatornamewithlimits{ess\, lim}_{\delta\to0^+} \intr P_{t}(x-y)\psi_{R}(y)u(y,\delta)\dy
    \leq \lim\limits_{\delta\to0^{+}}u(x,t+\delta)=u(x,t),
\end{equation*}
since $u\in C(\Rn\times(0,T))$. Letting $R\to\infty$, using the MCT, we arrive at~\eqref{a.16}.
\end{proof}

\begin{remark}\label{remark1}	
    Thanks to inequality \eqref{a.16} and the fact that $\mu_{0}$ is a nonnegative measure, the integral $\intr P_{t}(x-y)\dmuo(y)$  is well defined as long as the nonnegative classical solution $u$ exists. Moreover, $\intr P_{1}\dmuo\leq u(0,1)<\infty$, because $P_{1}$ is symmetric. Hence, we do not need to assume the integrability condition~\eqref{eq:growth.initial.trace} in advance in order for $U$ to be well defined.
\end{remark}

We can prove now the desired integrability, and even a bit more than required.

\begin{corollary}\label{a.cor2}
    Let $\mathcal{L}=\mathcal{L}_K$ with $K$ satisfying~\eqref{A0}--\eqref{A2}. If $u$ is a nonnegative classical solution of the nonlocal heat equation \eqref{a.nonlocalheatequation} with initial trace a nonnegative Radon measure $\mu_{0}$, then $u\in L^{1}_{\rm loc}([0,T);\lp)$.
\end{corollary}

\begin{proof}
Since $u$ is nonnegative, it is enough to prove that for all $T'\in(0,T)$ we have
\begin{equation*}
	\int_{0}^{T'}\intr P_{1}(y)u(y,\tau)\dy\,{\rm d}\tau<\infty.
\end{equation*}

Let $\tau\in (0,T')$ and $u^\tau(x,t):=u(x,t+\tau)$, $(x,t)\in\Rn\times[0,T-\tau)$. Since $u$ is a classical solution to~\eqref{a.nonlocalheatequation}, so it is $u^{\tau}$, for $t\in (0,T-\tau)$. The initial trace of $u^\tau$ is $u(\cdot,\tau)$. By Lemma~\ref{a.cor4},
\begin{equation*}
	\intr P_{t}(x-y)u(y,\tau)\dy\leq u(x,t+\tau)\quad\textrm{for all }\xt\in\Rn\times(0,T-\tau)
\end{equation*}
which implies that
\begin{equation}\label{eq:estimate.for.integrability}
    \int_{0}^{T'}\intr P_{t}(x-y)u(y,\tau)\dy\,{\rm d}\tau\leq \int_{t}^{t+T'}u(x,s)\,{\rm d}s.
\end{equation}
Let $T''\in (T',T)$. By \eqref{Atotal2} we know that $P_1(y)\leq C_{T',T''} P_{T''-T'}(y)$ for all $y\in\Rn$.
Then, since $P_t$ is symmetric for all times, using the estimate~\eqref{eq:estimate.for.integrability} with $x=0$, we get
\begin{align*}
	\int_{0}^{T'}\intr P_{1}(y)u(y,\tau)\dy\,{\rm d}\tau&\leq  C_{T',T''}\int_{0}^{T'}\intr P_{T''-T'}(-y)u(y,\tau)\dy\,{\rm d}\tau\\
    &\leq  C_{T''-T'}\int_{T''-T'}^{T''}u(0,\tau)\,{\rm d}\tau<\infty,
\end{align*}
due to the fact that  $u(0,\cdot)\in C([T''-T',T''])$.
\end{proof}

At this point we have gathered all the necessary ingredients to prove what may be considered as one of the main outcomes of the paper: a Widder-type representation formula (implying uniqueness) for nonnegative classical solutions (assuming that the \lq\lq energy assumption''~\eqref{energyH} holds).

\begin{theorem} \label{thm:main.K}
    Let $\mathcal{L}=\mathcal{L}_K$ with $K$ satisfying~\eqref{A0}--\eqref{A3}. Let $u$ be a nonnegative classical solution of the nonlocal heat equation~\eqref{a.nonlocalheatequation} with initial trace a nonnegative Radon measure $\mu_{0}$. If~\eqref{energyH} holds,  then $u$ is given by the representation formula~\eqref{representationformula}.
\end{theorem}

\begin{proof}
Thanks to its nonnegativity, $u\in \luloc$; see Corollary~\ref{a.cor2}. Therefore, since $u$ satisfies~\eqref{energyH}, it is  a nonnegative very weak solution of \eqref{a.nonlocalheatequation}; see Proposition~\ref{cprop1}. Finally, the integral growth condition~\eqref{eq:growth.initial.trace} holds; see Remark~\ref{remark1}. Thus, $u$ satisfies the hypotheses of Corollary~\ref{CorRepresentationFormulaVW}, from where the result follows immediately.
\end{proof}

%%%%%%%%%%%%%%%%%%%%%%%%%%%%%%%%%%%%%
\subsubsection{Existence of classical solutions}

We would like to prove now that given a Radon measure $\mu_0$ satisfying the integral growth condition~\eqref{q4}, there is a classical solution to~\eqref{a.nonlocalheatequation} having it as initial trace. We have a natural candidate: the function $U(x,t)=P_t*\mu_0$, which is a very weak solution of the equation and has $\mu_0$ as initial trace; see~Theorem~\ref{prop2}. Moreover, it belongs to the class $\li$ in which there is uniqueness. However, proving that $U$ is a classical solution is not a trivial task, and will require some assumptions, that we will write in terms of the heat kernel associated to the operator. These assumptions are satisfied, for instance, when the operator is the fractional Laplacian.

\begin{theorem}\label{Uclassical}
    Let $\mathcal{L}=\mathcal{L}_K$ with $K$ satisfying~\eqref{A0}--\eqref{A2}, $\mu_{0}$ a Radon measure satisfying the integral growth condition~\eqref{q4}, and $U$ the very weak solution to~\eqref{a.nonlocalheatequation} given in~\eqref{q5}. Suppose that for each $T>0$ and $\varepsilon\in(0,T)$ there is a constant $C>0$ such that
    \begin{equation}\label{eq:conditions.existence.classical}
    	|\partial_{t} P_{t}(x)|\leq CP_{1}(x), \quad \int_{\mathbb{R}^d}|\Lambda P(x,y,t)| K(y)\,\textup{d}y\le C P_1 (x)\quad\text{for all }\xt\in \Rn\times(\varepsilon,T).
    \end{equation}

    \noindent\textup{(i)} Then, $U$ is a classical solution to~\eqref{a.nonlocalheatequation} within $\li$ and with initial trace $\mu_{0}$.

    \noindent\textup{(ii)} If in addition for each $T>0$ and $\varepsilon\in(0,T)$ there is a constant $C>0$ such that
    \begin{equation}\label{eq:stronger.conditions.existence.classical}
        |D^{2} P_{t}(x)|\leq CP_{1}(x)\quad\text{for all }\xt\in \Rn\times(\varepsilon,T),
    \end{equation}
    then $U$ is the unique classical solution to~\eqref{a.nonlocalheatequation} with initial trace $\mu_{0}$ that satisfies~\eqref{energyH} and belongs to $\li$.

    \noindent\textup{(iii)} If, in addition to all the above, $K$ satisfies~\eqref{A3} and $\mu_0$ is nonnegative, then $U$ is the unique nonnegative classical solution to~\eqref{a.nonlocalheatequation} with initial trace $\mu_0$ that satisfies~\eqref{energyH}.
\end{theorem}

\begin{remark}
    Since $K$ satisfies the L\'evy condition~\eqref{A0}, it is easily checked that condition~\eqref{eq:stronger.conditions.existence.classical} implies the second condition in~\eqref{eq:conditions.existence.classical}.
\end{remark}

\begin{remark}
    Let $\Lk=\Ls$, $\alpha\in (0,2)$. Thanks to Proposition~\ref{pprop1} (iv) and (vi),
    \begin{equation*}
       |\partial_{t} P_{t}(x)|\leq C\varepsilon^{-1}P_{1}(x),\quad |D^2 P^\alpha_t(x)|\leq C\varepsilon^{-2\alpha}P_{1}(x)\quad\text{for all }\xt\in \Rn\times(\varepsilon,T),
    \end{equation*}
    whence Theorem~\ref{Uclassical} (ii)--(iii) apply to this case.
\end{remark}

Thanks to Theorem~\ref{prop2} we already know that the function $U$ has initial datum $\mu_0$ and belongs to $C(\Rn\times (0,T))\cap\li$, which implies that $U\in \luloc$. Thus, it only remains to check that $\partial_{t}U$ belongs to $C(\Rn\times (0,T))$ and that $U$ satisfies~\eqref{eq:classical.equation}. This is what we do in the next two lemmas.

\begin{lemma}\label{clemma2}
	Under the general assumptions of Theorem~\ref{Uclassical}, $\partial_{t}U\xt\in C(\Rn\times (0,T))$.
\end{lemma}

\begin{proof}
The first step is to prove that
\begin{equation}\label{eq:representation.derivative}
	\partial_{t}U\xt=\intr \partial_{t}P(x-y,t)\dmuo(y).
\end{equation}
Let $0<\varepsilon<T'<T$ and $\xt\in\Rn\times (\varepsilon,T')$. By definition,
\begin{equation*}
	\partial_{t} U\xt=\lim_{h\to0}\intr\frac{P(x-y,t+h)-P(x-y,t)}{h}\dmu_{0}(y).
\end{equation*}
The MVT and the first hypothesis in~\eqref{eq:conditions.existence.classical} imply that
\begin{equation*}
	\Big|\frac{P(x-y,t+h)-P(x-y,t)}{h}\Big|\leq |\partial_{t}P(x-y,\tilde{t})|\leq CP_{1}(x-y)
\end{equation*}
for all $x,y\in \Rn$, $t\in(\varepsilon,T')$ and $2|h|\leq \varepsilon\wedge (T-T')$, where $\tilde{t}$ depends on $x$, $y$, $t$, and  $h$ and $|\tilde{t}-t |\leq |h|$. Moreover, $C$ only depends on $\varepsilon$ and $T'$. Since $\mu_{0}$ satisfies \eqref{q4}, we already know that $\intr P(x-y,1)\,{\rm d}|\mu_{0}|(y)<\infty$ for each $(x,t)\in\Rn\times(\varepsilon,T')$; see~\eqref{c5}. Therefore, applying the DCT we obtain that $\partial_{t}U$ is well defined and
\begin{equation*}
	\partial_{t} U\xt=\lim_{h\to0}\intr\frac{P(x-y,t+h)-P(x-y,t)}{h}\dmu_{0}(y)=\intr \partial_{t}P(x-y,t)\dmu_{0}(y).
\end{equation*}

The continuity of $\partial_t U$ follows from the fact that it is the limit as $r\to+\infty$, uniform in compact sets, of the approximating functions
\begin{equation*}
    (\partial_tU)^{r}\xt=\intr \partial_{t}P(x-y,t)\,{\rm d} \mu^{r}(y),
\end{equation*}
arguing as in the proof of continuity of $U$ performed in Theorem~\ref{prop2}; we omit the details.
\end{proof}

\begin{lemma}\label{clemma1}
	Under the general assumptions of Theorem~\ref{Uclassical}, $U$ satisfies~\eqref{eq:classical.equation}.
\end{lemma}

\begin{proof}
We start from~\eqref{eq:representation.derivative}. Since $P$ is a classical solution to the nonlocal heat equation~\eqref{a.nonlocalheatequation},
\begin{equation}\label{c6}	
    \partial_{t} U\xt=-\intr \Lk P(x-y,t)\dmu_{0}(y)=-\intr\intr \Lambda P(x-y,z,t)K(z)\,{\rm d}z\dmu_{0}(y).
\end{equation}
At this point, we would like  to change the order of integration in the last integral to conclude that $\partial_{t} U+\Lk U=0$. By the second hypothesis in~\eqref{eq:conditions.existence.classical},
\begin{equation*}
	\intr\intr |\Lambda P(x-y,z,t)|K(z)\,{\rm d}z\,{\rm d}|\mu_{0}|(y)\leq C\intr P_1(x-y)\,{\rm d}|\mu_{0}|(y)<\infty,
\end{equation*}
because the last integral is finite for all $x\in\Rn$, as we proved in \eqref{c5}. Since  the last integral in \eqref{c6} converges absolutely, using Fubini-Tonelli's Theorem
\begin{equation*}
    \partial_{t} U\xt=-\intr\intr \Lambda P(x-y,z,t)K(z)\dmu_{0}(y)\,{\rm d}z= -\intr  \Lambda  U(x,z,t)K(z)\,{\rm d}z= -\Lk U(x,t)
\end{equation*}
for all $(x,t)\in\Rn\times(0,T)$.
\end{proof}

\begin{proof}[Proof of Theorem \ref{Uclassical}]
(i) By Theorem~\ref{prop2},  Lemma~\ref{clemma2} and Lemma~\ref{clemma1}, $U$ is a classical solution of \eqref{a.nonlocalheatequation} with initial trace $\mu_{0}$ and $U\in \li$.

\noindent (ii) It only remains to prove that $U$ satisfies~\eqref{energyH}, since then uniqueness will follow from Corollary~\ref{CorRepresentativeFormulaClassicalAcotadas}. Given $\theta\in C^{\infty}_{\textup{c}}(\Rn\times(0,T))$, let $0<\varepsilon<T'<T$ be such that $\operatorname{supp}(\theta)\subset\mathbb{R}^n\times(\varepsilon,T')$. Using the definition of $U$ and applying Taylor's expansion,
\begin{align*}
	 &\left|\int_0^T\intr\theta\xt\ints \Lambda U(x,y,t) K(y)\dy\dx\dt \right|\\
	&\leq\int_\varepsilon^{T'}\intr|\theta\xt|\ints K(y)\intr |\Lambda P(x-z,y,t)|\,{\rm d}|\mu_{0}|(z)\dy\dx\dt\\
    &\leq C\int_\varepsilon^{T'}\intr|\theta\xt|\ints|y|^{2}K(y)\intr |D^{2}P(x-z+w(x,z,y,t),t)|\,{\rm d}|\mu_{0}|(z)\dy\dx\dt,
\end{align*}
where $w(x,z,y,t)\in B_{1}(0)$ because $|y|\leq 1$. By assumption~\eqref{eq:stronger.conditions.existence.classical},
\begin{equation*}
	|D^{2}P(x-z+w(x,z,y,t),t)| \leq CP_1(x-z+w(x,z,y,t)),
\end{equation*}
where $C$ only depends on $\varepsilon$ and $T'$. Since $P_1$ is {slowly changing}, see Proposition~\ref{Pslowly}, and $|w|\leq 1$, and it is also symmetric we have $|P_1(x-z+w(x,z,y,t))|\leq CP_1(z-x)$. Therefore, using, one after another, Tonelli's Theorem, assumption~\eqref{A0},~\eqref{a.5} and~\eqref{q4}, we arrive at
\begin{align*}
    \int_0^T\intr|\theta\xt|&\ints|\Lambda U(x,y,t)|K(y)\dy\dx\dt\\
    &\leq C\ints |y|^{2}K(y)\dy\int_{\varepsilon}^{T'}\intr (P_1*|\theta(\cdot,t)|)(z)\,{\rm d}|\mu_{0}|(z)\dt\\
    &\leq C(T'-\varepsilon)\intr P_1(z)\,{\rm d}|\mu_{0}|(z)<\infty.
\end{align*}

\noindent (iii) If $\mu_0$ is nonnegative, then $U$ is nonnegative. Uniqueness in the class of nonnegative solutions satisfying~\eqref{energyH} is then granted by Theorem~\ref{thm:main.K}.
\end{proof}

%%%%%%%%%%%%%%%%%%%%%%%%%%%%%%%%%%%%%%%%%%%%%%%%%%%%%%%%%%%%%%%%%%%%%%%%%%%%%%%%%%%%%%%%%
%%%%%%%%%%%%%%%%%%%%%%%%%%%%%%%%%%%%%%%%%%%%%%%%%%%%%%%%%%%%%%%%%%%%%%%%%%%%%%%%%%%%%%%%%
\section{General purely nonlocal operators}\label{section4}
\setcounter{equation}{0}

We have developed in Section~\ref{section3} a quite complete Widder-type theory for operators of the form $\mathcal{L}=\mathcal{L}_K$ with $K$ satisfying~\eqref{A0}--\eqref{A2}, plus~\eqref{A3} for some results. A careful inspection of the proofs shows that they do not use these assumptions directly, but their implications for the heat kernel associated to the operator, that could be satisfied by more general operators. Our aim in this section is to rewrite the theorems in terms of this more general conditions, so that they may include a wider class of purely nonlocal operators $\mathcal{L}_\nu$ of the form~\eqref{a.Lnu} with $\nu$ satisfying~\eqref{eq:Levy.condition}. In order to check that this is not a useless generalization, we will show that these new versions of the theorems encompass an important class of anisotropic operators that were not covered by the theory in Section~\ref{section3}.

In addition to the L\'evy condition~\eqref{eq:Levy.condition}, we will always assume that the L\'evy process associated to $\nu$ has a heat kernel $P$ satisfying
\begin{equation}\label{eq:continuity.positivity}
    P_t\in C(\mathbb{R}^d),\quad P_t>0\quad\text{for all }t>0.
\end{equation}
Thus, we are leaving out, for instance, the case in which $\textup{d}\nu=J\,\textup{d}x$ with $J\in L^1(\mathbb{R}^d)$, for which there is a transition density but not a heat kernel, or the operator $\log(1-\Delta)$, with Fourier multiplier $\log(1+|\xi|^2)$, that has a positive heat kernel, though not with the desired continuity until some time has passed.

Observe that, since we are assuming that $\nu$ is symmetric, so it is $P_t$ for all $t>0$.

%%%%%%%%%%%%%%%%%%%%%%%%%%%%%%%%%%%%%%%%%%%%%%%%%%%%%%%
\subsection{Very weak solutions}

We will be able to recover all the important results for very weak solutions, with less restrictive assumptions than in Section~\ref{section3}.
%%%%%%%%%%%%%%%%%%%%%%%%%%%%%%%%%%%%%%%%%%
\subsubsection{Alternative definitions of very weak solution}

A quick inspection of the proof of Lemma~\ref{a.17} shows that it can be easily adapted to show the following more general result, that will be helpful to establish the equivalence of our two definitions of very weak solution if $u\in\li$.

\begin{lemma}\label{lem:Lnu.theta}
    Let $\mathcal{L}=\mathcal{L}_\nu$ of the form~\eqref{a.Lnu} with $\nu$ satisfying~\eqref{eq:Levy.condition}. Assume that it has associated a heat kernel~$P$ satisfying~\eqref{eq:continuity.positivity} and
    \begin{equation}\label{eq:conds.general.L}
        \int_{\{|y|\geq1\}}P_{1} (x-y)\,\textup{d}\nu(y)\le C P_1(x)\quad\text{for all } x\in\mathbb{R}^d
    \end{equation}
    for some constant $C$. If $\theta\in C^{\infty}_{\textup{c}}(\RN\times [\delta,\tau])$, $0\leq \delta<\tau\leq T$, there is a constant $C_{\theta}$ such that
    \begin{equation}\label{an7}	
        |\Lnu \theta\xt|\leq \intr |\Lambda \theta (x,y,t)| \dnu(y)\leq C_{\theta}P_{1}(x),\quad\xt\in \Rn\times[\delta,\tau].
    \end{equation}
\end{lemma}

\begin{remark}\label{an.prop1}
    We proved in Section~\ref{section3}, under certain assumptions on $K$, that, given $\theta\in C^{\infty}_{\textup{c}}(\RN\times [\delta,\tau])$, $0\leq \delta<\tau\leq T$, there are constants $C,R_0>0$ depending on $\theta$ such that $|\mathcal{L}_K\theta(x)|\ge CP_1(x)$ if $|x|\ge R_0>0$. This was helpful to prove that very weak solutions belong to $L^1_{\textup{loc}}((0,T);L^1_{P_1})$. Unfortunately, hypotheses~\eqref{eq:Levy.condition},~\eqref{eq:continuity.positivity} and~\eqref{eq:conds.general.L} are not enough to ensure the validity of such an inequality. As a counterexample (in spatial dimension $d=2$) we consider the operator $\mathcal{L}_\nu=(-\partial^2_{x_1x_1})^{1/2}+(-\partial^2_{x_2x_2})^{1/2}$, which satisfies the three hypotheses, and the function $\theta(x)=\theta_1(x_1)\theta_2(x_2)$, with $\theta_1,\theta_2\in C^\infty_{\textup{c}}(\mathbb{R})$, so that $\theta\in C^\infty_{\textup{c}}(\mathbb{R}^2)$. If $\operatorname{supp}(\theta)\subset B_R(0)$ for some $R>0$, then
    \begin{equation*}
        \mathcal{L}_\nu\theta(x)=\theta_2(x_2)(-\partial^2_{x_1x_1})^{1/2}\theta_1(x_1)+\theta_1(x_1)(-\partial^2_{x_2x_2})^{1/2}\theta_2(x_2)=0\quad \text{if }|x_1|>R\text{ and }|x_2|>R.
    \end{equation*}
\end{remark}

Lemma~\ref{lem:Lnu.theta} is enough to show that the integrability condition~\eqref{a.2} is satisfied if $u\in L^{1}([\delta,\tau];L^{1}_{P_1})$.

\begin{proposition} \label{anprop1}
    Let $\mathcal{L}=\mathcal{L}_\nu$ and $P$ as in Lemma~\ref{lem:Lnu.theta}. If $u\in L^{1}([\delta,\tau];L^{1}_{P_1})$, $0\leq \delta<\tau\leq T$, then
	\begin{equation*}
        \int_{\delta}^{\tau}\intr|u(x,t)||\Lk \theta(x,t)|\dx\dt<\infty\quad\textrm{for all }\theta\in C^{\infty}_{\textup{c}}(\RN\times [\delta,\tau]).
	\end{equation*}
    As a consequence, if  $u\in \luloc$, $u$ satisfies~\eqref{a.2}.
\end{proposition}

With this integrability condition at hand, we can now prove the equivalence of both definitions of very weak solution if $u\in\li$, just following step by step the proof of Proposition~\ref{equivDefvwProp}.

\begin{proposition}\label{an.equivDefvwProp}
    Let $\mathcal{L}=\mathcal{L}_\nu$ and $P$ as in Lemma~\ref{lem:Lnu.theta}, $\mu_{0}$ a Radon measure and $u\in\li$.
    	
    \noindent{\rm (i)} If $u$ satisfies~\eqref{v3.10}, then it is a very weak solution of~\eqref{a.nonlocalheatequation} with initial trace~$\mu_0$.
    	
    \noindent{\rm(ii)} If $u$ is a very weak solution to~\eqref{a.nonlocalheatequation} with initial trace $\mu_{0}$, then it  satisfies~\eqref{eq:space.alternative.definition}--\eqref{v3.10}.
\end{proposition}

%%%%%%%%%%%%%%%%%%%%%%%%%%%%%%%%%%%%%%%%%%%%%
\subsubsection{Uniqueness for very weak solutions}

We have now all the tools to prove uniqueness for very weak solutions in the sense of~\eqref{eq:space.alternative.definition}--\eqref{v3.10} under two new assumptions on the kernel, that have already appeared in Section~\ref{section3}. This is done by following the steps in the proof of Theorem~\ref{UniquenessSoria}, where now $\varphi$ is a solution of the backward problem $\partial_{t}\varphi-\Lk_{\nu}\varphi=\tilde{\theta}$ and $B$ is the bilinear form
\begin{equation}\label{eq:def.Bnu}
    B_{\nu}(u,v)(x)=\intr\big(u(x)-u(x-y)\big)\big(v(x)-v(x-y)\big)\dnu(y),
\end{equation}
taking profit of the nonlocal Leibniz's formula
\begin{equation*}
    B_\nu(uv)=u\mathcal{L}_\nu v+v\mathcal{L}_\nu u-B_\nu(u,v).
\end{equation*}

\begin{theorem}\label{anUniquenessSoria}
    Let $\mathcal{L}=\mathcal{L}_\nu$ of the form~\eqref{a.Lnu} with $\nu$ satisfying~\eqref{eq:Levy.condition}. Assume that it has associated a heat kernel~$P$ satisfying~\eqref{eq:continuity.positivity}, \eqref{eq:conds.general.L}, \eqref{Atotal2} and~\eqref{eq:Pz.Px}. Let $\mu_{0}$ be a Radon measure. If $u_1$ and $u_2$ satisfy~\eqref{eq:space.alternative.definition}--\eqref{v3.10}, then, $u_1=u_2$ a.e.\,in $\mathbb{R}^d\times(0,T)$.
\end{theorem}

Since by Proposition~\ref{an.equivDefvwProp} a very weak solution $u\in \li$ with initial datum a Radon measure $\mu_{0}$ satisfies~\eqref{v3.10}, we obtain the following corollary.

\begin{corollary}[Uniqueness for  very weak solutions within $\li$] \label{anUniquenessAcotadas}
    Let $\mathcal{L}=\mathcal{L}_\nu$ and $P$ as in Theorem~\ref{anUniquenessSoria}. There is at most one nonnegative very weak solution of \eqref{a.nonlocalheatequation}	 belonging to $\li$ with a given initial trace.
\end{corollary}

We consider now the class of nonnegative very weak solutions that belong to $\luloc$. As in Section~\ref{section3}, the strategy to prove uniqueness in this class is to show that very weak solutions within it belong to $\li$. Uniqueness will then be inherited from Corollary~\ref{anUniquenessAcotadas}.

The required boundedness is obtained following the proof of Lemma~\ref{lemmatrace}, that depends only on the existence of a function $\phi\in C^2(\mathbb{R}^d)$ satisfying~\eqref{a.46}--\eqref{Lphi}, with $\mathcal{L}=\mathcal{L}_\nu$ and $B=B_\nu$, remembering that~\eqref{an7} holds. Thus we have the following result.

\begin{lemma}\label{thm:uniqueness.very.weak.nu}
    Let $\mathcal{L}=\mathcal{L}_\nu$ of the form~\eqref{a.Lnu} with $\nu$ satisfying~\eqref{eq:Levy.condition} and $B=B_\nu$ the bilinear form defined in \eqref{eq:def.Bnu}. Assume that~$\mathcal{L}_\nu$ has associated a heat kernel~$P$ satisfying~\eqref{eq:continuity.positivity} and that there is a function $\phi\in C^{2}(\Rn)$ satisfying~\eqref{a.46}--\eqref{Lphi}.

    \noindent\textup{(i)} If $u$ is a nonnegative very weak solution to~\eqref{a.nonlocalheatequation} in $\luloc$, there is a constant $c>0$ such that~\eqref{lematecnicodesigualdad} holds.

    \noindent\textup{(ii)} If $u\in\luloc$ is a nonnegative very weak solution of~\eqref{a.nonlocalheatequation}, then  $u\in\li$.
\end{lemma}

\begin{remark}
    Lemma~\ref{lemmatrace} does not have $u\in \luloc$ among its hypotheses, though this fact is used in its proof. The reason is that when  we consider $\Lk_K$ instead of $\Lk_{\nu}$ any very weak solution has this integrability, as noted in Remark~\ref{remark3}. But this is not guaranteed \emph{a priori} in the case of more general operators.
\end{remark}

Mimicking the proof of Proposition~\ref{prop:regularity.in.time} we get also some regularity in time for general operators.

\begin{proposition}
    Let $\mathcal{L}=\mathcal{L}_\nu$ and $P$ as in Lemma~\ref{thm:uniqueness.very.weak.nu}, and let $u$ be a nonnegative very weak solution to~\eqref{a.nonlocalheatequation}.

    \noindent\textup{(i)} Let $\phi\in C^2(\mathbb{R}^d)$ satisfying~\eqref{a.46}--\eqref{Lphi}. There is a constant $C>0$ such that~\eqref{eq:continuity.phi} holds.

    \noindent\textup{(ii)} Given $\psi\in C^2_{\textup{c}}(\mathbb{R}^d)$, there is a constant $C_\psi>0$ such that~\eqref{eq:continuity.psi} holds.
\end{proposition}

The combination of Lemma~\ref{thm:uniqueness.very.weak.nu} (ii) and Corollary~\ref{anUniquenessAcotadas} yields the desired uniqueness result for nonnegative very weak solutions.

\begin{corollary}[Uniqueness for nonnegative very weak solutions]\label{UniquenessNonnegative.nu}
    Let $\mathcal{L}=\mathcal{L}_\nu$ and $P$ as in Theorem~\ref{anUniquenessSoria}, and $B=B_\nu$ the bilinear form defined in \eqref{eq:def.Bnu}. If there is a function $\phi\in C^{2}(\Rn)$ satisfying~\eqref{a.46}--\eqref{Lphi}, then there is at most one nonnegative very weak solution of~\eqref{a.nonlocalheatequation} belonging to $\luloc$ with a given initial trace.
\end{corollary}

%%%%%%%%%%%%%%%%%%%%%%%%%%%%%%%%%%%%%
\subsubsection{Existence of an initial trace}

The existence of an initial trace satisfying the integral growth condition~\eqref{eq:growth.initial.trace} for nonnegative very weak solutions is based on the estimates~\eqref{lematecnicodesigualdad}; see the proof of Theorem~\ref{theoremtraces}. Following that proof we obtain easily our next result.

\begin{theorem}[Existence of initial trace integrable against $P_{1}$] \label{antheoremtraces}
    Assume the same hypotheses as in Lemma~\ref{thm:uniqueness.very.weak.nu}. If $u\in\luloc$ is a nonnegative very weak solution to~\eqref{a.nonlocalheatequation}, then there is a nonnegative Radon measure $\mu_{0}$ such that~\eqref{a.1} holds; that is, $\mu_0$ is the initial trace of $u$.  Moreover, $\mu_0$ satisfies the integral growth condition\eqref{eq:growth.initial.trace}.
\end{theorem}

%%%%%%%%%%%%%%%%%%%%%%%%%%%%%%%%%%%%%
\subsubsection{Existence of very weak solutions}

The existence of very weak solutions, for suitable initial data, is obtained by mimicking the proof of Theorem~\ref{prop2}, with two new hypotheses.

\begin{theorem}\label{anprop2}
     Let $\mathcal{L}=\mathcal{L}_\nu$ of the form~\eqref{a.Lnu} with $\nu$ satisfying~\eqref{eq:Levy.condition}. Assume that it has associated a heat kernel~$P$ that is a very weak solution of \eqref{a.nonlocalheatequation} and satisfies~\eqref{eq:continuity.positivity},~\eqref{Atotal2} and~\eqref{eq:comparison.translations}. Let $\mu_{0}$ be a Radon measure with the admissible growth at  infinity~\eqref{q4}. Then, for each $T>0$, the function $U$ given in~\eqref{q5} is a very weak solution of \eqref{a.nonlocalheatequation} in $C(\Rn\times(0,T))\cap\li$ with initial trace~$\mu_{0}$.
\end{theorem}

As a corollary to Theorem~\ref{anprop2} and the uniqueness results Corollary~\ref{anUniquenessAcotadas} and Corollary~\ref{UniquenessNonnegative.nu}, solutions are given by a representation formula.

\begin{corollary}[Representation formula for very weak solutions]\label{anCorRepresentationFormulaVW}
   Let $\mathcal{L}=\mathcal{L}_\nu$ of the form~\eqref{a.Lnu} with $\nu$ satisfying~\eqref{eq:Levy.condition}. Assume that it has associated a heat kernel~$P$ satisfying~\eqref{eq:continuity.positivity}, \eqref{eq:conds.general.L}, \eqref{Atotal2},~\eqref{eq:comparison.translations} and~\eqref{eq:Pz.Px}, and let $\mu_{0}$ be a Radon measure satisfying~\eqref{q4}. Let $u$ be a very weak solution to~\eqref{a.nonlocalheatequation} with trace $\mu_{0}$. If either $u\in \li$, or $u$ is nonnegative and in addition there is a function $\phi\in C^{2}(\Rn)$ satisfying~\eqref{a.46}--\eqref{Lphi}, then $u$ is given by representation formula~\eqref{representationformula}.
\end{corollary}

%%%%%%%%%%%%%%%%%%%%%%%%%%%%%%%%%%%%%%%%%%%%%%%%%%%%%%%
\subsection{Classical solutions}

It is turn now to analyze classical solutions for general operators. We have to tailor the energy condition~\eqref{energyH} on the solutions to the general situation. Thus, we will look for solutions satisfying
\begin{equation}\label{energyH'}\tag{H'}	
        \ints|\Lambda u(x,y,t)|\dnu(y)\in L^{1}_{\textup{loc}}(\Rn\times (0,T)).
\end{equation}

%%%%%%%%%%%%%%%%%%%%%%%%%%%%%%%%%%%%%%%
\subsubsection{Uniqueness of classical solutions.}

The clue for uniqueness is to prove that classical solutions satisfying the energy condition~\eqref{energyH'} are very weak solutions, so that we can apply to them the results that we have already obtained for the latter class. This is done easily following the proof of Proposition~\ref{cprop1}, adding assumption~\eqref{eq:conds.general.L}, so that we have the following result.

\begin{proposition}\label{anprop3}
    Let $\mathcal{L}=\mathcal{L}_\nu$ and $P$ as in Lemma~\ref{lem:Lnu.theta}, and $u$ a classical solution to~\eqref{a.nonlocalheatequation}. If $u\in\luloc$ and satisfies~\eqref{energyH'}, then it is a very weak solution to the same equation.
\end{proposition}

As a consequence of Proposition~\ref{anprop3}, Corollary~\ref{anUniquenessAcotadas} and Corollary~\ref{anCorRepresentationFormulaVW}, we obtain uniqueness for classical solutions in $\li$, and a representation formula if there is existence in this class.

\begin{corollary}\label{anCorRepresentativeFormulaClassicalAcotadas}
    Let $\mathcal{L}=\mathcal{L}_\nu$ and $P$ as in Theorem~\ref{anUniquenessSoria}.

    \noindent{\rm (i)} There is at most one classical solution $u\in\li$ to~\eqref{a.nonlocalheatequation} satisfying~\eqref{energyH'} with initial trace a given Radon measure.

    \noindent{\rm (ii)} If there exists a classical solution $u\in\li$ to~\eqref{a.nonlocalheatequation} satisfying~\eqref{energyH'} with initial trace a measure $\mu_0$ verifying~\eqref{q4}, then $u$ is given by the representation formula~\eqref{representationformula}.
\end{corollary}

We move on now to the class of nonnegative classical solutions $u$ to~\eqref{a.nonlocalheatequation} satisfying~\eqref{energyH'}. Following the strategy of Section~\ref{section3}, we will prove that they satisfy inequality~\eqref{a.16}. This will imply that they belong to $\luloc$, so that they are nonnegative very weak solutions because of Proposition~\ref{anprop3}. Uniqueness and a representation formula will then follow from the results for this class of solutions.

Inequality~\eqref{a.16} follows easily if we have a comparison principle, $P\in C^{2,1}_{x,t}(\mathbb{R}^d\times(0,\infty))$ and $P_1\in C_0(\mathbb{R}^d)$. An inspection of the proof of Lemma~\ref{a.mp} shows that such a comparison principle will be available for $\mathcal{L}_\nu$ if the L\'evy measure $\nu$ satisfies that
\begin{equation}\label{eq:cond.nu.maximum.principle}
	\nu\big((B_{R}(0))^{\textup{c}}\big)>0\quad\textrm{for all }R\geq R_{0}\text{ for some }R_0>0.
\end{equation}

\begin{lemma}[A maximum principle]\label{anmp3}
    Let $\mathcal{L}=\mathcal{L}_\nu$ of the form~\eqref{a.Lnu} with $\nu$ satisfying~\eqref{eq:Levy.condition} and~\eqref{eq:cond.nu.maximum.principle}. Let $\Omega\subset\RN$ be a non-empty, open, bounded  and smooth set, and  $v\in C(\Omega\times[0,T))$ such that $\partial_t v\in C(\bar{\Omega}\times(0,T))$ and $\Lk v\xt$ is defined for all $\xt\in\RN\times (0,T)$. If~\eqref{eq:ineq.MP} holds, then $v\leq 0$ for all $x\in\mathbb{R}^d\times[0,T)$.
\end{lemma}

We can now show the validity of~\eqref{a.16}, following the proof of~Lemma~\ref{a.cor4}.

\begin{lemma}\label{a.cor4.nu}
    Let $\mathcal{L}=\mathcal{L}_\nu$ of the form~\eqref{a.Lnu} with $\nu$ satisfying~\eqref{eq:Levy.condition} and~\eqref{eq:cond.nu.maximum.principle}. Assume that  it has associated a heat kernel~$P$ satisfying~\eqref{eq:continuity.positivity} and~\eqref{Atotal2}, and such that $P\in C^{2,1}_{x,t}(\mathbb{R}^d\times(0,\infty))$ and $P_1\in C_0(\mathbb{R}^d)$. If $u$ is a nonnegative classical solution of the nonlocal heat equation~\eqref{a.nonlocalheatequation} with initial data a nonnegative Radon measure $\mu_0$, then~\eqref{a.16} holds.
\end{lemma}

\begin{remark}\label{remark1.nu}	
    Since $\mu_{0}$ is a nonnegative measure, $U(x,t)$ is well defined as long as the nonnegative classical solution $u$ exists, thanks to inequality \eqref{a.16}. Moreover, thanks to the symmetry of $P_1$,~\eqref{eq:growth.initial.trace} holds, since $\intr P_{1}\dmuo\leq u(0,1)<\infty$.
\end{remark}

The desired integrability is now obtained easily, proceeding as in the proof of Corollary~\ref{a.cor2}.

\begin{corollary}\label{a.cor2.nu}
    Let $\mathcal{L}=\mathcal{L}_\nu$ and $P$ as in Lemma~\ref{a.cor4.nu}. If $u$ is a nonnegative classical solution of the nonlocal heat equation \eqref{a.nonlocalheatequation} with initial trace a nonnegative Radon measure $\mu_{0}$, then $u\in L^{1}_{\rm loc}([0,T);\lp)$.
\end{corollary}

We have now all the elements to obtain a representation formula (implying uniqueness) for nonnegative classical solutions, following the proof of Theorem~\ref{thm:main.K}.

\begin{theorem} \label{thm:main.nu}
    Let $\mathcal{L}=\mathcal{L}_\nu$ of the form~\eqref{a.Lnu} with $\nu$ satisfying~\eqref{eq:Levy.condition} and~\eqref{eq:cond.nu.maximum.principle}, and $B=B_\nu$ the bilinear form defined in \eqref{eq:def.Bnu}. Assume that~$\mathcal{L}_\nu$ has associated a heat kernel~$P$ satisfying~\eqref{eq:continuity.positivity}, \eqref{eq:conds.general.L}, \eqref{Atotal2},~\eqref{eq:comparison.translations} and~\eqref{eq:Pz.Px}, and such that $P\in C^{2,1}_{x,t}(\mathbb{R}^d\times(0,\infty))$ and $P_1\in C_0(\mathbb{R}^d)$, and that there is a function $\phi\in C^{2}(\Rn)$ satisfying~\eqref{a.46}--\eqref{Lphi}. Let $u$ be a nonnegative classical solution to the nonlocal heat equation~\eqref{a.nonlocalheatequation} with initial trace a nonnegative Radon measure $\mu_{0}$ satisfying~\eqref{q4}. If~\eqref{energyH'} holds,  then $u$ is given by the representation formula~\eqref{representationformula}.
\end{theorem}

%%%%%%%%%%%%%%%%%%%%%%%%%%%%%%%%%%%%%%%%%%%%%
\subsubsection{Existence of classical solutions}

Let $\mu_0$ a Radon measure (that may have sign changes) satisfying the integral growth condition~\eqref{q4}. We have already obtained conditions ensuring the function $U(\cdot,t)=P_t*\mu_0$ to be a very weak solution to~\eqref{a.nonlocalheatequation} having $\mu_0$ as initial trace; see~Theorem~\ref{anprop2}. Mimicking the proof of Theorem~\ref{Uclassical} we get conditions showing that it is also a classical solution. Let us note that $U$ belongs to the class $\li$ in which there is uniqueness.

Before proceeding, let us make a short technical detour. The starting point of our desired existence result is Theorem~\ref{anprop2}, that requires $P$ to be a very weak solution to~\eqref{a.nonlocalheatequation}. But our proof needs $P$ to be a classical solution also. Do we need both assumptions? Fortunately not. As can be easily seen following the proof Lemma~\ref{PisaVeryWeaksolution.cor}, under some additional mild hypothesis that we need anyway, if $P$ is a classical solution, it is also a very weak solution.

\begin{lemma}
    Let $\mathcal{L}=\mathcal{L}_\nu$ of the form~\eqref{a.Lnu} with $\nu$ satisfying~\eqref{eq:Levy.condition}. Assume that it has associated a heat kernel~$P$ that is a classical solution to~\eqref{a.nonlocalheatequation} and satisfies~\eqref{eq:continuity.positivity} and~\eqref{Atotal2}. If
    \begin{equation*}
        \int_{\mathbb{R}^d}|\Lambda P(x,y,t)|\,\textup{d}\nu(y)\le C P_1(x)\quad\text{for all }\xt\in \Rn\times(\varepsilon,T),
    \end{equation*}
    then $P$ is also a very weak solution to~\eqref{a.nonlocalheatequation}.
\end{lemma}

We can now state the existence theorem.

\begin{theorem}\label{Uclassical.nu}
    Let $\mathcal{L}=\mathcal{L}_\nu$ of the form~\eqref{a.Lnu} with $\nu$ satisfying~\eqref{eq:Levy.condition}. Assume that it has associated a heat kernel~$P$ that is a classical solution to~\eqref{a.nonlocalheatequation} and satisfies~\eqref{eq:continuity.positivity},~\eqref{Atotal2} and~\eqref{eq:comparison.translations}. Let $\mu_{0}$ a Radon measure satisfying the growth condition~\eqref{q4}, and $U$ the very weak solution to~\eqref{a.nonlocalheatequation} given in~\eqref{q5}. Suppose that for each $T>0$ and $\varepsilon\in(0,T)$ there is a constant $C>0$ such that
    \begin{equation*}\label{eq:conditions.existence.classical.nu}
        |\partial_{t} P_{t}(x)|\leq CP_{1}(x), \quad \int_{\mathbb{R}^d}|\Lambda P(x,y,t)|\,\textup{d}\nu(y)\le C P_1(x)\quad\text{for all }\xt\in \Rn\times(\varepsilon,T).
    \end{equation*}

    \noindent\textup{(i)} Then, $U$ is a classical solution to~\eqref{a.nonlocalheatequation} within $\li$ and with initial trace $\mu_{0}$.

    \noindent\textup{(ii)} If, in addition, for each $T>0$ and $\varepsilon\in(0,T)$ there is a constant $C>0$ such that
    \begin{equation}\label{eq:stronger.conditions.existence.classical.nu}
        |D^{2} P_{t}(x)|\leq CP_{1}(x)\quad\text{for all }\xt\in \Rn\times(\varepsilon,T),
    \end{equation}
    then $U$ is the unique classical solution to~\eqref{a.nonlocalheatequation} with initial trace $\mu_{0}$ that satisfies~\eqref{energyH'} and belongs to $\li$.

    \noindent\textup{(iii)} If, in addition to all the above, $\nu$ satisfies~\eqref{eq:cond.nu.maximum.principle},~$P$ satisfies~\eqref{eq:conds.general.L},~\eqref{eq:Pz.Px}, $P\in C^{2,1}_{x,t}(\mathbb{R}^d\times(0,\infty))$ and $P_1\in C_0(\mathbb{R}^d)$, and there is a function $\phi\in C^{2}(\Rn)$ satisfying~\eqref{a.46}--\eqref{Lphi}, where $B=B_\nu$ is the bilinear form defined in \eqref{eq:def.Bnu}, and $\mu_0$ is nonnegative, then $U$ is the unique nonnegative classical solution to~\eqref{a.nonlocalheatequation} with initial trace $\mu_0$ that satisfies~\eqref{energyH'}.
\end{theorem}

%%%%%%%%%%%%%%%%%%%%%%%%%%%%%%%%%%%%%%%%%%%%%%%%%%%%%%%
\subsection{A family of anisotropic operators to which the theory applies}

Our goal now is to show that the general theory we have just presented includes many interesting examples that were not covered by the results of Section~\ref{section3}. To this aim, we will introduce a family of anisotropic operators $\mathcal{L}_\nu$ that satisfy all the assumptions in the above results.

%%%%%%%%%%%%%%%%%%%%%%%%%%%%%%%%%%%%%%%%%%%%%%%%%%%
\subsubsection{The L\'evy measure and the operator}

We consider a family of L\'evy kernels $\{K_j\}_{j=1}^\ell$, with $\ell\in\{1,\dots,d\}$, where each $K_j:\mathbb{R}^{d_j}\to(0,\infty)$ satisfies  assumptions~\eqref{A0}--\eqref{A2} (and~\eqref{A3} for some results)  with $d=d_j$, and $\sum_{j=1}^\ell d_j=d$. Given $y\in\mathbb{R}^d$, we write it as $y=(y_1,\dots,y_\ell)$, with $y_j\in\mathbb{R}^{d_j}$, $j\in\{1,\dots,\ell\}$. The (in general anisotropic) measure defining our operator $\mathcal{L}_\nu$ is
\begin{equation}\label{an2}
    {\rm d}\nu(y):=\sum_{j=1}^\ell K_j(y_j)\,{\rm d}y_j\prod_{k=1,k\not = j}^\ell\delta_{0}^{d_k}(y_k),\quad y \in\Rn,
\end{equation}
where $\delta_{0}^{d_k}$ is the Dirac delta in $\mathbb{R}^{d_k}$. It is easily checked that this is a L\'evy measure.

\begin{proposition}\label{prop:nu.levy}
    If $K_j$, $j\in\{1,\dots,\ell\}$, satisfy~\eqref{A0} with $d=d_j$, then $\nu$ in~\eqref{an2} satisfies~\eqref{eq:Levy.condition}.	
\end{proposition}

\begin{proof}
Since $K_j$ is a L\'evy kernel in $\mathbb{R}^{d_j}$ for all $j\in\{1,\dots,\ell\}$,
\begin{equation*}
	\intr (1\wedge |y|^2)\,{\rm d}\nu (y)=\sum_{j=1}^\ell \int_{\mathbb{R}^{d_j}}(1\wedge |y_j|^2)K_j(y_j)\,{\rm d}y_j<\infty.\qedhere
\end{equation*}
\end{proof}
The L\'evy operator $\mathcal{L}_\nu$ defined by $\nu$ can be written as $\Lk_{\nu}=\sum_{j=1}^\ell\Lk_{K_j}$, where each of the operators $\mathcal{L}_{K_j}$ is of the form~\eqref{a.Lk} and acts only on the subspace $\mathbb{R}^{d_j}$; to be specific,
\begin{equation*}	
    \mathcal{L}_{K_j} u(x)=\int_{\mathbb{R}^{d_j}}\Big(u(x)-\frac{u(x_1,\dots,x_j+y_j,\dots, x_\ell)+u(x_1,\dots,x_j-y_j,\dots, x_\ell)}{2}\Big)K_j(y_j)\,\textup{d}y_j.
\end{equation*}
In what follows we write for short $\mathcal{L}^j$ to denote $\mathcal{L}_{K_j}$.

The measure $\nu$ in~\eqref{an2} satisfies the condition needed to apply the comparison principle.
\begin{proposition}
    Let $\nu$ as in Proposition~\ref{prop:nu.levy}. Then $\nu$ satisfies~\eqref{eq:cond.nu.maximum.principle}.	
\end{proposition}

\begin{proof}
Since $K_j>0$ for all $j\in\{1,\dots,\ell\}$, then $\displaystyle\nu\big((B_{R}(0))^{\textup{c}}\big) =\sum_{j=1}^\ell\int_{\{|y_j|>R\}} K_j(y_j)\,\textup{d}y_j>0$.
\end{proof}

Just for the record, $-\Lk_{\nu}$ is the infinitesimal generator of the L\'evy process $X_{t}=(Y^{1}_{t},\dots,Y^\ell_{t})$, $t\ge0$,
where each $\big(Y^j_{t}\big)_{t\ge0}$, $j\in\{1,\dots,\ell\}$, is a L\'evy process with infinitesimal generator $-\Lk^j$.

%%%%%%%%%%%%%%%%%%%%%%%%%%%%%%%%
\subsubsection{The heat kernel}

As explained in Section 2, the Fourier symbol of the operator $\Lk^j$ is
\begin{equation*}
	m^j(\xi_j)=\int_{\mathbb{R}^{d_j}} (1-\cos\langle\xi_j,y_j\rangle)\,{\rm d}y_j,\quad\xi_j\in\mathbb{R}^{d_j},\quad j\in\{1,\dots,\ell\}.
\end{equation*}
and the corresponding heat kernel $P^j(x_j,t)=P^j_{t}(x_j)$, $(x_j,t)\in\mathbb{R}^{d_j}\times (0,\infty)$, is given by
\begin{equation*}
    \mathcal{F}(P^j_{t})(\xi_j)=e^{-tm^j(\xi_j)},\quad(\xi_j,t)\in\mathbb{R}^{d_j}\times (0,\infty),\quad j\in\{1,\dots,\ell\}.
\end{equation*}
Using L\'evy-Khintchine's formula we obtain the L\'evy-Khintchine exponent $m$ for the operator $\mathcal{L}_\nu$,
\begin{equation*}\label{a.18}	
    m(\xi)=\!\intr (1-\cos \langle \xi,y\rangle)\,{\rm d}\nu(y)=\sum_{j=1}^\ell\!\int_{\mathbb{R}^{d_j}} (1-\cos\langle\xi_j,y_j\rangle)\,{\rm d}y_j=\sum_{j=1}^\ell m^j(\xi_j),\quad\!\!\xi=(\xi_1,\dots,\xi_\ell)\in\Rn,
\end{equation*}
with $\xi_j\in\mathbb{R}^{d_j}$ for $j\in\{1,\dots,\ell\}$. Therefore, the heat kernel $P_t$ associated to the nonlocal heat equation \eqref{a.nonlocalheatequation} for this operator is given by
\begin{equation}\tag{AHK}\label{anHK}
    \begin{aligned}
        P_t(x)&=\intr e^{-tm(\xi)}e^{\langle\xi, x\rangle}\,{\rm d}\xi=\intr e^{-t\sum_{j=1}^\ell m^j(\xi_j)}e^{i\sum_{j=1}^\ell\langle\xi_j, x_j\rangle}\,{\rm d}\xi_1\cdots{\rm d}\xi_\ell\\	
        &=\prod_{j=1}^\ell\int_{\mathbb{R}^{d_j}} e^{-tm^j(\xi_j)}e^{i\langle\xi_j, x_j\rangle}\,{\rm d }\xi_j=\prod_{j=1}^\ell P^j_{t}(x_j),\quad\xt\in\Rn\times (0,\infty).
    \end{aligned}
\end{equation}
This special structure in separate variables, combined with the special form of the L\'evy measure, is the clue to prove all the desired properties for the kernel, that will be inherited in most of the cases from the corresponding ones for the factors $P^j$, $j\in\{1,\dots,\ell\}$.

\begin{proposition}
    Let $\mathcal{L}_\nu$ of the form~\eqref{a.Lnu} with $\nu$ given by~\eqref{an2} and $K_j$, $j\in\{1,\dots,\ell\}$, satisfying~\eqref{A0}--\eqref{A2} with $d=d_j$. The corresponding heat kernel $P$ belongs to $C^{2,1}_{x,t}(\mathbb{R}^d\times(0,\infty))\cap C_0(\mathbb{R}^d)$,  satisfies~\eqref{Atotal2},~\eqref{eq:comparison.translations},~\eqref{eq:Pz.Px}, \eqref{eq:continuity.positivity} and \eqref{eq:conds.general.L}, and solves the equation~\eqref{a.nonlocalheatequation} in the classical and in the very weak sense.
\end{proposition}

\begin{proof}
All the properties, except~\eqref{eq:conds.general.L}, have already been proved for the factors $P^j$. The corresponding properties for~$P$ can be easily recovered from the results for the factors thanks to the special form of the L\'evy measure. We omit the details.

The proof of~\eqref{eq:conds.general.L} also uses the special form~\eqref{an2} of the L\'evy measure $\nu$ and the product structure~\eqref{anHK} of $P_1$. Indeed,
\begin{equation*}
    \int_{\{|y|\geq1\}}P_{1} (x-y)\,\textup{d}\nu(y)=\sum_{j=1}^\ell\underbrace{\int_{\{|y_j|\ge 1\}}P_{1}^j (x_j- y_j)|K^j(y_j)\,\textup{d}y_j}_{\textup{I}_j(x_j)}\,\prod_{k=1,k\not = j}^\ell P_1^k(x_k).
\end{equation*}
But we can control the tail of $K^j$ by the tail of $P^j_1$; see~\eqref{Atotal}. Thus, using also the semigroup property~\eqref{eq:semigroup.property}, for some constant $C_j>0$,
\begin{equation*}
    \textup{I}_j(x_j)\leq C_j\int_{\{|y_j|\geq 1\}}P_{1}^j (x_j- y_j)P_1^j(y_j)\,\textup{d}y_j\leq  C_j\int_{\mathbb{R}^{d_j}}P_{1}^j (x_j- y_j)P_1^j(y_j)\,\textup{d}y_j=C_jP_2^j(x_j).
\end{equation*}
On the other hand, because of~\eqref{Atotal2}, there exists $C_j>0$ such that $P^j_2\leq C_jP^j_1$, and we conclude that $\textup{I}_j\le C_jP^j_1$ for some constant $C_j>0$ for each $j\in \{1,\dots,\ell\}$. Therefore,
\begin{equation*}
	 \int_{\{|y|\geq1\}}P_{1} (x-y)\,\textup{d}\nu(y)\le\sum_{j=1}^\ell C_jP^j_1(x_j)\prod_{k=1,k\not = j}^\ell P_1^k(x_k)=\sum_{j=1}^\ell C_j P_1(x).
\end{equation*}

To check that $P$ is a classical solution to~\eqref{a.nonlocalheatequation}, we use that $\partial_t P^j+\mathcal{L}^j P^j=0$, $j\in\{1,\dots,\ell\}$. Hence, for all $\xt\in\Rn\times (0,\infty)$,
\begin{align*}
	-\partial_{t}P(x,t)&=-\sum_{j=1}^\ell\partial_{t}P^j(x_j,t)\prod_{k=1,k\neq j}^\ell P^k(x_k,t)=\sum_{j=1}^\ell\Lk^jP^j(x_j,t)\prod_{k=1,k\neq j}^\ell P^k(x_k,t)=\Lk_{\nu}P(x,t).
\end{align*}

Once we know that $P$ is a classical solution to the equation we can show that it is also a very weak solution mimicking the proof of Lemma~\ref{PisaVeryWeaksolution.cor}.
\end{proof}

%%%%%%%%%%%%%%%%%%%%%%%%%%%%%%%%
\subsubsection{The function $\phi$}

The last thing that we have to check is the existence of a function $\phi\in C^2(\mathbb{R}^d)$ satisfying~\eqref{a.46}--\eqref{Lphi}. We have already proved that such function $\phi^j$ exists for each of the operators $\mathcal{L}^j$, $j\in\{1,\dots,\ell\}$ (and the corresponding bilinear form $B_j$) if the kernels $K_j$ satisfy \eqref{A0}--\eqref{A3}. Our function $\phi$ will then be given by
\begin{equation}\label{anPhi}
    \phi(x):=\prod_{j=1}^\ell\phi^j(x_j)\quad\textrm{for all }x=(x_1,\dots,x_\ell)\in\Rn.
\end{equation}

\begin{lemma}\label{lemmaGettingPhiAnisotropic}
    Let $\mathcal{L}=\mathcal{L}_\nu$ of the form~\eqref{a.Lnu} with $\nu$ given by~\eqref{an2} with functions $K_j$, $j=1,\dots,\ell$ satisfying \eqref{A0}--\eqref{A3}. Let $P$ be the associated heat kernel, given by~\eqref{anHK}, and $B=B_\nu$ the bilinear form defined in \eqref{eq:def.Bnu}. The function $\phi$ defined in~\eqref{anPhi} belongs to $C^2(\mathbb{R}^d)$ and satisfies~\eqref{a.46}--\eqref{Lphi}.
\end{lemma}

\begin{proof}
Since $P^j_1$ is comparable to $\phi^j$ for each $j\in\{1,\dots,\ell\}$,~\eqref{a.46} follows easily thanks to the structure in separate variables of both $\phi$ and $P$.

On the other hand, since $|\Lk^j\phi^j|\le c_j\phi^j$ for some $c_j$ for all $j\in\{1,\dots,\ell\}$, using again the structure in separate variables of $\phi$ and the special form~\eqref{an2} of the L\'evy measure we arrive at
\begin{equation*}
	|\Lk_{\nu}\phi(x)|=\Big|\sum_{j=1}^\ell\Lk^j\phi^j(x_j)\prod_{k=1,k\neq j}^\ell\phi^k(x_k)\Big|
    \leq\sum_{j=1}^\ell c_j\phi(x)=c\phi(x)\quad\textrm{for all }x\in\Rn.
\end{equation*}

As for the second inequality in~\eqref{Lphi}, we have $|B_{\nu}(\psi,\phi)(x)|\leq \|\nabla \psi\|_{\infty}\textup{I}(x)+2\|\psi\|_{\infty}\textup{II}(x)$, where
\begin{align*}
    \textup{I}(x)&:=\ints|y||\phi(x)-\phi(x-y)|\dnu(y),\\
    \textup{II}(x)&:=\underbrace{\phi(x)\intb\dnu(y)}_{\textup{II}_1(x)}+\underbrace{\intb\phi(x-y)\dnu(y)}_{\textup{II}_2(x)}.
\end{align*}
Because of the special form~\eqref{an2} of the L\'evy measure $\nu$ and the product structure~\eqref{anPhi} of $\phi$, since  $K_j$ and $\phi^j$ satisfy~\eqref{z26}--\eqref{z25} and the kernels $K^j$ satisfy the L\'evy condition~\eqref{A0}, then
\begin{align*}
    \textup{I}(x)&=\sum_{j=1}^\ell\int_{\{|y_j|\leq 1\}}|y_j||\phi^j(x_j)-\phi^j(x_j-y_j)|K_j(y_j)\dy_j\prod_{k=1,k\neq j}^\ell\phi^k(x_k)\\
    &\leq\sum_{j=1}^\ell C_j \int_{\{|y_j|\leq 1\}}|y_j|^2K_j(y_j)\dy_j\prod_{k=1}^\ell\phi^k(x_k)=C\phi(x).
\end{align*}
On the other hand, since $\nu$ is a L\'evy measure, then $\textup{II}_1(x)\le C\phi(x)$. Finally, using again the special form of $\nu$ and $\phi$, and then, for all $j\in\{1,\dots,\ell\}$, the comparability~\eqref{Atotal} of $K^j$ and $P^j_1$ outside the origin, the semigroup property~\eqref{eq:semigroup.property},~\eqref{Atotal2} (for $P^j$) and the comparability of $P_1^j$ and $\phi^j$,
\begin{align*}
    \textup{II}_2(x)&=\sum_{j=1}^\ell\int_{\{|y_j|\geq 1\}}\phi^j(x_j-y_j)K^j(y_j)\dy_j\prod_{k=1,k\neq j}^\ell\phi^k(x_k)\\
    &\leq\sum_{j=1}^\ell\big(\prod_{k=1,k\neq j}^\ell\phi^k(x_k)\big)\int_{\mathbb{R}^{d_j}}C_jP^j_1(x_j-y_j)P^j_1(y_j)\dy_j
    \leq\sum_{j=1}^\ell\big(\prod_{k=1,k\neq j}^\ell\phi^k(x_k)\big)C_jP^j_2(x_j)\\
    &\leq\sum_{j=1}^\ell\big(\prod_{k=1,k\neq j}^\ell\phi^{j}(x_{j})\big)C_jP^j_1(x_j)\le
    \sum_{j=1}^\ell\big(\prod_{k=1,k\neq j}^\ell\phi^{j}(x_{j})\big)C_j\phi^j(x_j)=C\phi(x).\qedhere
\end{align*}
\end{proof}

%%%%%%%%%%%%%%%%%%%%%%%%%%%%%%%%%%%%%%%%%%%%%%%%%%%%%%
\subsubsection{Existence in a special case}

As in the case of operators of the form $\mathcal{L}_K$, we will only be able to prove existence for a special case, namely
$\mathcal{L}_\nu=\sum_{j=1}^\ell(-\Delta_{d_j})^{\alpha/2}$, where $\alpha\in (0,2)$ and $(-\Delta_{d_j})^{\alpha/2}$ is the $d_j$-dimensional $\alpha$-fractional Laplacian.

We denote the heat kernels associated to~$(-\Delta_{d_j})^{\alpha/2}$, $j=1,\dots,\ell$, and~$\sum_{j=1}^\ell(-\Delta_{d_j})^{\alpha/2}$  respectively by $P^{j,\alpha}$, $j=1,\dots,\ell$ and $P^\alpha$. We recall that
\begin{equation}\label{anHkalpha}	
    P^\alpha(x,t)=\prod_{j=1}^\ell P^{j,\alpha}(x_j,t),\quad x=(x_1,\dots,x_\ell)\in\Rn,\ t>0.
\end{equation}
Then the heat kernel $P^\alpha$ inherits the self-similar structure from the kernels $P_j^\alpha$, $j=1,\dots,\ell$,
\begin{equation}\label{eq:Palpha.selfsimilar}
    P^{\alpha}(x,t)=\prod_{j=1}^\ell\big(t^{-d_j/\alpha}P^{j,\alpha}(x_jt^{-\alpha},1)\big)=t^{-d/\alpha}
    \prod_{j=1}^d P^{j,\alpha}(x_jt^{-\alpha},1)=t^{-d/\alpha}P^{\alpha}(xt^{-\alpha},1).
\end{equation}
This is the clue to prove the necessary conditions for the existence result.

\begin{proposition}
    For all $\varepsilon\in(0,T)$  there is a constant $C>0$ such that
    \begin{equation*}
        \textup{(i)}\ |\partial_{t}P^{\alpha}\xt|\leq C P^{\alpha}_1(x),\quad\textup{(ii)}\ |D^{2} P_{t}(x)|\leq CP_{1}(x),\qquad x\in\Rn, \ t\in(\varepsilon,T).
    \end{equation*}
\end{proposition}

\begin{proof}
(i) The selfsimilar structure~\eqref{eq:Palpha.selfsimilar} of $P^\alpha$ and its product structure~\eqref{anHkalpha} imply that
\begin{equation*}
    \partial_{t}P^{\alpha}(x,t)=-\frac{d}{\alpha t}t^{-d/\alpha}P^{\alpha}_1(xt^{-\alpha})-\frac{t^{-d/\alpha}}{\alpha t}\sum_{j=1}^\ell
    \langle x_jt^{-\alpha},\nabla_{x_j}P^{j,\alpha}_1(x_jt^{-\alpha})\rangle\prod_{k=1,k\neq j}^\ell P^{k,\alpha}_1(x_kt^{-\alpha}).
\end{equation*}
On the other hand, by Proposition \ref{pprop1} (iii), for each $j\in\{1,\dots,\ell\}$ there is a constant $C_j$ such that $|\langle x_jt^{-\alpha},\nabla P^{j,\alpha}_1(x_jt^{-\alpha})\rangle|\leq C_jP^{j,\alpha}_1(x_jt^{-\alpha})$, and the result follows using~\eqref{anHkalpha} and~\eqref{Atotal2}.

\noindent (ii) There are two possibilities. If the two variables, $x_i,x_l$, $i,l\in\{1,\dots,d\}$, with respect to which we differentiate fall within the same subspace $\mathbb{R}^{d_j}$, $j\in\{1,\dots,\ell\}$, then, thanks to~\eqref{anHkalpha},
\begin{equation*}
    \partial_i\partial_l P_t^\alpha(x)=\partial_i\partial_l P_t^{j,\alpha}(x_j)\prod_{k=1,k\neq j}^\ell P_t^{k,\alpha}(x_k),
\end{equation*}
and the result follows using Proposition~\ref{pprop1} (vi). If, instead, each of them falls in a different subspace, say $x_i$ in $\mathbb{R}^{d_{j_1}}$ and $x_l$ in $\mathbb{R}^{d_{j_2}}$, $j_1,j_2\in\{1,\dots,\ell\}$, $j_1\neq j_2$, then,
\begin{equation*}
    \partial_i\partial_l P_t^\alpha(x)=\partial_i P_t^{j_1,\alpha}(x_{j_1})\partial_l P_t^{j_2,\alpha}(x_{j_2})\prod_{k=1,k\neq j_1,j_2}^\ell P_t^{k,\alpha}(x_k),
\end{equation*}
and now the result follows using Proposition~\ref{pprop1} (v).
\end{proof}

%%%%%%%%%%%%%%%%%%%%%%%%%%%%%%%%%%%%%%%%%%%%%%%%%%%%%%%%%%%%%%%%%%%%%%%%%%%%%%%%%%%%%%%%%
%%%%%%%%%%%%%%%%%%%%%%%%%%%%%%%%%%%%%%%%%%%%%%%%%%%%%%%%%%%%%%%%%%%%%%%%%%%%%%%%%%%%%%%%%
\section{Operators having a local and a nonlocal part}\label{section5}
\setcounter{equation}{0}

We consider now operators {that combine both a local and a nonlocal part, more precisely those of the form } $\mathcal{L}=\mathcal{L}_{\Delta,K}:=-\Delta+\mathcal{L}_K$, with $\mathcal{L}_K$ as in~\eqref{a.Lk} and $K$ satisfying~\eqref{A0}--\eqref{A2} (and~\eqref{A3} in some cases).

%%%%%%%%%%%%%%%%%%%%%%%%%%%%%%%%%%%%%%%%%%%%%%%%%%%%%%%
\subsection{The heat kernel}

When $\mathcal{L}=\mathcal{L}_{\Delta,K}$, with $K$ satisfying~\eqref{A0}--\eqref{A2}, the L\'evy process associated to equation~\eqref{a.nonlocalheatequation} is the sum of two independent L\'evy processes, each one with its heat kernel: the L\'evy process associated to $\partial_t u+\mathcal{L}_K u=0$,  with heat kernel~$P^{K}$, and the Brownian process associated to the classical heat equation, with the well-known heat kernel $G$, given in~\eqref{p12}.

Let $m_K$ be the L\'evy-Khintchine exponent of $\mathcal{L}_K$, given by~\eqref{p6}. Since the Fourier symbol of $-\Delta$ is $|\xi|^2$, the heat kernel $P$ associated to $\mathcal{L}_{\Delta,K}$ is then given in terms of Fourier transform by
\begin{equation*}
	\mathcal{F}(P_{t})(\xi)=e^{-t(m_K(\xi)+|\xi|^{2})},\quad\xi\in\Rn,\ t>0.
\end{equation*}
Therefore, using the properties of the Fourier transform,
\begin{equation}\label{51.9}
    P_t(x)=P^{K}_t*G_t(x),\quad x\in\Rn,\ t>0.
\end{equation}

Since $G_t$ decays as the inverse of a quadratic exponential and $\tilde P_t$ as the inverse of some polynomial (see~\eqref{Atotal},~\eqref{A1} and~\eqref{A2}), there is a constant $C_t>0$ such that
\begin{equation}\label{eq:slowest.decay}
    G_t(x)\le C_t P^K_t(x)\quad\text{for all }x\in\mathbb{R}^d.
\end{equation}
From this and~\eqref{51.9} we deduce now that $P_{t}$ is comparable with $P^{K}_{t}$, a key fact for what follows.

\begin{proposition}\label{51.prop1}
    Let $\mathcal{L}=\mathcal{L}_{\Delta,K}$ with $K$ satisfying~\eqref{A0}--\eqref{A2} and let $P$ be corresponding heat kernel. For each $t>0$ there is a constant $C_t>1$ such that $C_t^{-1}P^{K}_t\leq P_t\leq C_tP^{K}_t$.
\end{proposition}

The result can be seen in \cite{Zoran2}. Here we provide a, somehow, shorter proof.

\begin{proof}
As for the first inequality, since $P^{K}_t, G_t>0$, then $P_t>0$. Therefore, there is a constant $C_t>0$ such that $P^{K}_t(x)\leq C_tP_t(x)$ for all $|x|\leq 2$. Consider now $|x|\geq 2$. If $|y|<1$, then $|x-y|<|x|+1\le 3|x|/2$. Hence, since $P^{K}_t$ is almost decreasing and {slowly changing}, see Remark~\ref{rk:Pt.sv.and.ad},
\begin{equation*}
    P_{t}(x)\geq \int_{\{|y|< 1\}} \tilde P_t(x-y)G_{t}(y)\dy\ge C_tP^{K}_{t}(3x/2)\int_{\{|y|<1\}}G_{t}(y)\dy\ge C_tP^{K}_t(x).
\end{equation*}

To prove the second inequality, we split $P_t$ in two integrals,
\begin{equation*}
    P_{t}(x)=\int_{\{|x-y|\leq\frac{1}{2}|x|\}}P^{K}_{t}(x-y)G_{t}(y)\dy+\int_{\{|x-y|\geq\frac{1}{2}|x|\}}P^{K}_{t}(x-y)G_{t}(y)\dy.
\end{equation*}
If $|x-y|\leq|x|/2$, then $|x|/2\leq |y|$. Since $G_{t}$ is radially decreasing and $\|P^{K}_t\|_{L^1(\mathbb{R}^d)}=1$, using also~\eqref{eq:slowest.decay} and the {slow change} of $P^{K}_t$,
\begin{equation*}
    \int_{\{|x-y|\leq\frac{1}{2}|x|\}}P^{K}_{t}(x-y)G_{t}(y)\dy\leq G_t(x/2)\intr P^{K}_{t}(x-y)\dy= G_t(x/2)\leq C_tP^{K}_{t}(x/2)\leq C_tP^{K}_{t}(x).
\end{equation*}
Besides, since $P^{K}_{t}$ is almost decreasing and {slowly changing}, and $\|G_t\|_{L^1(\mathbb{R}^d)}=1$, we have that
\begin{equation*}
	\int_{\{|x-y|\geq\frac{|x|}{2}\}}P^{K}_{t}(x-y)G_{t}(y)\dy\leq P^{K}_t(x/2)\intr G_{t}(y)\dy\leq C P^{K}_t(x).\qedhere
\end{equation*}
\end{proof}

As a corollary of Proposition \ref{51.prop1}, since $P^{K}$ is comparable in two different times, see \eqref{Atotal2}, so it is~$P$.

\begin{corollary}\label{51.cor1}
    Let $\mathcal{L}=\mathcal{L}_{\Delta,K}$ with $K$ satisfying~\eqref{A0}--\eqref{A2} and let $P$ be corresponding heat kernel. For all $t_{1}, t_{2}>0$ there exists $C>1$ such that
	\begin{equation*}\label{5.Atotal2}
        C^{-1}P_{t_{1}}\leq P_{t_{2}}\leq CP_{t_{1}}.
	\end{equation*}
\end{corollary}

Moreover, since $P^{K}_t$ is {slowly changing} and almost decreasing for all $t>0$ and $P^{K}_{1}$ and $P_{1}$ are comparable, the heat kernel associated to $\mathcal{L}_{\Delta,K}$ has also these properties.
\begin{corollary}\label{51.cor2}
    Let $\mathcal{L}=\mathcal{L}_{\Delta,K}$ with $K$ satisfying~\eqref{A0}--\eqref{A2} and let $P$ be corresponding heat kernel. Then, $P_t$ is {slowly changing} and almost decreasing for all $t>0$ .
\end{corollary}

%%%%%%%%%%%%%%%%%%%%%%%%%%%%%%%%%%%%%%%%%%%%%%%%%%%%%%%
\subsection{Very weak solutions}

Thanks to the comparability between $P^K_t$ and $P_t$, we can develop a complete Widder-type theory for very weak solutions for the operator $\mathcal{L}_{\Delta,K}$.

%%%%%%%%%%%%%%%%%%%%%%%%%%%%%%%%%%%%%%%%%%%%%%%%
\subsubsection{Alternative definitions of very weak solution}

We can adapt easily the proof of Proposition~\ref{prop1} to check that integrability against $\mathcal{L}_{\Delta,K}\theta$ with $\theta$ a test function is equivalent to integrability against $P_1$. The only delicate point is to check that if $\psi\in C^{\infty}_{\textup{c}}(\Rn)$ is such that $\mathcal{X}_{B_{1}(0)}\leq \psi\leq \mathcal{X}_{B_{2}(0)}$, then there is a constant $C>0$ such that $|\mathcal{L}_{\Delta,K}\psi(x)|\geq CP_{1}(x)$ for all $|x|\geq 4$. But this follows easily from~\eqref{v3.24},  since $-\Delta$ is a local operator and $\textrm{supp}(\psi)\subset B_{2}(0)$. Thus we have the following result.

\begin{proposition}\label{51.propequivalenceveryweak}
    Let $\mathcal{L}=\mathcal{L}_{\Delta,K}$, with $K$ satisfying~\eqref{A0}--\eqref{A2}.

    \noindent{\rm (i)} Let  $0\leq\delta<\tau\leq T$. Then $u\in L^{1}((\delta,\tau);\lp)$ if and only if $u\in L^{1}_{\rm{loc}}(\Rn\times[\delta,\tau])$ and \eqref{v3.9} holds.

    \noindent{\rm (ii)}  $u\in\luloc$ if and only if  $u\in L^{1}_{\rm{loc}}(\Rn\times(0,T))$ and~\eqref{a.2} holds.
\end{proposition}

It is now easy to mimic the proof of Proposition~\ref{equivDefvwProp} (that essentially only has to deal with the way the initial datum is taken) to show the equivalence  in the class $\li$ between the two definitions of very weak solution that we are handling.

\begin{proposition}\label{mixequivDefvwProp}
    Let $\mathcal{L}=\mathcal{L}_{\Delta,K}$ with $K$ satisfying~\eqref{A0}--\eqref{A2}. Let $\mu_{0}$ be a Radon measure and $u\in\li$.
    	
    \noindent{\rm (i)} If $u$ satisfies~\eqref{v3.10}, then it is a very weak solution to~\eqref{a.nonlocalheatequation} with initial trace~$\mu_0$.
    	
    \noindent{\rm(ii)} If $u$ is a very weak solution to~\eqref{a.nonlocalheatequation} with initial trace $\mu_{0}$, then it  satisfies~\eqref{eq:space.alternative.definition}--\eqref{v3.10}.
\end{proposition}

%%%%%%%%%%%%%%%%%%%%%%%%%%%%%%%%%%%%%%%%%%%%%%%
\subsubsection{Uniqueness for very weak solutions}

The proof of uniqueness for very weak solutions is similar to that of Theorem~\ref{UniquenessSoria}, taking care of an extra term in the bilinear form associated to the operator.

\begin{theorem}\label{mixUniquenessSoria}
    Let $\mathcal{L}=\mathcal{L}_{\Delta,K}$ with $K$ satisfying~\eqref{A0}--\eqref{A2}. Let $\mu_{0}$ be a Radon measure. Let $u_1$ and $u_2$ satisfy~\eqref{eq:space.alternative.definition}--\eqref{v3.10}. Then, $u_1=u_2$ a.e.\,in $\mathbb{R}^d\times(0,T)$.
\end{theorem}

\begin{proof}
The main difference with respect to the proof of Theorem \ref{UniquenessSoria} is that the bilinear form has an extra term with respect to the one appearing there, so that now it is given by
\begin{equation*}\label{eq:bilinear.local.nonlocal}
    B_{\Delta,K}(u,v)=B(u,v)+2\langle \nabla u,\nabla v\rangle,
\end{equation*}
with $B$ as in~\eqref{51.2}. Hence, using $\psi_{R}\varphi\in C_{\textup{c}}^\infty(\Rn\times[0,T))$ as test function in~\eqref{v3.15}, where $\{\psi_{R}\}_{R\geq R_{0}}$ is the family of cutoff functions given by~\eqref{cutoffspace} and $\varphi$ is the solution to the backward problem $\partial_t\varphi-\mathcal{L}_{\Delta,K}\varphi=\tilde\theta$ in $\mathbb{R}^d\times(0,T)$ given by~\eqref{eq:solution.backward.problem}, we arrive at
\begin{align*}
	0&=\int_{0}^{T}\intr u(x,t)\tilde{\theta}(x,t)\dx\dt-\int_{0}^{T}\intr u\xt\varphi\xt\mathcal{L}_{\Delta,K}\psi_{R}\xt\dx\dt\\
	\notag&\quad+\int_{0}^{T}\intr u\xt B(\phi_{R},\varphi)\xt\dt+2\int_{0}^{T}\intr u\xt \langle\nabla\phi_{R},\nabla\varphi\rangle\xt\dt.
\end{align*}
If we check that the three last terms on the right-hand side go to 0 as $R\to\infty$ the proof will be complete. The two first of them are handled essentially as in the proof of Theorem~\ref{UniquenessSoria}, using that $P_{1}$ is comparable to the kernel $K$ far from the origin, since $P_{1}$ is comparable to $P^{K}_{1}$, and observing also that $\|\Delta\psi_{R}\|_{\infty}=R^{-2}\|\Delta\psi\|_{\infty}$.

As for the last term, it is also handled easily, since $|\nabla \varphi|\leq CP_{1}$ (see once more the proof of Theorem~\ref{UniquenessSoria}), $\|\nabla\psi_{R}\|_{\infty}=R^{-1}\|\nabla \psi\|_{\infty}$ and  $u\in L^{1}((0,\too);L^{1}_{P_{1}})$. Therefore, we may apply the DCT to arrive at
\begin{equation*}\label{z28}	
    \lim\limits_{R\to \infty}\int_{0}^{T}\intr |u(x,t)||\langle \nabla\phi_{R}(x),\nabla\varphi\xt\rangle|\dx\dt=0.\qedhere
\end{equation*}
\end{proof}
As a consequence, since by Proposition~\ref{mixequivDefvwProp} a very weak solution $u\in\li$ with initial datum a Radon measure $\mu_{0}$ satisfies~\eqref{v3.10}, we obtain the following corollary.

\begin{corollary}[Uniqueness for  very weak solutions within $\li$]\label{mixUniquenessAcotadas}
    Let $\mathcal{L}=\mathcal{L}_{\Delta,K}$ with $K$ satisfying~\eqref{A0}--\eqref{A2}. There is  at most one very weak solution to equation \eqref{a.nonlocalheatequation} belonging to $\li$ with a given initial trace.
\end{corollary}

Now, we move on to the class of nonnegative very weak solutions, following the strategy of Section~\ref{section3}.

Let $\phi\in C^{2}(\Rn)$ be the function given by Lemma~\ref{lemmaGettingPhi}. By Proposition~\ref{51.prop1}, $P_{1}$ and $P^{K}_{1}$ are comparable, and by~\eqref{a.46}, $\tilde P_{1}$ and $\phi$ are comparable. Hence,~\eqref{a.46} holds: $P_1$ and $\phi$ are comparable. Besides, by~\eqref{v3.26} and the positivity and continuity of $\phi>0$, there is some $C>0$ such that
\begin{equation*}
    |\Delta\phi| \le C\phi,\qquad |\langle\nabla\psi,\nabla\phi\rangle|\le \|\nabla\psi\|_\infty |\nabla\phi|\le C\|\nabla\psi\|_\infty\phi \quad \text{for all }\psi\in C^{\infty}_{\textup{c}}(\Rn).
\end{equation*}
Combining this with~\eqref{Lphi} (applied to $\mathcal{L}_K$), we arrive at~\eqref{Lphi} with $\mathcal{L}=\mathcal{L}_{\Delta,K}$ and $B=B_{\Delta,K}$. Therefore, arguing as in the proof of Lemma~\ref{lemmatrace}, using also Proposition~\ref{51.propequivalenceveryweak}, we have the desired smoothing effect.

\begin{lemma}\label{lemmatrace.nu}
    Let $\mathcal{L}=\mathcal{L}_{\Delta,K}$ with $K$ satisfying~\eqref{A0}--\eqref{A3}.

    \noindent\textup{(i)} If $u$ is a nonnegative very weak solution to~\eqref{a.nonlocalheatequation}, there is $c>0$ such that~\eqref{lematecnicodesigualdad} holds.

    \noindent\textup{(ii)} If $u$ is a nonnegative very weak solution to~\eqref{a.nonlocalheatequation}, then $u\in\li$.
\end{lemma}

Imitating the proof of Proposition~\ref{prop:regularity.in.time} we get some regularity in time for general operators.

\begin{proposition}
    Let $\mathcal{L}=\mathcal{L}_{\Delta,K}$ with $K$ satisfying~\eqref{A0}--\eqref{A3}, and let $u$ be a nonnegative very weak solution to~\eqref{a.nonlocalheatequation}.

    \noindent\textup{(i)} Let $\phi\in C^2(\mathbb{R}^d)$ satisfying~\eqref{a.46}--\eqref{Lphi}. There is a constant $C>0$ such that~\eqref{eq:continuity.phi} holds.

    \noindent\textup{(ii)} Given $\psi\in C^2_{\textup{c}}(\mathbb{R}^d)$, there is a constant $C_\psi>0$ such that~\eqref{eq:continuity.psi} holds.
\end{proposition}

The desired uniqueness result for nonnegative very weak solutions now follows from the combination of Lemma~\ref{lemmatrace.nu} (ii) and Corollary~\ref{mixUniquenessAcotadas}.

\begin{corollary}[Uniqueness for nonnegative very weak solutions]\label{mixUniquenessNonnegative}
    Let $\mathcal{L}=\mathcal{L}_{\Delta,K}$ with $K$ satisfying~\eqref{A0}--\eqref{A3}. There is at most one nonnegative very weak solution of~\eqref{a.nonlocalheatequation} with a given initial trace.
\end{corollary}

%%%%%%%%%%%%%%%%%%%%%%%%%%%%%%%%%%%%%%%%%%%%%%%%%
\subsubsection{Existence of initial trace}

The proof of Theorem~\ref{theoremtraces} giving the existence of an initial trace satisfying the integral growth condition~\eqref{eq:growth.initial.trace} for nonnegative very weak solutions is based on the estimates~\eqref{lematecnicodesigualdad}, and may therefore be easily adapted to cope with operators~$\mathcal{L}=\mathcal{L}_{\Delta,K}$.

\begin{theorem}[Existence of  initial trace integrable against $P_{1}$] \label{51.theoremtraces}
    Let $\mathcal{L}=\mathcal{L}_{\Delta,K}$ with $K$ satisfying~\eqref{A0}--\eqref{A3}. Let $u$ be a nonnegative very weak solution of~\eqref{a.nonlocalheatequation} in $\Rn\times(0,T)$. There is a nonnegative Radon measure $\mu_{0}$ such that~\eqref{a.1} holds; that is, $\mu_0$ is the initial trace of $u$.  Moreover, $\mu_0$ satisfies the integral growth condition~\eqref{eq:growth.initial.trace}.
\end{theorem}

%%%%%%%%%%%%%%%%%%%%%%%%%%%%%%%%%%%%%%%%%%%%%%%%
\subsubsection{Existence of very weak solutions}

The existence of very weak solutions, for suitable initial data $\mu_0$, is obtained by imitating the proof of Theorem~\ref{prop2}. They will be given by $U(\cdot,t)=P_t*\mu_0$. The news with respect to that theorem are that now $U\in C^{\infty}(\Rn\times (0,\infty))$. Indeed, because of~\eqref{51.9}, $U(\cdot,t)=G_{t}*(P^K_t*\mu_0)$ for all $t>0$, whence, since $(P^K_t*\mu_0)$ is continuous, the regularity of $U$ coincides with that of $G$.

\begin{theorem}\label{mixprop2}
    Let $\mathcal{L}=\mathcal{L}_{\Delta,K}$ with $K$ satisfying~\eqref{A0}--\eqref{A2}. Let $\mu_{0}$ be a Radon measure with the admissible growth at infinity~\eqref{q4}. Then, for each $T>0$, the function $U$ given in~\eqref{q5} is a very weak solution of \eqref{a.nonlocalheatequation} in $C^{\infty}(\Rn\times(0,T))\cap \li$ \normalcolor with initial trace~$\mu_{0}$.
\end{theorem}

As a consequence, thanks to the uniqueness results, Corollary~\ref{mixUniquenessAcotadas} and Corollary~\ref{mixUniquenessNonnegative}, we have a representation formula.

\begin{corollary}[Representation formula]\label{mixCorRepresentationFormulaVW}
    Let $\mathcal{L}=\mathcal{L}_{\Delta,K}$ with $K$ satisfying~\eqref{A0}--\eqref{A2} and let $\mu_{0}$ be a Radon measure satisfying~\eqref{q4}. Let $u$ be a very weak solution to~\eqref{a.nonlocalheatequation} with trace $\mu_{0}$. If either $u\in \li$, or $u$ is nonnegative and $K$ satisfies in addition~\eqref{A3}, then $u$ is given by representation formula~\eqref{representationformula}.
\end{corollary}

%%%%%%%%%%%%%%%%%%%%%%%%%%%%%%%%%%%%%%%%%%%%%%%%%%%%%%%
\subsection{Classical solutions}

The main novelty in the analysis of classical solutions is that for operators $\mathcal{L}_{\Delta,K}$ we don't need to impose the energy assumption~\eqref{energyH}, since in this case classical solutions belong by definition to $C^{2}(\Rn\times(0,T))$ and, therefore, are already known to satisfy it.

%%%%%%%%%%%%%%%%%%%%%%%%%%%%%%%%%%%%%%%%%%%%%%
\subsubsection{Uniqueness of classical solutions.}

We can now imitate the proof of Proposition~\ref{cprop1} to prove that classical solutions are very weak solutions, the key step towards uniqueness in $\li$.

\begin{proposition}\label{52prop3}
    Let $\mathcal{L}=\mathcal{L}_{\Delta,K}$ with $K$ satisfying~\eqref{A0}--\eqref{A2}. If $u\in\luloc$ is a classical solution to~\eqref{a.nonlocalheatequation}, then it is a very weak solution to the same equation.
\end{proposition}

As a byproduct of Proposition~\ref{52prop3}, Corollary~\ref{mixUniquenessAcotadas} and Corollary~\ref{mixCorRepresentationFormulaVW}, we obtain uniqueness for classical solutions in $\li$, and a representation formula if there is existence in this class.

\begin{corollary}\label{mixCorRepresentativeFormulaClassicalAcotadas}
    Let $\mathcal{L}=\mathcal{L}_{\Delta,K}$ with $K$ satisfying~\eqref{A0}--\eqref{A2}.

    \noindent{\rm (i)} There is at most one classical solution $u\in\li$ to~\eqref{a.nonlocalheatequation} with initial trace a given Radon measure.

    \noindent{\rm (ii)} If $K$ satisfies also~\eqref{A3}, if there exists a classical solution $u\in\li$ to~\eqref{a.nonlocalheatequation} with initial trace a measure $\mu_0$ verifying~\eqref{q4}, then $u$ is given by the representation formula~\eqref{representationformula}.
\end{corollary}

To prove uniqueness of nonnegative classical solutions, we first obtain a comparison principle. This can be done following the proof of Lemma~\ref{a.mp}, observing that at a point $(x_0,t_0)$ where $v$ achieves a maximum we have $\Delta v(x_{0},t_{0})\leq 0$ and $\Lk_K v(x_{0},t_{0})> 0$, so that $\mathcal{L}_{\Delta,K}v(x_0,t_0)>0$.

\begin{lemma}[A maximum principle] \label{52mp3}
    Let $\mathcal{L}=\mathcal{L}_{\Delta,K}$ with $K$ satisfying~\eqref{A0}. Let $\Omega\subset\RN$ be non-empty, open, bounded and smooth, and $v\in C(\Rn\times[0,T))$ such that $v\in C^{2,1}_{x,t}(\Omega\times(0,T))$ and $\Lk_K v\xt$ is defined for all $\xt\in\Omega\times (0,T)$. If~\eqref{eq:ineq.MP} holds, then $v\leq 0$ for all $x\in\mathbb{R}^d\times[0,T)$.
\end{lemma}

We can now show the validity of~\eqref{a.16}, following the proof of~Lemma~\ref{a.cor4}.

\begin{lemma}\label{52a.cor4}
    Let $\mathcal{L}=\mathcal{L}_{\Delta,K}$  with $K$ satisfying~\eqref{A0}--\eqref{A2}, and $u$ a nonnegative classical solution to~\eqref{a.nonlocalheatequation} with initial data a nonnegative Radon measure $\mu_0$. Then,~\eqref{a.16} holds.
\end{lemma}

As a consequence, proceeding as in the proof of Corollary~\ref{a.cor2}, we obtain that~$u\in L^{1}_{\rm loc}([0,T);\lp)$.

\begin{corollary}\label{a.cor2.Delta.K}
    Let $\mathcal{L}=\mathcal{L}_{\Delta,K}$ with $K$ satisfying~\eqref{A0}--\eqref{A2}, and $u$ a nonnegative classical solution to~\eqref{a.nonlocalheatequation} with initial data a nonnegative Radon measure $\mu_0$. Then, $u\in L^{1}_{\rm loc}([0,T);\lp)$.
\end{corollary}

We have now all the ingredients to get a representation formula (implying uniqueness) for nonnegative classical solutions, mimicking the proof of Theorem~\ref{thm:main.K}.

\begin{theorem}\label{mixa.maintheorem1}
    Let $\mathcal{L}=\mathcal{L}_{\Delta,K}$  with $K$ satisfying~\eqref{A0}--\eqref{A3}. Let $u$ be a nonnegative classical solution to~\eqref{a.nonlocalheatequation} with initial trace a nonnegative Radon measure $\mu_{0}$. Then $u$ is given by the representation formula~\eqref{representationformula}.
\end{theorem}

%%%%%%%%%%%%%%%%%%%%%%%%%%%%%%%%%%%%%%%%%%%%%%
\subsubsection{Existence of classical solutions}

Our goal now is to prove that, if $\mathcal{L}=\mathcal{L}_{\Delta,K}$, then for any Radon measure $\mu_0$ satisfying the integral growth condition~\eqref{q4} there is a classical solution to~\eqref{a.nonlocalheatequation} having $\mu_0$ as initial trace. The natural candidate is the function $U(x,t)=P_t*\mu_0$, which is a very weak solution of the equation and has $\mu_0$ as initial trace, and moreover belongs to the class $\li$ in which there is uniqueness. In the case of operators of the form $\mathcal{L}_K$ we had to impose certain conditions on the heat kernel, namely~\eqref{eq:conditions.existence.classical}, in order for such $U$ to be a classical solution. When $\mathcal{L}=\mathcal{L}_{\Delta,K}$, the local part of the diffusion, the Laplacian, helps us, and these conditions are always satisfied, as we see next. Let us remark that the energy condition~\eqref{energyH} is not required now for uniqueness; nevertheless, $P$ also satisfies~\eqref{eq:stronger.conditions.existence.classical}. Besides, the local part of the operator does not require any extra hypothesis.

\begin{lemma}\label{52.prop1}
    Let $\mathcal{L}=\mathcal{L}_{\Delta,K}$  with $K$ satisfying~\eqref{A0}--\eqref{A2}, and let $P$ be the corresponding heat kernel.
    For each $T>0$ and $\varepsilon\in(0,T)$ there is a constant $C>0$ such that~\eqref{eq:conditions.existence.classical}--\eqref{eq:stronger.conditions.existence.classical} hold.
\end{lemma}
\begin{proof}
We start  by proving that for all $0<\varepsilon<T$ there is $C_{\varepsilon,T}>0$ such that $\left|\partial_{t} P_t(x)\right|\leq C_{\varepsilon,T}P_{1}$ for all $t\in(\varepsilon,T)$. On the one hand, a direct computation shows that
\begin{equation*}
	\begin{aligned}
        |\partial_{i}\partial_{j}G_{t}(x)|&\leq\Big(\frac{1}{4t}+\frac{|x|^{2}}{(4t)^{2}}\Big)G_{t}(x)\leq \Big(\frac{1}{4\varepsilon}+\frac{|x|^{2}}{(4\varepsilon)^{2}}\Big)\frac{1}{(4\pi\varepsilon)^{d/2}}e^{-|x|^{2}/(4T)}\\
		&\leq C_\varepsilon(1+|x|^{2})e^{ -|x|^{2}/(4T)}\leq C_{\varepsilon,T}G_{T}(x)\quad\textrm{for all }x\in\Rn,\ t\in(\varepsilon,T).
	\end{aligned}
\end{equation*}
Moreover, thanks to~\eqref{eq:slowest.decay} and the fact that $P^{K}$ satisfies~\eqref{Atotal2}, we have $G_T\leq C_T P^{K}_1$, whence
\begin{equation}\label{eqLemma5.20.DerivadasG}
    \max_{t\in[\varepsilon,T]}	\max\limits_{i,j\in\{1,\dots,d\}}	|	\partial_{i}\partial_{j}G_{t}(x)|\leq C_{\varepsilon,T}P^{K}_1(x)\quad\textrm{for all }x\in\mathbb{R}^d.
\end{equation}
	
On the other hand, from the properties of $P^{K}$ and the MVT, we obtain that
\begin{align*}
    \left|\mathcal{L}_K G_t(x)\right|&\leq G_t(x)\intb K(y)\dy+\intb G_t(x-y)K(y)\dy+\ints|\Lambda G_t(x,y)|K(y)\dy\\
    &\leq 	C_{\varepsilon,T}\Big(P^{K}_1(x)+\intb P^{K}_{1}(x-y)P^{K}_{1}(y)\dy\Big)+\ints|\Lambda G_t(x,y)|K(y)\dy\\
    &\leq C_{\varepsilon,T}\Big( P^{K}_1(x)+P^{K}_2(x)\Big)+\ints|D^{2}G_t(z)||y|^2K(y)\dy,
\end{align*}
for some $z=z(x,y)\in B_{1}(x)$. Now, using~\eqref{eqLemma5.20.DerivadasG} we conclude that
\begin{equation}\label{800}	
    \left|\mathcal{L}_K G_t(x)\right|\leq C_{\varepsilon,T} \Big(P^{K}_1(x)+\ints P^{K}_1(z)|y|^2K(y)\dy\Big)\leq C_{\varepsilon,T}P^{K}_1(x),
\end{equation}
where, for the last inequality, we have used that $P^K_1(z)\le C P^K_1(x)$, due to Lemma~\ref{v3.19}, and the fact that $|x-z|\le1$. As a consequence of~\eqref{800}, since $P_t=P^K_t*G_t$ is the fundamental solution to~\eqref{a.nonlocalheatequation} with $\mathcal{L}=\mathcal{L}_{\Delta,K}$, using~\eqref{eqLemma5.20.DerivadasG}, the semigroup property of $P^{K}$ and Proposition~\ref{51.prop1}, we have
\begin{align*}
    \left|\partial_{t} P_t(x)\right|&=\left|\mathcal{L}_{\Delta,K} P_t(x)\right|=\left|P^{K}_t*\mathcal{L}_{\Delta,K} G_t(x)\right|
    \leq P^{K}_t*\left|\mathcal{L}_K G_t\right|(x)+P^K_t*\left|\Delta G_t\right|(x)\\&\leq C_{\varepsilon,T}P^{K}_t*P^{K}_1(x)\leq C_{\varepsilon,T}P^{K}_{1+t}(x)\leq C_{\varepsilon,T}P_{1}(x)\quad\textrm{for all }x\in\Rn,
\end{align*}
which is the first inequality in~\eqref{eq:conditions.existence.classical}.

Second, by~\eqref{eqLemma5.20.DerivadasG}, the semigroup property of $P^{K}_t$  and Proposition~\ref{51.prop1}, we get
\begin{equation}\label{eq.DerivadasPt}	
    |D^{2}P_{t}|=\max\limits_{i,j\in\{1,\dots,d\}}	|	\partial_{i}\partial_{j}P_{t}|\leq P^{K}_{t}*\max\limits_{i,j\in\{1,\dots,d\}}	|\partial_{i}\partial_{j}G_{t}|\leq C_{\varepsilon,T}P^{K}_{t}*P^{K}_{1}=C_{\varepsilon,T}P^{K}_{1+t}\leq C_{\varepsilon,T}P_{1}
\end{equation}
for all $t\in(\varepsilon,T)$, which is  estimate~\eqref{eq:stronger.conditions.existence.classical}.

Finally, we will obtain the second inequality in~\eqref{eq:conditions.existence.classical}. Procceding as before, using  the  MVT and~\eqref{eq.DerivadasPt}, we get
\begin{align*}
	\intr |\Lambda P_{t}(x,t)|K(y)\dy&\leq\intb \big(P_{t}(x)+P_{t}(x-y)\big)K(y)\dy+\ints |D^{2}P_{t}(z)|y|^{2}K(y)\dy\\
	&\leq CP_t(x)+\intb P_{t}(x-y)K(y)\dy+C_{\varepsilon,T}\ints P_1(z)|y|^{2}K(y)\dy,
\end{align*}
where $z=z(x,y,t)\in B_{1}(x)$. We estimate the second term as follows
	
\begin{equation*}
    \intb  P_{t}(x-y)K(y)\dy\leq C\intb  P_{t}(x-y)P_1(y)\dy\leq CP_{1+t}(x)\leq C_{T}P_1(x),
\end{equation*}
whereas for the third term we use that 	$P_1$ is {slowly changing} (see Corollary~\ref{51.cor2}), that $z\in B_1(x)$ and that $K$ is a L\'evy kernel, to obtain
\begin{equation*}
	\ints P_1(z)|y|^{2}K(y)\dy\leq CP_1(x)\ints |y|^2 K(y)\dy=CP_1(x).
\end{equation*}
Putting all together we have that
\begin{equation*}
	\intr |\Lambda P_{t}(x,t)|K(y)\dy\leq CP_1(x)\quad\textrm{for all }x\in\Rn.\qedhere
\end{equation*}
\end{proof}

We can now reproduce the proof of Theorem~\ref{Uclassical} to arrive at our last result.
\begin{theorem}\label{mixUclassical}
    Let $\mathcal{L}=\mathcal{L}_{\Delta,K}$ with $K$ satisfying~\eqref{A0}--\eqref{A2}, $\mu_{0}$ a Radon measure satisfying the integral growth condition~\eqref{q4}, and $U$ the very weak solution to~\eqref{a.nonlocalheatequation} given in~\eqref{q5}. Then, for each $T>0$, $U$ is the unique classical solution to~\eqref{a.nonlocalheatequation} within  $\li$  with initial datum $\mu_{0}$. If in addition $K$ satisfies~\eqref{A3} and $\mu_0$ is nonnegative, then $U$ is the unique nonnegative classical solution to~\eqref{a.nonlocalheatequation} with initial trace $\mu_0$.
\end{theorem}

%%%%%%%%%%%%%%%%%%%%%%%%%%%%%%%%%%%%%%%%%%%%%%%%%%%%%%%%%%%%%
% ACKNOWLEDGMENTS
%%%%%%%%%%%%%%%%%%%%%%%%%%%%%%%%%%%%%%%%%%%%%%%%%%%%%%%%%%%%%
\section*{Acknowledgments}

\noindent I. Gonz\'alvez and F. Quir\'os were supported by grants PID2023-146931NB-I00, RED2022-134784-T and CEX2023-001347-S, and I. Gonz\'alvez also by grant PID2019-110712GB-I00, all of them funded by MICIU/AEI/10.13039/501100011033. Both of them were also supported by the Madrid Government
(Comunidad de Madrid – Spain) under the multiannual Agreement with UAM in the line for the Excellence of the University Research Staff in the
context of the V PRICIT (Regional Programme of Research and Technological Innovation).

\noindent F. Soria was supported by grants PID2019-110712GB-I00 and PID2022-142202NB-I00, both of them funded by MICIU/AEI/10.13039/501100011033.

\noindent Z. Vondra\v{c}ek was supported in part by the Croatian Science Foundation under the project IP-2022-10-2277.

%%%%%%%%%%%%%%%%%%%%%%%%%%%%%%%%%%%%%%%%%%%%%%%%%%%%%%%%%%%%%
% REFERENCES
%%%%%%%%%%%%%%%%%%%%%%%%%%%%%%%%%%%%%%%%%%%%%%%%%%%%%%%%%%%%%

\end{document}